\documentclass[11pt]{amsart}

\usepackage{amsmath}
\usepackage{amsfonts}
\usepackage{amssymb}
\usepackage{amsthm}
\usepackage{mathtools}
\usepackage{url}
\usepackage{comment}
\usepackage{bbold}
\usepackage{appendix}
\usepackage{upgreek}

\usepackage{extsizes}
\usepackage[margin=3cm]{geometry}

\usepackage{enumitem}
\setenumerate[0]{label=(\roman*)}

\usepackage{hyperref}
 \hypersetup{
     colorlinks=true,
     linkcolor=red,
     filecolor=blue,
     citecolor = blue,      
     urlcolor=cyan,
     }

\newcommand{\EE}{\mathbb{E}}
\newcommand{\RR}{\mathbb{R}}
\newcommand{\CC}{\mathbb{C}}
\newcommand{\NN}{\mathbb{N}}
\newcommand{\ZZ}{\mathbb{Z}}
\newcommand{\QQ}{\mathbb{Q}}
\newcommand{\FF}{\mathbb{F}}

\newcommand{\C}{\mathcal{C}}
\newcommand{\D}{\mathcal{D}}
\newcommand{\F}{\mathcal{F}}
\newcommand{\Q}{\mathcal{Q}}

\newcommand{\PP}{\mathcal{P}}
\renewcommand{\P}{\mathcal{P}}

\newcommand{\floor}[1]{\left\lfloor #1 \right\rfloor}

\newcommand{\Lip}{{\rm{Lip}}}
\newcommand{\poly}{{\rm{poly}}}

\newcommand{\Span}{{\rm{Span}}}

\theoremstyle{theorem}
\newtheorem{theorem}{Theorem}[section]
\newtheorem{definition}[theorem]{Definition}
\newtheorem{proposition}[theorem]{Proposition}

\newtheorem{corollary}[theorem]{Corollary}
\newtheorem{conjecture}[theorem]{Conjecture}
\newtheorem{lemma}[theorem]{Lemma}

\theoremstyle{definition}

\title{True complexity of polynomial progressions in finite fields}
\author{Borys Kuca}
\address{University of Manchester}
\email{borys.kuca@manchester.ac.uk}
\date{}

\begin{document}

\maketitle

\begin{abstract}
    The true complexity of a polynomial progression in finite fields corresponds to the smallest-degree Gowers norm that controls the counting operator of the progression over finite fields of large characteristic. We give a conjecture that relates true complexity to algebraic relations between the terms of the progression, and we prove it for a number of progressions, including $x, x+y, x+y^2, x+y+y^2$ and $x, x+y, x+2y, x+y^2$. As a corollary, we prove an asymptotic for the count of certain progressions of complexity 1 in subsets of finite fields. In the process, we obtain an equidistribution result for certain polynomial progressions, analogous to the counting lemma for systems of linear forms proved by Green and Tao in \cite{green_tao_2010a}.
    
    %give a construction of a nilmanifold on which certain polynomial sequences equidistribute, which is of independent interest. This construction extends the notion of Leibman group used by Green and Tao in \cite{green_tao_2010a} to prove counting lemma for systems of linear forms.

    %$$x, x+y, x+y^2, x+y+y^2 \\  x, x+y, x+2y, x+y^3, x+2y^3$$
    %(1) $x, x+y, x+y^2, x+y+y^2$; (2) $x, x+y, x+2y, x+y^3, x+2y^3$; and (3) $x, x+y, x+2y, x+y^2$. In the process, we give a construction of a nilmanifold on which certain polynomial sequences equidistribute, which is of independent interest.
\end{abstract}

\section{Introduction}
%Throughout the paper, we use the letter $p$ to denote a (large) prime.

%A recurrent theme in combinatorics is to count the number of certain structures in subsets of integers, finite fields, or other groups and rings. 

Let $P_1, ..., P_{t-1}$ be distinct polynomials in $\ZZ[y]$ with zero constant terms. A finite field version of the polynomial Szemeredi theorem states that for any $\alpha>0$, there exists $p_0=p_0(\alpha)\in\NN$ with the following property: if $p>p_0$ is prime and $A\subset\FF_p$ has size $|A|\geqslant \alpha p$, then $A$ contains a polynomial progression
\begin{align}\label{progressions of a special form}
   x,\; x+P_1(y),\; ...,\; x+P_{t-1}(y) 
\end{align}
for some $y\neq 0$. This theorem follows from a multiple recurrence result of Bergelson and Leibman in ergodic theory \cite{bergelson_leibman_1996}. Recently, there have been several attempts at proving a quantitative version of the theorem using ideas from analytic number theory \cite{bourgain_chang_2017}, algebraic geometry \cite{dong_li_sawin_2017} or Fourier analysis \cite{peluse_2018, peluse_2019b, kuca_2019}, which gave explicit estimates for the quantity $p_0$ for certain families of progressions (\ref{progressions of a special form}).\ A recurrent idea in these recent approaches is to estimate the number of progressions (\ref{progressions of a special form}) in an arbitrary subset $A\subset\FF_p$. In the paper, we prove qualitative estimates for the counts of certain polynomial configurations by relating them to the counts of certain linear forms. Throughout, we let $p$ denote a (large) prime. 

%A recurrent idea in the recent approaches to polynomial Szemeredi theorem is to estimate the number of progressions (\ref{progressions of a special form}) in an arbitrary subset $A\subset\FF_p$. In the paper, we prove qualitative estimates for the counts of certain polynomial configurations by relating them to the counts of certain linear forms. Throughout, we let $p$ denote a (large) prime. 

\begin{theorem}\label{asymptotic count for systems of complexity 1}
Let $A\subseteq\FF_p$.
\begin{enumerate}
    \item The count of $x,\; x+y,\; x+y^2,\; x+y+y^2$ in $A$ is given by
    \begin{align*}
    &|\{(x,\; x+y,\; x+y^2,\; x+y+y^2)\in A^4: x,y\in\FF_p\}|\\ 
    &= \frac{1}{p}|\{(x,y,u,z)\in A^4: x+y=u+z\}| + o(p^2).
\end{align*}
More generally, this estimate holds whenever $x,\; x+y,\; x+y^2,\; x+y+y^2$ is replaced by $x,\; x+Q(y),\; x+R(y),\; x+Q(y)+R(y)$ for any polynomials $Q,R\in\ZZ[y]$ with $1\leqslant\deg Q < \deg R$ and zero constant terms.
\item The count of $x,\; x+y,\; x+2y,\; x+y^3,\; x+2y^3$ in $A$ is given by
\begin{align*}
    &|\{(x,\; x+y,\; x+2y,\; x+y^3,\; x+2y^3)\in A^5: x,y\in\FF_p\}|\\
    &= \frac{1}{p}|\{(x,\; x+y,\; x+2y,\; x+z,\; x+2z)\in A^5: x,y,z\in\FF_p\}| + o(p^2).
\end{align*}
More generally, this estimate holds whenever $x,\; x+y,\; x+2y,\; x+y^3,\; x+2y^3$ is replaced by $x,\; x+Q(y),\; x+2Q(y),\; x+R(y),\; x+2R(y)$ for any polynomials $Q,R\in\ZZ[y]$ with $1\leqslant\deg Q<(\deg R)/2$ and zero constant terms.
\end{enumerate}
\end{theorem}

We obtain results like Theorem \ref{asymptotic count for systems of complexity 1} by analysing counting operators of the form
\begin{align}\label{counting operator for polynomial progressions}
    \EE_{x,y\in\FF_p}f_0(x)f_1(x+P_1(y))\cdots f_{t-1}(x+P_{t-1}(y))
\end{align}
for distinct polynomials $P_1, ..., P_{t-1}\in\ZZ[y]$ and functions $f_0, ..., f_{t-1}:\FF_p\to\CC$ that are \emph{1-bounded}, i.e.\ satisfy $\|f_i\|_\infty\leqslant 1$. A useful tool to study polynomial progressions is a family of norms on functions $f:\FF_p\to\CC$ defined by
\begin{align}\label{Gowers norm}
    \|f\|_{U^s}&=\left(\EE_{x, h_1, ..., h_s\in\FF_p}\prod_{w\in\{0,1\}^s} \C^{|w|}f(x+w_1 h_1 + ... + w_s h_s)\right)^\frac{1}{2^s},
\end{align}
where $\C: z\mapsto \overline{z}$ is the conjugacy operator and $|w|=w_1+...+w_s$. We call $\|f\|_{U^s}$ the \emph{Gowers norm} of $f$ of {degree} $s$, and we discuss its properties in Section \ref{section on higher order Fourier analysis}. It was proved in \cite{gowers_2001} that Gowers norms control arithmetic progressions, in the sense that 
\begin{align}\label{Gowers norms control arithmetic progression}
    |\EE_{x,y\in\FF_p}f_0(x)f_1(x+y)\cdots  f_{s}(x+sy)|\leqslant\min_{0\leqslant i\leqslant s} \|f_i\|_{U^{s}}
\end{align}
for all 1-bounded $f_0, ..., f_{s}:\FF_p\to\CC$. 
%The norm $U^s$ on functions $f:\FF_p\to\CC$ is called the \emph{Gowers norm of degree} $s$. 
A similar argument has been used to show that Gowers norms control any system of linear forms that are pairwise linearly independent (Proposition 7.1 of \cite{green_tao_2010b}). Finally, Gowers norms are also known to control polynomial progressions of the form (\ref{progressions of a special form}) for distinct nonzero polynomials $P_1, ..., P_{t-1}\in\ZZ[y]$ with zero constant terms (Proposition 2.2 of \cite{peluse_2019b}), in that there exist $s\in\NN_+$ and $c>0$ depending only on $P_1, ..., P_{t-1}$ such that
\begin{align*}
    |\EE_{x,y\in\FF_p}f_0(x)f_1(x+P_1(y)) \cdots f_{t-1}(x+P_{t-1}(y))|\leqslant \min_{0\leqslant i\leqslant t-1}\|f_i\|_{U^{s}}^c+O(p^{-c})
\end{align*}
for all 1-bounded $f_0, ..., f_{t-1}:\FF_p\to\CC$.

In the light of the monotonicity property of Gowers norms
\begin{align*}
    \|f\|_{U^1}\leqslant \|f\|_{U^2}\leqslant \|f\|_{U^3}\leqslant ...,
\end{align*}
derived e.g. in Section 1 of \cite{green_tao_2008}, it is natural to ask what is the smallest-degree Gowers norm controlling a given configuration. The smallest $s$ such that $U^{s+1}$ controls the configuration is called its \emph{true complexity}; the precise definition shall be given in Section \ref{subsection on true complexity}. The question of determining true complexity has been posed and partially resolved for linear configurations in \cite{gowers_wolf_2010, gowers_wolf_2011a, gowers_wolf_2011b, gowers_wolf_2011c, green_tao_2010a}, where the authors relate true complexity to algebraic relations between the linear forms in the configuration. It remains largely open for general polynomial progressions (\ref{progressions of a special form}).

In the paper, we determine the true complexity of several polynomial progressions. Our main results are the following theorems.
\begin{theorem}[True complexity of $x,\; x+y,\; x+y^2,\; x+y+y^2$]\label{True complexity of x, x+y, x+y^2, x+y+y^2}
For any $\epsilon>0$, there exist $\delta>0$ and $p_0\in\NN$ such that for all primes $p>p_0$ and for all 1-bounded functions $f_0, f_1, f_2, f_3:\FF_p\to\CC$, at least one of which satisfies $\|f_i\|_{U^2}\leqslant \delta$, we have
\begin{align}\label{control of x, x+y, x+y^2, x+y+y^2}
    |\EE_{x,y\in\FF_p}f_0(x)f_1(x+y)f_2(x+y^2)f_3(x+y+y^2)|\leqslant \epsilon.
\end{align}
\end{theorem}
The $U^2$ norm is related to Fourier analysis via ${\|\hat{f}\|_{\infty}\leqslant\|f\|_{U^2}\leqslant \|\hat{f}\|_\infty^\frac{1}{2}}$ for 1-bounded functions $f$, and so Theorem \ref{True complexity of x, x+y, x+y^2, x+y+y^2} can be informally rephrased as follows: if at least one of $f_0, f_1, f_2, f_3$ has no large Fourier coefficient, then the operator in (\ref{control of x, x+y, x+y^2, x+y+y^2}) is small. We can similarly interpret the next three results.

\begin{theorem}[True complexity of $x,\; x+Q(y),\; x+R(y),\; x+Q(y)+R(y)$]\label{True complexity of x, x+Q(y), x+R(y), x+Q(y)+R(y)}
Let $Q,R\in\ZZ[y]$ be polynomials of zero constant terms satisfying $1\leqslant\deg Q < \deg R$. For any $\epsilon>0$, there exist $\delta>0$ and $p_0\in\NN$ such that for all primes $p>p_0$ and for all 1-bounded functions $f_0, f_1, f_2, f_3:\FF_p\to\CC$, at least one of which satisfies $\|f_i\|_{U^2}\leqslant \delta$, we have
\begin{align}\label{control of x, x+Q(y), x+R(y), x+Q(y)+R(y)}
    |\EE_{x,y\in\FF_p}f_0(x)f_1(x+Q(y))f_2(x+R(y))f_3(x+Q(y)+R(y))|\leqslant \epsilon.
\end{align}
\end{theorem}

\begin{theorem}[True complexity of $x,\; x+y,\; x+2y,\; x+y^3,\; x+2y^3$]\label{True complexity of x, x+y, x+2y, x+y^3, x+2y^3}
For any $\epsilon>0$, there exist $\delta>0$ and $p_0\in\NN$ such that for all primes $p>p_0$ and for all 1-bounded functions $f_0, f_1, f_2, f_3, f_4:\FF_p\to\CC$, at least one of which satisfies $\|f_i\|_{U^2}\leqslant \delta$, we have
\begin{align*}
    |\EE_{x,y\in\FF_p}f_0(x)f_1(x+y)f_2(x+2y)f_3(x+y^3)f_4(x+2y^3)|\leqslant \epsilon.
\end{align*}
\end{theorem}
\begin{theorem}[True complexity of $x,\; x+Q(y),\; x+2Q(y),\; x+R(y),\; x+2R(y)$]\label{True complexity of x, x+Q(y), x+2Q(y), x+R(y), x+2R(y)}
Let $Q,R\in\ZZ[y]$ be polynomials of zero constant terms satisfying $1\leqslant\deg Q<(\deg R)/2$. For any $\epsilon>0$, there exist $\delta>0$ and $p_0\in\NN$ such that for all primes $p>p_0$ and for all 1-bounded functions $f_0, f_1, f_2, f_3, f_4:\FF_p\to\CC$, at least one of which satisfies $\|f_i\|_{U^2}\leqslant \delta$, we have
\begin{align*}
    |\EE_{x,y\in\FF_p}f_0(x)f_1(x+Q(y))f_2(x+2Q(y))f_3(x+R(y))f_4(x+2R(y))|\leqslant \epsilon.
\end{align*}
\end{theorem}
In the aforementioned configurations, each term can be controlled by the same Gowers norm, the $U^2$ norm. However, there exist configurations where different Gowers norms control different terms. The following is but one example. The norm $u^3$ appearing below belongs to a family of norms called \emph{polynomial bias norms} which satisfy $\|f\|_{u^s}\leqslant \|f\|_{U^s}$ and will be discussed more broadly in Section \ref{section on higher order Fourier analysis}. 
\begin{theorem}[True complexity of $x,\; x+y,\; x+2y,\; x+y^2$]\label{True complexity of x, x+y, x+2y, x+y^2}
For any $\epsilon>0$, there exist $\delta>0$ and $p_0\in\NN$ such that for all primes $p>p_0$ and for all 1-bounded functions ${f_0, f_1, f_2, f_3:\FF_p\to\CC}$, where at least one of $f_0, f_1, f_2$ satisfies $\|f_i\|_{u^3}\leqslant \delta$ or $f_3$ satisfies $\|f_3\|_{U^2}\leqslant \delta$, we have
\begin{align}\label{control of x, x+y, x+2y, x+y^2}
    |\EE_{x,y\in\FF_p}f_0(x)f_1(x+y)f_2(x+2y)f_3(x+y^2)|\leqslant \epsilon.
\end{align}
\end{theorem}
What Theorem \ref{True complexity of x, x+y, x+2y, x+y^2} is saying is that if all Fourier coefficients of $f_3$ are small, i.e. if the correlation $$\EE_{x\in\FF_p}f_3(x)e_p(\alpha x)$$ is small for all $\alpha\in\FF_p$, or if the correlation $$\EE_{x\in\FF_p}f_i(x)e_p\left(\alpha x^2 +\beta x\right)$$ is small for all $\alpha, \beta\in\ZZ$ and $i\in\{0,1,2\}$, then the operator (\ref{control of x, x+y, x+2y, x+y^2}) is small as well. The function $e_p$ used here is $e_p(x):=e^{\frac{2\pi i}{p}x}.$

%at least one of $f_0, f_1, f_2$ has small Fourier coefficient, i.e. the correlation $$\EE_{x\in\FF_p}f_i(x)e_p(\alpha x)$$ is small for all $\alpha\in\FF_p$ and $i\in\{0,1,2\}$, or if the correlation $$\EE_{x\in\FF_p}f_3(x)e_p\left(\alpha x^2 +\beta x\right)$$ is small for all $\alpha, \beta\in\ZZ$, then the operator (\ref{control of x, x+y, x+2y, x+y^2}) is small as well. 

\begin{theorem}[True complexity of $x,\; x+y,\;...,\; x+(m-1)y,\; x+y^d$]\label{True complexity of x, x+y, ..., x+(m-1)y, x+y^d}
Let $m,d\in\NN_+$ satisfy $2\leqslant d\leqslant m-1$. For any $\epsilon>0$, there exist $\delta>0$ and $p_0\in\NN$ such that for all primes $p>p_0$ and for all 1-bounded functions $f_0, ...,f_{m}:\FF_p\to\CC$, we have
\begin{align}\label{control of x, x+y,..., x+(m-1)y, x+y^d}
    |\EE_{x,y\in\FF_p}f_0(x)f_1(x+y) \cdots f_{m-1}(x+(m-1)y)f_m(x+y^d)|\leqslant \epsilon
\end{align}
if $f_m$ satisfies $\|f_m\|_{U^{\floor{\frac{m-1}{d}}+1}}\leqslant \delta$ or at least one of $f_0, ..., f_{m-1}$ satisfies $\|f_i\|_{U^s}\leqslant \delta$ for
\[ s = \begin{cases}
m, \; &d\ |\ m-1\\
m-1, \; &d \nmid m-1.
\end{cases}
\]
\end{theorem}

The main technical innovation used in proving Theorems \ref{True complexity of x, x+y, x+y^2, x+y+y^2}, \ref{True complexity of x, x+Q(y), x+R(y), x+Q(y)+R(y)}, \ref{True complexity of x, x+y, x+2y, x+y^3, x+2y^3} and \ref{True complexity of x, x+Q(y), x+2Q(y), x+R(y), x+2R(y)} is the following equidistribution result which should be seen as an extension of (the periodic version of) Theorem 1.2 from \cite{green_tao_2010a}, which itself generalizes classical equidistribution theorems by Weyl and van der Corput. All the concepts appearing in Theorem \ref{equidistribution theorem in the intro} shall be defined and discussed in Sections \ref{section on higher order Fourier analysis} and \ref{section on Leibman nilmanifold}. Mimicking the notation of \cite{green_tao_2010a}, we say that an expression $E(A,M)$ depending on parameters $A,M>0$ satisfies $o_{A\to\infty,M}(1)$ if $\lim\limits_{A\to\infty}E(A,M)=0$ for each fixed $M>0$, and similarly for other choices of parameters.
\begin{theorem}\label{equidistribution theorem in the intro}
Let $M>0$, and let $\vec{P}\in\ZZ[x,y]^t$ with $\vec{P}(0,0)=\vec{0}$ take one of the following forms:
\begin{enumerate}
    \item $\vec{P}(x,y) = (x, \; x+Q(y), \; x+R(y), \; x+Q(y)+R(y))$ for $1\leqslant \deg Q<\deg R$;
    \item $\vec{P}(x,y) = (x, \; x+Q(y),\; x+2Q(y), \; x+R(y), \; x+2R(y))$ for $1\leqslant \deg Q<(\deg R)/2$.
\end{enumerate}
Given a filtered nilmanifold $G/\Gamma$  of complexity $M$, there exists a filtered nilmanifold $G^P/\Gamma^P\subseteq G^t/\Gamma^t$ of complexity $O_{M}(1)$ such that for any $p$-periodic, $A$-irrational sequence $g\in\poly(\ZZ,G_\bullet)$ satisfying $g(0)=1$, the sequence $g^P\in\poly(\ZZ^2,G^P_\bullet)$ given by
\begin{align*}
    g^P(x,y) = (g(P_1(x,y)), ..., g(P_t(x,y)))
\end{align*}
satisfies
\begin{align*}
    \EE_{x,y\in\FF_p}F(g^P(x,y)\Gamma^P) = \int_{G^P/\Gamma^P} F + o_{A\to\infty, M}(1)
\end{align*}
uniformly for any $M$-Lipschitz function $F:G^P/\Gamma^P\to\CC$.
\end{theorem}
What Theorem \ref{equidistribution theorem in the intro} is saying is that if $g$ is a ``highly irrational" sequence on $G/\Gamma$ in the sense of Definition \ref{irrational sequences}, then the sequence $g^P$ is ``close to being equidistributed" on the nilmanifold $G^P/\Gamma^P$. Combining Theorem \ref{equidistribution theorem in the intro} with Theorem \ref{regularity lemma}, a version of the celebrated arithmetic regularity lemma (Theorem 1.2 of \cite{green_tao_2010a}), we can then approximate sums of the form (\ref{counting operator for polynomial progressions}) by integrals of Lipschitz functions on nilmanifolds.

Although we only prove Theorem \ref{equidistribution theorem in the intro} for two specific families of polynomial progressions, the construction of the nilmanifold $G^P/\Gamma^P$ mentioned in the statement of Theorem \ref{equidistribution theorem in the intro} is quite general. This nilmanifold, definable for any polynomial progression, has originally appeared in Section 5 of \cite{leibman_2007}. Its versions for linear forms have been used in \cite{green_tao_2010a}, where the authors call it ``Leibman nilmanifold", and we shall stick to this terminology. In Section \ref{section on Leibman nilmanifold}, we show that Leibman nilmanifold admits a natural filtration $G^P_\bullet$ such that $g^P\in\poly(\ZZ^D,G^P_\bullet)$ whenever ${g\in\poly(\ZZ,G_\bullet)}$. 
%When $\vec{P}$ is a system of linear forms, the nilmanifold $G^P/\Gamma^P$ is just the Leibman group defined in Definition 1.10 of \cite{green_tao_2010a}. 

Interestingly, there exist progressions for which Theorem \ref{equidistribution theorem in the intro} fails. In Lemma \ref{example for failure of equidistribution}, we give an example of a torus $G/\Gamma$ and a sequence ${g\in\poly(\ZZ,G_\bullet)}$ that is ``highly irrational" on $G/\Gamma$, yet the corresponding sequence $g^P$ for $\vec{P}(x,y)=(x, \; x+y, \; x+2y, \; x+y^2)$ is contained in a ``low-rank" subtorus of $G^P/\Gamma^P$ and is ``far from being equidistributed" on $G^P/\Gamma^P$. This seems to be a novel phenomenon, unregistered in the existing literature, and it is essentially connected with the fact that the terms of the progression satisfy an algebraic relation
\begin{align}\label{quadratic relation satisfied by x, x+y, x+2y, x+y^2}
    \left(\frac{1}{2}x^2+x\right) - \left(x+y\right)^2 + \frac{1}{2}\left(x+2y\right)^2 - (x+y^2) = 0
\end{align}
that is ``inhomogeneous", in the sense that the polynomials in $x, \; x+y$ and $x+2y$ are quadratic, but the polynomial in $x+y^2$ is linear. This phenomenon does not appear for linear forms, where such inhomogeneous relations are impossible, but it is an important feature of the polynomial world. It also does not show up in the ergodic work on this configuration \cite{frantzikinakis_2008} since the author only has to deal with linear sequences $g(n)=a^n$ as opposed to more general polynomial sequences. The natural analogue of Theorem \ref{equidistribution theorem in the intro} cannot therefore be used to prove Theorems \ref{True complexity of x, x+y, x+2y, x+y^2} and \ref{True complexity of x, x+y, ..., x+(m-1)y, x+y^d}, and we apply a different method to handle these configurations, which essentially comes down to homogenizing the progression $\vec{P}$ using the Cauchy-Schwarz inequality. 

%Although we only prove Theorem \ref{equidistribution theorem in the intro} for two specific families of polynomial progressions, the construction of the nilmanifold $G^P/\Gamma^P$ mentioned in the statement of Theorem \ref{equidistribution theorem in the intro} is quite general. This nilmanifold can be defined for any polynomial progression, and for any progression we can endow it with a filtration such that $g^P\in\poly(\ZZ^t,G^P_\bullet)$ whenever ${g\in\poly(\ZZ,G_\bullet)}$. When $\vec{P}$ is a system of linear forms, the nilmanifold $G^P/\Gamma^P$ is just the Leibman group defined in Definition 1.10 of \cite{green_tao_2010a}. Interestingly, however, there exist progressions for which Theorem \ref{equidistribution theorem in the intro} fails. In Lemma \ref{example for failure of equidistribution}, we give an example of a torus $G/\Gamma$ and a sequence ${g\in\poly(\ZZ,G_\bullet)}$ that is ``highly irrational" on $G/\Gamma$, yet the corresponding sequence $g^P$ for $\vec{P}(x,y)=(x, \; x+y, \; x+2y, \; x+y^2)$ is contained in a subtorus of $G^P/\Gamma^P$ and is ``far from being equidistributed" on $G^P/\Gamma^P$. Therefore, the natural analogue of Theorem \ref{equidistribution theorem in the intro} cannot be used to prove Theorems \ref{True complexity of x, x+y, x+2y, x+y^2} and \ref{True complexity of x, x+y, ..., x+(m-1)y, x+y^d}, and we use a different method to handle these configurations.

 We have not been able to extend Theorems \ref{True complexity of x, x+y, x+y^2, x+y+y^2}, \ref{True complexity of x, x+Q(y), x+R(y), x+Q(y)+R(y)}, \ref{True complexity of x, x+y, x+2y, x+y^3, x+2y^3}, \ref{True complexity of x, x+Q(y), x+2Q(y), x+R(y), x+2R(y)}, \ref{True complexity of x, x+y, x+2y, x+y^2} and \ref{True complexity of x, x+y, ..., x+(m-1)y, x+y^d} to an arbitrary polynomial progression of the form (\ref{progressions of a special form}) for at least three reasons. First, not all progressions satisfy Theorems \ref{equidistribution theorem in the intro}, as evidenced by the aforementioned example of $x,\; x+y,\; x+2y,\; x+y^2$. Second, some progressions fail to satisfy ``filtration condition" (Definition \ref{filtration condition}), which makes the corresponding nilmanifolds $G^P/\Gamma^P$ harder to analyse. Sections \ref{section on Leibman nilmanifold} and \ref{section on a progression not satisfying filtration condition} will explain why this condition is useful; it is unclear at the moment if this is a mere technical annoyance or a genuine obstruction. Third, even though our arguments in the proof of Theorem \ref{equidistribution theorem in the intro} follow the same general strategy for the two families of configurations that the theorem concerns, we have to resort to different tricks in dealing with arising technicalities, which makes it hard to generalize the arguments. 
 
 %These differences are significant enough for us to doubt that one can generalize Theorems \ref{True complexity of x, x+y, x+y^2, x+y+y^2}, \ref{True complexity of x, x+Q(y), x+R(y), x+Q(y)+R(y)}, \ref{True complexity of x, x+y, x+2y, x+y^3, x+2y^3}, \ref{True complexity of x, x+Q(y), x+2Q(y), x+R(y), x+2R(y)}, \ref{True complexity of x, x+y, x+2y, x+y^2} and \ref{True complexity of x, x+y, ..., x+(m-1)y, x+y^d} to an arbitrary progression of the form (\ref{progressions of a special form}) using just simple variations of our methods.
%can be defined for any polynomial progression. It moreover satisfies the property  
\subsection{True complexity: formal definition, conjecture and known results}\label{subsection on true complexity}
Our primary object of study are \emph{integral polynomial maps}, i.e. configurations of the form $\vec{P}=(P_1, ..., P_t)\in\QQ[\textbf{x}]^t$ for integer-valued polynomials $P_1, ..., P_t$ with zero constant terms. %satisfying $P_i(x+p)-P_i(x)\equiv 0 \mod p$ for all sufficiently large primes $p$.
Following the convention of \cite{green_tao_2010a}, we use $\vec{v}$ to denote $t$-dimensional vectors and $\textbf{x}$ to denote $D$-dimensional vectors. We are now ready to state our main definition.

\begin{definition}[True complexity and Gowers controllability]\label{true complexity}
Let $\vec{P}=(P_1, ..., P_t)\in\QQ[\textbf{x}]^t$ be an integral polynomial map. We say that $\vec{P}$ has \emph{true complexity $s$ at an index $1\leqslant i\leqslant t$} if $s$ is the smallest natural number such that for every $\epsilon>0$, there exist $\delta>0$ and $p_0\in\NN$ such that for all primes $p>p_0$ and all 1-bounded functions $f_1, ..., f_t:\FF_p\to\CC$, we have
\begin{align*}
    |\EE_{\textbf{x}\in\FF_p^D}f_1(P_1(\textbf{x})) \cdots f_t(P_t(\textbf{x}))| < \epsilon
\end{align*}
whenever $\|f_i\|_{U^{s+1}}<\delta$. If no such $s$ exists, we say that the true complexity of $\vec{P}$ at $i$ is $\infty$. We say $\vec{P}$ has \emph{true complexity} $s$ if it has true complexity $s$ at the index $i$ for all $1\leqslant i\leqslant t$. We call $\vec{P}$ \emph{Gowers controllable} if its true complexity at every index is finite.
\end{definition}
The true complexity question could also be posed over the integers; however, the problem becomes much harder in this context because each variable would be drawn from an interval of different length depending on the degree of the polynomial map. For instance, when studying the true complexity of $x,\; x+y,\; x+y^2$, the variable $x$ would be drawn from an interval of length $N$ while $y$ would be taken from an interval of length only $\Theta(\sqrt{N})$ to ensure that the term $x+y^2$ lies inside the interval $\{1, ..., N\}$. As a consequence, not every term of the progression would be globally controlled by a Gowers norm. We refer the reader to \cite{peluse_prendiville_2019, peluse_prendiville_2020, peluse_2019a} for an in-depth discussion of these issues.

Because of the issues highlighted above, we study true complexity over finite fields as opposed to integers.
%look at finite fields since the technical details are more manageable in this case and do not divert attention from underlying structure. 
In the language of Definition \ref{true complexity}, Theorems \ref{True complexity of x, x+y, x+y^2, x+y+y^2}, \ref{True complexity of x, x+Q(y), x+R(y), x+Q(y)+R(y)}, \ref{True complexity of x, x+y, x+2y, x+y^3, x+2y^3}, \ref{True complexity of x, x+Q(y), x+2Q(y), x+R(y), x+2R(y)}, \ref{True complexity of x, x+y, x+2y, x+y^2} and \ref{True complexity of x, x+y, ..., x+(m-1)y, x+y^d} can be restated as follows.
\begin{theorem}\label{corrolary for true complexity}
~\
\begin{enumerate}
    \item The configuration $x,\; x+Q(y),\; x+R(y),\; x+Q(y)+R(y)$ has true complexity 1 for any $Q,R\in\ZZ[y]$ of zero constant terms satisfying $1\leqslant \deg Q<\deg R$.
    \item The configuration $x,\; x+Q(y),\; x+2Q(y),\; x+R(y),\; x+2R(y)$ has true complexity 1 for any $Q,R\in\ZZ[y]$ of zero constant terms satisfying $1\leqslant \deg Q<(\deg R)/2$.
    \item The configuration $x,\; x+y,\; ..., x+(m-1)y,\; x+y^{d}$ has true complexity $m-1$ at $i\in\{0, 1, ...,m-1\}$ and $\frac{m-1}{d}$ at $i = m$ whenever $2\leqslant d\leqslant m-1$ and $d\ |\ m-1$.
    \item The configuration $x,\; x+y,\; ..., x+(m-1)y,\; x+y^d$ has true complexity $m-2$ at $i\in\{0, 1, ...,m-1\}$ and $\left\lfloor\frac{m-1}{d}\right\rfloor$ at $i = m$ whenever $2\leqslant d<m-1$ and $d \nmid m-1$.
\end{enumerate}
\end{theorem}

True complexity turns out to be intimately connected with the algebraic relations between the terms of a polynomial progression. We first state the following definition.
\begin{definition}[Algebraic independence]
Let $\vec{P}=(P_1, ..., P_t)\in\QQ[\textbf{x}]^t$ be an integral polynomial map and fix $1 \leqslant i \leqslant t$. The progression $\vec{P}$ is \emph{algebraically independent of degree $s+1$ at $i$} if, whenever we have
\begin{align*}
    Q_1(P_1(\textbf{x}))+ ... + Q_t(P_t(\textbf{x})) = 0,
\end{align*}
for some $Q_1, ..., Q_t\in\ZZ[y]$, the polynomial $Q_i$ has degree at most $s$. We moreover say that $\vec{P}$ is \emph{algebraically independent of degree $s+1$} if it is algebraically independent of degree $s+1$ at $i$ for all $1\leqslant i\leqslant t$.
\end{definition}
\begin{conjecture}[Conjecture for true complexity]\label{conjecture for true complexity}
Let $\vec{P}=(P_1, ..., P_t)\in\QQ[\textbf{x}]^t$ be a Gowers controllable integral polynomial map and fix $1 \leqslant i \leqslant t$. The true complexity of $\vec{P}$ at $i$ is the smallest natural number $s$ for which $\vec{P}$ is algebraically independent of degree $s+1$ at $i$.
\end{conjecture}

Theorems \ref{True complexity of x, x+y, x+y^2, x+y+y^2}, \ref{True complexity of x, x+Q(y), x+R(y), x+Q(y)+R(y)}, \ref{True complexity of x, x+y, x+2y, x+y^3, x+2y^3}, \ref{True complexity of x, x+Q(y), x+2Q(y), x+R(y), x+2R(y)}, \ref{True complexity of x, x+y, x+2y, x+y^2} and \ref{True complexity of x, x+y, ..., x+(m-1)y, x+y^d} confirm  Conjecture \ref{conjecture for true complexity} in special instances. The terms of $x,\; x+y,\; x+y^2,\; x+y+y^2$ satisfy one linear relation (up to scaling) $$x - (x+y) - (x+y^2) + (x+y+y^2) = 0, $$
and so the configuration has complexity 1. Analogously, there is a unique linear relation (up to scaling)
\begin{align*}
    x - (x+Q(y)) - (x+R(y)) + (x+Q(y)+R(y)) = 0
\end{align*}
between the terms of $x, \; x+Q(y),\; x+R(y),\; x+Q(y)+R(y)$ whenever $Q,R\in\ZZ[y]$ have zero constant terms and satisfy $1\leqslant \deg Q<\deg R$. For $x,\; x+y,\; x+2y,\; x+y^3,\; x+2y^3$, there are two linearly independent relations
\begin{align*}
    x - 2(x+y) + (x+2y) = 0 \quad {\rm{and}} \quad x - 2(x+y^3) - (x+2y^3)=0,
\end{align*}
and similarly for  $x,\; x+Q(y),\; x+2Q(y),\; x+R(y),\; x+2R(y)$ whenever $Q,R\in\ZZ[y]$ have zero constant terms and satisfy $1\leqslant \deg Q<(\deg R)/2$
\begin{align*}
    x - 2(x+Q(y)) + (x+2Q(y)) = 0 \quad {\rm{and}} \quad x - 2(x+R(y)) - (x+2R(y))=0.
\end{align*}
One can check by hand that none of these progressions satisfy a higher order relation. By contrast, the terms of  $x,\; x+y,\; x+2y,\; x+y^2$ satisfy a quadratic relation
\begin{align*}
%\label{quadratic relation satisfied by x, x+y, x+2y, x+y^2}
    \left(\frac{1}{2}x^2+x\right) - \left(x+y\right)^2 + \frac{1}{2}\left(x+2y\right)^2 - (x+y^2) = 0
\end{align*}
in addition to the linear relation $x - 2(x+y) + (x+2y) = 0$, which explains why this configuration has true complexity 2 at $i = 0, 1, 2$ and 1 at $i=3$. For $x,\; x+y,\; ...,\; x+(m-1)y,\; x+y^d$ with $2\leqslant d\leqslant m-1$, there exist polynomials $Q_0,\; ...,\; Q_{m-1}$ of degree $d\left\lfloor\frac{m-1}{d}\right\rfloor$ satisfying
\begin{align*}
    Q_0(x) + Q_1(x+y) + ... + Q_{m-1}(x+(m-1)y) = (x+y^d)^{\left\lfloor\frac{m-1}{d}\right\rfloor},
\end{align*}
in addition to lower-degree relations and an algebraic relation of degree $m-2$ between the terms of $x,\; x+y,\; ...,\; x+(m-1)y$.

The lower bound in Conjecture \ref{conjecture for true complexity} is straightforward to settle. The difficulty lies in proving the upper bound.
\begin{theorem}[Lower bound for true complexity]\label{lower bound for true complexity}
Let $\vec{P}=(P_1, ..., P_t)\in\QQ[\textbf{x}]^t$ be an integral polynomial map and fix $1 \leqslant i \leqslant t$. Suppose that $\vec{P}$ is not algebraically independent of degree $s$ at index $i$. Then the true complexity of $\vec{P}$ at $i$ is at least $s$.
\end{theorem}
\begin{proof}
By assumption, there exists an algebraic relation 
\begin{align*}
    Q_1(P_1(\textbf{x}))+ ... + Q_t(P_t(\textbf{x})) = 0,
\end{align*}
for some $Q_1, ..., Q_t\in\ZZ[y]$, where $Q_i$ has degree at least $s$. Let $f_j(x) = e_p(Q_j(x))$ for each $1\leqslant j\leqslant t$. The functions $f_j$ are clearly 1-bounded. It follows from the properties of additive characters that 
$$\EE_{\textbf{x}\in\FF_p^D}f_1(P_1(\textbf{x})) \cdots f_t(P_t(\textbf{x})) = \EE_{\textbf{x}\in\FF_p^D}e_p(Q_1(P_1(\textbf{x}))+ ... + Q_t(P_t(\textbf{x}))) = 1.$$
To prove the theorem, we want to show that $\|f_i\|_{U^{s}}$ is small, which will imply that the $U^{s}$ norm cannot control the $P_i$ term of the configuration. The definition (\ref{Gowers norm}) of Gowers norms can be restated as
\begin{align}\label{expansion in Gowers norm}
    \|f_i\|_{U^{s}}^{2^{s}}=\EE_{x, {h}_1,..., {h}_{s}\in\FF_p}\Delta_{{h}_1, ...,{h}_{s}}f_i({x}),
\end{align}
where $\Delta_{h}f({x}):=f({x}+{h})\overline{f({x})}$ and $\Delta_{{h}_1, ...,{h}_{s}}=\Delta_{{h}_1}\cdots\Delta_{{h}_{s}}$. Since $Q_i$ has degree at least $s$ and $e_p(\cdot)$ is an additive character, the function $\Delta_{{h}_1, ...,{h}_{s}}f_i({x})$ is of the form $e_p(Q(x, {h}_1, ..., {h}_{s}))$ for a nonconstant polynomial $Q$. By properties of exponential sums, the sum in (\ref{expansion in Gowers norm}) is of size $O_{s}(p^{-c_{s}})$, and so 
\begin{align*}
    \|f_i\|_{U^{s}} \ll_{s} p^{-c_{s}}.
\end{align*}
Thus $U^{s}$ norm does not control the $P_i$ term of the configuration, implying the theorem.
\end{proof}
The constants appearing in the proof of Theorem \ref{lower bound for true complexity}, like in all other proofs in the paper, are allowed to depend on the choice of the progression $\vec{P}$. We do not record this dependence so as not to clutter the notation more than necessary.

Theorem \ref{corrolary for true complexity} follows from combining Theorems \ref{True complexity of x, x+y, x+y^2, x+y+y^2}, \ref{True complexity of x, x+Q(y), x+R(y), x+Q(y)+R(y)}, \ref{True complexity of x, x+y, x+2y, x+y^3, x+2y^3}, \ref{True complexity of x, x+Q(y), x+2Q(y), x+R(y), x+2R(y)}, \ref{True complexity of x, x+y, x+2y, x+y^2} and \ref{True complexity of x, x+y, ..., x+(m-1)y, x+y^d} with Theorem \ref{lower bound for true complexity} and remarks above it.

Conjecture \ref{conjecture for true complexity} has so far been proved in a number of special cases. The inequality (\ref{Gowers norms control arithmetic progression}), together with the fact that arithmetic progressions of length $t$ satisfy an algebraic relation of degree $t-2$, proves it for arithmetic progressions. The work by Green and Tao \cite{green_tao_2010a} settles it for all linear configurations satisfying a technical condition called flag condition\footnote{The published version of \cite{green_tao_2010a} claims to prove Conjecture \ref{conjecture for true complexity} for all linear configurations. However, it has been announced in November 2020 that there is an error in Green and Tao's argument, and that the argument only works if a linear form satisfies the flag condition. See \cite{tao_2020} for discussion.}, with certain cases having been previously proved by Gowers and Wolf \cite{gowers_wolf_2010, gowers_wolf_2011a, gowers_wolf_2011b, gowers_wolf_2011c}. Green and Tao's results are all but ineffective while the work of Gowers and Wolf gives some quantitative bounds. The best bounds, of polynomial type, have been obtained by Manners for linear configurations of length 6 in 3 variables by a skillful use of the Cauchy-Schwarz inequality \cite{manners_2018}. 

As far as nonlinear configurations are concerned, the work of Peluse \cite{peluse_2019b} proves Conjecture \ref{conjecture for true complexity} in a quantitative manner for $x,\; x+P_1(y),\; ...,\; x+P_{t-1}(y)$ whenever the polynomials $P_1$, ..., $P_{t-1}$ are linearly independent; it thus settles the complexity 0 case. For configurations of the form
\begin{align*}
    x,\; x+y,\; ...,\; x+(m-1)y,\; x+P_{m}(y),\; ...,\; x+P_{m+k-1}(y),
\end{align*}
where any nontrivial linear combination of $P_m,\; ...,\; P_{m+k-1}$ has degree at least $m$, the conjecture has been settled quantitatively in \cite{kuca_2019}; similarly for the systems of linear forms with variables being higher powers.

In the works on the true complexity of linear forms \cite{gowers_wolf_2010, gowers_wolf_2011a, gowers_wolf_2011b, gowers_wolf_2011c, green_tao_2010a}, the authors only look at the case $f_1 = ... = f_t$ and care about the degree $s$ such that $U^{s+1}$ controls all the terms of the configuration. For more general polynomial progressions, it however makes sense to allow these functions to be different. Since the terms of the progression may have different degrees, the polynomials $Q_1, ..., Q_t$ appearing in Conjecture \ref{conjecture for true complexity} may have different degrees as well. As a result, it is reasonable to allow that different terms of the progression be controlled by different Gowers norms. 

Like many other questions surrounding Szemer\'{e}di theorem, finding true complexity has a natural analogue in ergodic theory: the problem of determining the smallest characteristic factor. 
%$\mathcal{Y}$ of a measure-preserving dynamical system $(X, \mathcal{X}, \mu, T)$. 
A \emph{characteristic factor} of a measure-preserving dynamical system $(X, \mathcal{X}, \mu, T)$ with respect to (\ref{progressions of a special form}) is a factor $\mathcal{Y}$ of $\mathcal{X}$ such that if $f_0, ..., f_{t-1}\in L^\infty(\mu)$ and the projection of $f_i\in L^\infty(\mu)$ onto $\mathcal{Y}$ satisfies $\EE(f_i|\mathcal{Y})=0$ for some $0\leqslant i\leqslant t-1$, then the product 
\begin{align*}
    f_0(x) f_1(T^{P_1(n)}x) \cdots f_{t-1}(T^{P_{t-1}(n)}x)
\end{align*}
converges to 0 in $L^2(\mu)$. Host and Kra prove in \cite{host_kra_2005a} that there exists a sequence of factors $(\mathcal{Z}_k)_{k\in\NN_+}$ such that $\mathcal{Z}_{t-2}$ is characteristic for arithmetic progressions of length $t$ for $t\geqslant 3$. They show furthermore that each $\mathcal{Z}_k$ is an inverse limit of $k$-step nilsystems; the theory of these factors is fully set up in \cite{host_kra_2018}. In \cite{host_kra_2005b}, it has been shown that for each progression (\ref{progressions of a special form}), there exists $k\in\NN$ such that $\mathcal{Z}_k$ is characteristic, which corresponds to the result of Peluse (Proposition 2.2 of \cite{peluse_2019b}) that (\ref{progressions of a special form}) is controlled by some Gowers norm. Further works \cite{frantzikinakis_kra_2005}, \cite{frantzikinakis_kra_2006}, \cite{frantzikinakis_2008} give the smallest characteristic factors for some specific families of progressions. In particular, Theorems \ref{True complexity of x, x+y, x+y^2, x+y+y^2} and \ref{True complexity of x, x+y, x+2y, x+y^2} are combinatorial, finite-field analogues of the results from \cite{frantzikinakis_2008} that the Kronecker factor $\mathcal{K}$ and the affine factor $\mathcal{A}_2$ are characteristic for $x, \; x+y, \; x+y^2, \; x+y+y^2$ and $x,\; x+y,\; x+2y,\; x+y^2$ respectively. Finally, \cite{leibman_2007} has resolved the question of finding the smallest characteristic factor in the case where the underlying dynamical system is a nilsystem.

%Like many other problems in additive combinatorics, the question of finding true complexity has a natural analogue in ergodic theory, which is the problem of finding the minimal characteristic factor. 
%The question of finding true complexity is analogous to finding minimal characteristic factor in ergodic theory. See e.g. \cite{leibman_2007, frantzikinakis_2008} for more results in this direction.
\subsection*{Outline of the paper}
We start the paper by presenting necessary definitions and results from higher order Fourier analysis in Section \ref{section on higher order Fourier analysis}. We then proceed in Section \ref{section on Leibman nilmanifold} to define Leibman group for a polynomial progression and describe its properties. In particular, we state a filtration condition on a polynomial progression (Definition \ref{filtration condition}) which makes the corresponding Leibman nilmanifold easier to analyze, and which is satisfied by all the progressions that we are interested in. In Section \ref{section on x, x+y, x+y^2, x+y+y^2}, we deduce Theorem \ref{equidistribution theorem in the intro} for the progression $x, \; x+y,\; x+y^2,\; x+y+y^2$
%$\vec{P}(x,y)=(x, \; x+y,\; x+y^2,\; x+y+y^2)$
so as to illustrate our arguments with a specific example. In Sections \ref{section on x, x+Q(y), x+R(y), x+Q(y)+R(y)} and \ref{section on x, x+Q(y), x+2Q(y), x+R(y), x+2R(y)}, we prove Theorem \ref{equidistribution theorem in the intro} for two general families of progressions for which the theorem is stated. We then show in Section \ref{section on linear forms} how our definition of Leibman group extends the definition of Leibman group for linear forms presented in \cite{green_tao_2010a}. 

Having showed that Leibman nilmanifolds are the right thing to look at, we prove in Section \ref{section on true complexity for equidistributing progressions} that Conjecture \ref{conjecture for true complexity} on the connection between true complexity and algebraic relations holds for all progressions that satisfy a variant of Theorem \ref{equidistribution theorem in the intro}. In Section \ref{section on the asymptotic count}, we deduce Theorem \ref{asymptotic count for systems of complexity 1}. 

Our method, however, has certain limitations which we outline in the subsequent two sections. In Section \ref{section on a progression not satisfying filtration condition}, we give an example of a configuration which does not satisfy the filtration condition. In Section \ref{section on failure of x, x+y, x+2y, x+y^2}, we show that our method fails for $x,\; x+y,\; x+2y,\; x+y^2$, which does not equidistribute on the corresponding Leibman nilmanifold. To handle this progression, we therefore develop a different method in Section \ref{section on x, ..., x+(m-1)y, x+y^d}, which proves true complexity for $x,\; x+y,\; x+2y,\; x+y^2$ as a result of a more general Theorem \ref{True complexity of x, x+y, ..., x+(m-1)y, x+y^d}.

\subsection*{Acknowledgements}
The author would like to thank Sean Prendiville for suggesting the problem, directing to relevant literature and help with editing the paper. The author is also indebted to Tuomas Sahlsten for comments on an earlier version of the paper, and to the anonymous referee for helpful suggestions.

%%%
%%%
%%% HOFA
%%%
%%%

\section{Higher order Fourier analysis}\label{section on higher order Fourier analysis}
To understand operators of the form
\begin{align*}
    \EE_{\textbf{x}\in\FF_p^D}f_1(P_1(\textbf{x}))\cdots f_t(P_t(\textbf{x})),
\end{align*}
we need to understand how certain polynomial sequences distribute on nilmanifolds. We use this section to define necessary concepts from higher order Fourier analysis, such as the notions of filtered nilmanifold and polynomial sequence. All of these definitions have appeared in \cite{green_tao_2008, green_tao_2010a, green_tao_2012, candela_sisask_2012}. We then state some classical results that we shall need in the paper.
\begin{definition}[Filtrations]\label{filtration}
A \emph{filtration $G_\bullet=(G_i)_{i=0}^\infty$ of degree at most} $s$ on a group $G$ is a sequence of subgroups
\begin{align*}
    G=G_0 = G_1 \supseteq G_2 \supseteq ... \supseteq G_s \supseteq G_{s+1} = G_{s+2} =... = 1
\end{align*}
satisfying $[G_i, G_j]\subseteq G_{i+j}$ for all $i,j\geqslant 0$.
\end{definition}

A standard example of a filtration is the \emph{lower central series filtration} defined by setting $G_i = [G,G_{i-1}]$ for $i\geqslant 2.$

\begin{definition}[Filtered nilmanifolds]\label{nilmanifolds}
A \emph{filtered nilmanifold} $G/\Gamma$ of degree $s$ and complexity $M$ consists of the following data:
\begin{enumerate}
    \item a quotient $G/\Gamma$, where $G$ is a connected, simply connected, nilpotent Lie group of dimension $m\leqslant M$ with identity 1, and $\Gamma\subseteq G$ is a cocompact lattice;
    \item a filtration $G_\bullet$ of degree at most $s\leqslant M$ such that $G_i$ are closed and connected subgroups of $G$ and $\Gamma_i:=\Gamma\cap G_i$ is a cocompact lattice in $G_i$ for each $i\in\NN_+$;
    \item an $M$-rational Mal'cev basis $\chi=\{X_1, ..., X_m\}$ adapted to $G_\bullet$.
\end{enumerate}
\end{definition}
Mal'cev basis that appears in Definition \ref{nilmanifolds} is a vector space basis for the Lie algebra $\mathfrak{g}$ of $G$ that respects the filtration $G_\bullet$. Its utility comes from the fact that it provides a natural coordinate system on $G$. The definition and properties of Mal'cev basis are discussed in details in  \cite{green_tao_2012}. The important consequence for us is that it induces a \emph{Mal'cev coordinate map}, i.e. a diffeomorphism 
    \begin{align*}
        \psi: G &\to\RR^m\\
        g &\mapsto (t_1, ..., t_m)
    \end{align*}
satisfying $\psi(\Gamma)=\ZZ^m$ and $\psi(G_i)=\{0\}^{m-m_i}\times\RR^{m_i}$ for all $1\leqslant i\leqslant s$, where $m_i = \dim G_i$.

Nilmanifolds turn out to be a proper framework to define polynomial sequences, which can be thought of as generalizations of polynomials on the torus.
\begin{definition}[Polynomial sequences]\label{polynomial sequences}
Let $D\in\NN_+$. A \emph{polynomial sequence on $G$ adapted to the filtration} $G_\bullet$ is a map $g:\ZZ^D\to G$ satisfying $\partial_{{\textbf{h}_1}, ..., {\textbf{h}_i}}g\in G_i$ for each $\textbf{h}_1, ..., \textbf{h}_i\in\ZZ^D$, where $\partial_{\textbf{h}} g(\textbf{n}):=g(\textbf{n}+\textbf{h})g(\textbf{n})^{-1}$ and $\partial_{\textbf{h}_1, ...,\textbf{h}_i}=\partial_{\textbf{h}_1} ...\partial_{\textbf{h}_i}$ for $i>1$. Polynomial sequences adapted to $G_\bullet$ form a group denoted as $\poly(\ZZ^D,G_\bullet)$. The \emph{degree} of $g$ is the degree of the filtration $G_\bullet$. 
\end{definition}

In this paper, we shall primarily be interested in the polynomial sequences arising in problems over finite fields. These sequences are periodic in the sense made clear by the following definition.
\begin{definition}[Periodic sequences]\label{periodic sequences}
Let $D\in\NN_+$. A sequence $g\in\poly(\ZZ^D,G_\bullet)$ is $p$-periodic if $g(\textbf{n}_1+p \textbf{n}_2)\Gamma=g(\textbf{n}_1)\Gamma$ for all $\textbf{n}_1, \textbf{n}_2\in\ZZ^D$. In particular, for a $p$-periodic sequence ${g\in\poly(\ZZ^D,G_\bullet)}$, the map $\textbf{n}\mapsto g(\textbf{n})\Gamma$ can be viewed as a function from $\FF_p^D$ to $G/\Gamma$.
\end{definition}

It turns out that polynomial sequences can be written in a more explicit manner.
\begin{lemma}[Taylor expansion, Lemma A.1 of \cite{green_tao_2010a}]\label{taylor expansion}
Let $D\in\NN_+$. A sequence $g$ is in $\poly(\ZZ^D,G_\bullet)$ if and only if for each multiindex $\textbf{i}=(i_1, ..., i_D)$, there exists $g_{\textbf{i}}\in G_{|\textbf{i}|}$ satisfying
\begin{align}\label{Taylor expansion as a product}
    g(\textbf{n})=\prod_{\textbf{i}}g_{\textbf{i}}^{{\textbf{n}}\choose{\textbf{i}}}
\end{align}
for all $\textbf{n}\in\ZZ^D$. The representation in (\ref{Taylor expansion as a product}) is unique.
%, where the product in (\ref{equation for Taylor expansion}) is carried out in the order given by $<$. 
The binomial coefficients are defined as
\begin{align*}
    {{\textbf{n}}\choose{\textbf{i}}} := {{n_1}\choose{i_1}}...{{n_D}\choose{i_D}}\quad {\rm{and}}\quad {{n_k}\choose{i_k}}=\frac{n(n-1)...(n-i_k+1)}{i_k!},
\end{align*}
and the size of the vector $\textbf{i}\in\NN^D$ is $|\textbf{i}|:=i_1+...+i_D$.
\end{lemma}

To examine the distribution of $p$-periodic polynomial sequences on nilmanifolds in a quantitative manner, it is useful to introduce the notion of a nilsequence.

\begin{definition}[Nilsequences]\label{nilsequences}
Suppose $D\in\NN_+$. A function $f:\ZZ^D\to\CC$ is a \emph{nilsequence} of {degree} $s$ and {complexity} $M$ if $f(\textbf{n})=F(g(\textbf{n})\Gamma)$, where $F$ is an $M$-Lipschitz function on a filtered nilmanifold $G/\Gamma$ of degree $s$ and complexity $M$, and $g\in\poly(\ZZ^D,G_\bullet)$. A nilsequence $f$ is $p$-periodic if the underlying polynomial sequence is.
\end{definition}

\begin{definition}[Equidistribution]\label{equidistribution}
Let $D\in\NN_+$ and $\delta>0$. A $p$-periodic sequence $g\in\poly(\ZZ^D,G)$ is $\delta$-\emph{equidistributed} on a nilmanifold $G/\Gamma$ if
\begin{align*}
    \left|\EE_{\textbf{n}\in\FF_p^D}F(g(\textbf{n})\Gamma)-\int_{G/\Gamma}F \right|\leqslant\delta\|F\|_\Lip
\end{align*}
for all Lipschitz functions $F:G/\Gamma\to\CC$, where the integral is taken with respect to the (left-invariant) Haar measure $\mu$ on $G/\Gamma$ normalized so that $\mu(G/\Gamma)=1$.
\end{definition}
%The metric $d_{G/\Gamma}$ that appears in Definition \ref{equidistribution} has been constructed in Definition 2.2 of \cite{green_tao_2012}. Its exact definition does not matter for us, as we shall always quote results from \cite{green_tao_2012} that require its explicit construction.

It is natural to ask about obstructions to equidistribution. To state them formally, we need the notion of a horizontal character.
\begin{definition}[Horizontal characters]
A \emph{horizontal character} on $G$ is a continuous group homomorphism $\eta:G\to\RR$ such that $\eta(\Gamma)\in\ZZ$. Each horizontal character can be given in the form $\eta(x) = k\cdot\psi(x)$ for some $k\in\ZZ^m$, where $\psi:G\to\RR^m$ is the Mal'cev coordinate map. The \emph{modulus} of $\eta$ is $|\eta|:=|k|=|k_1|+...+|k_m|$.
\end{definition} 
Horizontal characters in fact annihilate $[G,G]\Gamma$ and can be viewed as maps on the quotient $G/[G,G]\Gamma$ which is isomorphic to $(\RR/\ZZ)^{m_{ab}}$ for $m_{ab}=\dim G-\dim[G,G]$.

In Theorem 1.16 of \cite{green_tao_2012}, Green and Tao gave a condition for when a polynomial sequence is close to being equidistributed. We present its periodic version.
\begin{theorem}[Equidistribution theorem for $p$-periodic sequences]\label{equidistribution theorem}
Let  $2< M\leqslant A$ and ${D\in\NN_+}$, and let $G/\Gamma$ be a filtered nilmanifold of complexity $M$. Suppose that the sequence ${g\in\poly(\ZZ^D, G_\bullet)}$ is $p$-periodic. Then at least one of the following holds:
\begin{enumerate}
    \item $(g(\textbf{n})\Gamma)_{n\in\FF_p^D}$ is $A^{-1}$-equidistributed.
    \item There exists a nontrivial horizontal character $\eta$ with $|\eta|\ll_{M,D} A^{C_{M,D}}$ such that $\eta\circ g$ is constant mod $\ZZ$.
\end{enumerate}
\end{theorem}
In the case of a general polynomial sequence, Theorem 1.16 of \cite{green_tao_2012} only guarantees that if $g$ is not close to being equidistributed, then the coefficients of the polynomial $\eta\circ g - \eta\circ g(0)$ are major arc for a nontrivial horizontal character $\eta$ of small modulus. However, the rigidity imposed by the $p$-periodicity of $g$ allows us to conclude that $\eta\circ g$ is in fact constant mod $\ZZ$. More precisely, Theorem \ref{equidistribution theorem} can be deduced from Theorem 1.16 of \cite{green_tao_2012} as follows: keeping $p$ fixed, we set $N = k p$ in Theorem 1.16 of \cite{green_tao_2012} for some $k\in\NN_+$. If the sequence $g(1), ..., g(N)$ is not $A^{-1}$-equidistributed, then there exists a nontrivial horizontal character $\eta$ with $|\eta|\ll_{M,D} A^{C_{M,D}}$ such that the nonzero coefficients of the polynomial $\eta\circ g(\textbf{n}) = \sum_{\textbf{i}}a_{\textbf{i}}{{\textbf{n}}\choose{\textbf{i}}}$ satisfy the bound $\|a_{\textbf{i}}\|_{\RR/\ZZ}\ll A^{-C_{M,D}} N^{-|\textbf{i}|}$. Importantly, the $p$-periodicity of $g$ implies that we can take the same $A$ for all $k\in\NN_+$; letting $k\to\infty$, we therefore deduce that each nonzero coefficient $a_{\textbf{i}}$ is an integer.

%, letting $k\to\infty$, and observing that the $p$-periodicity of $\eta\circ g$ implies that the smoothness norm $\|\eta\circ g\|_{C^\infty[p]} = \|\eta\circ g\|_{C^\infty[2p]} = \|\eta\circ g\|_{C^\infty[3p]} = ...$ defined in Definition 2.7 of \cite{green_tao_2012} is only bounded from above if it equals 0, i.e. if the nonconstant coefficients of $\eta\circ g$ are all integers. 

We want to define the extent to which a polynomial sequence $g$ is irrational. Essentially, irrationality captures how well the sequence $g$ interacts with objects called $i$-th level characters, or how close to being in $\Gamma$ its Taylor coefficients $g_{\textbf{i}}$ are.
\begin{definition}[$i$-th level character]\label{i-th level characters}
Let $G/\Gamma$ be a filtered nilmanifold.
%, and let $G_i^\nabla$ be the subgroup of $G$ generated by $G_{i+1}$ and $[G_j, G_{i-j}]$ for any $0\leqslant j\leqslant i$. 
An \emph{$i$-th level character} on $G/\Gamma$ is a continuous group homomorphism from $G$ to $\RR$ that is $\ZZ$-valued on $\Gamma$ and vanishes on $G_{i+1}$ and $[G_j, G_{i-j}]$ for any $0\leqslant j\leqslant i$.  It is \emph{nontrivial} if it is nonzero. Every $i$-th level character can be written in the form $\eta(x)=k\cdot\psi_i(x)$ for a unique $k\in\ZZ^{m_i-m_{i+1}}$, where $\psi_i(g_i)$ is a tuple consisting of the entries of $g_i$ in $\psi(g_i)$ indexed by $m-m_i+1, ...,m- m_{i+1}$.
The \emph{modulus} of $\eta$ is defined to be $|\eta|:=|k|=|k_1|+...+|k_{m_i-m_{i+1}}|$.
\end{definition}

\begin{definition}[$A$-irrationality]\label{irrational sequences}
Let $G/\Gamma$ be a filtered nilmanifold of degree $s$. An element $g_i\in G_i$ is \emph{$A$-irrational} if we have $\eta_i(g_i)\notin\ZZ$ for all nontrivial $i$-th level characters $\eta_i$ of complexity $|\eta_i|\leqslant A$. A sequence $g\in\poly(\ZZ^D,G)$ is \emph{$A$-irrational}  if $g_{\textbf{i}}$ is $A$-irrational for each $\textbf{i}\in\NN^D$ with $0<|\textbf{i}|\leqslant s$. 
\end{definition}
Being highly irrational is a stronger property than being close to equidistributed, as implied by the following lemma.
\begin{lemma}[Irrationality implies equidistribution, Lemma 3.7 of \cite{green_tao_2010a}]\label{irrationality implies equidistribution}
Let $D\in\NN_+$. Let $G/\Gamma$ be a filtered nilmanifold of complexity $M$, and suppose that $g\in\poly(\ZZ^D, G_\bullet)$ is $p$-periodic and $A$-irrational. Then $g$ is $O_{M,D}(A^{-c_{M,D}})$-equidistributed.
\end{lemma}

In our arguments, we shall want to approximate sums like (\ref{counting operator for polynomial progressions}) by integrals of Lipschitz functions on nilmanifolds. A key step in doing so is to decompose an arbitrary 1-bounded function into a nilsequence of an appropriate degree and two error terms. We do this via the following lemma, which is a simultaneous and periodic version of the celebrated arithmetic regularity lemma (Theorem 1.2 of \cite{green_tao_2010a}).
\begin{lemma}[Simultaneous periodic irrational arithmetic regularity lemma]\label{regularity lemma}
Let $s\geqslant 2$ and $t\geqslant 1$ be integers, $\epsilon>0$, and let $\mathcal{F}:\RR_+\to\RR_+$ be a growth function. There exists $M=O_{s,t,\epsilon,\mathcal{F}}(1)$ with the property that for all 1-bounded functions $f_1, ..., f_t:\FF_p\to\CC$ there exist decompositions
\begin{align*}
    f_i = f_{i, nil} + f_{i,sml} + f_{i,unf}
\end{align*}
such that for each $1\leqslant i\leqslant t$, the functions $f_{i, nil},  f_{i,sml}, f_{i,unf}$ satisfy the following:
\begin{enumerate}
    \item $f_{i,nil}(n)=F_i(g(n)\Gamma)$ for a $M$-Lipschitz function $F_i: G/\Gamma\to\CC$, where $G/\Gamma$ is a filtered nilmanifold of degree $s$ and complexity at most $M$, and $g\in\poly(\ZZ,G_\bullet)$ is a $p$-periodic, $\mathcal{F}(M)$-irrational sequence satisfying $g(0)=1$;
    \item $\|f_{i,sml}\|_2\leqslant \epsilon$;
    \item $\|f_{i,unf}\|_{U^{s+1}}\leqslant \frac{1}{\mathcal{F}(M)}$;
    \item the functions $f_{i,nil}$, $f_{i,sml}$ and $f_{i,unf}$ are 4-bounded.
    %\item if $f_i$ takes values in $[0,1]$, then so do $f_{i,nil}$ and $f_{i,nil}+f_{i,sml}$.
\end{enumerate}
\end{lemma}
The important thing about Lemma \ref{regularity lemma} is that we decompose each $f_1, ..., f_t$ with respect to the same sequence $g$ and the same nilmanifold $G/\Gamma$. We give its proof in Appendix \ref{section on ARL}.

Finally, we shall briefly state the relevant properties of Gowers norms and polynomial bias norms, all of which are discussed more extensively in \cite{green_2007} and \cite{green_tao_2008}. While Gowers norms are the central objects of study in this paper, polynomial bias norms only appear in Theorem \ref{True complexity of x, x+y, x+2y, x+y^2}.
\begin{definition}[Polynomial bias norms]\label{polynomial bias norms}
For $s\in\NN_+$, the \emph{polynomial bias norm} of degree $s$ of a function $f:\FF_p\to\CC$ is given by
\begin{align*}
    \|f\|_{u^s}=\max_{\alpha_1, ..., \alpha_s\in\FF_p}\left|\EE_{x\in\FF_p}f(x)e_p \left(\alpha_{s-1} x^{s-1}+ ... + \alpha_1 x\right)\right|.
\end{align*}
\end{definition}
Gowers norms and polynomial bias norms are seminorms for $s=1$ and genuine norms for $s\geqslant 2$. They satisfy the monotonicity property
\begin{align*}
    \|f\|_{U^1} &\leqslant \|f\|_{U^2}\leqslant \|f\|_{U^3}\leqslant ...\\
    \|f\|_{u^1} &\leqslant \|f\|_{u^2}\leqslant \|f\|_{u^3}\leqslant ...
\end{align*}
for any $f:\FF_p\to\CC$. Gowers norms also bound polynomial bias norms in that
$$\|f\|_{u^s}\leqslant \|f\|_{U^s}.$$
For $s=1$, we in fact have $\|f\|_{u^1}=\|f\|_{U^1}$, and for $s=2$, we have $\|f\|_{u^2}\leqslant \|f\|_{U^2}\leqslant\|f\|_{u^2}^\frac{1}{2}$ for all 1-bounded $f:\FF_p\to\CC$, a result known as \emph{$U^2$ inverse theorem}. For $s>2$, however, there exist functions that have large $U^s$ norms and small $u^s$ norms. This is due to the fact that having a large $u^s$ norm is equivalent to correlating with a polynomial phase of degree $s-1$, while having a large $U^s$ norm corresponds to correlating with a broader category of nilsequences of degree $s-1$. See \cite{green_tao_2008, green_tao_ziegler_2011, green_tao_ziegler_2012} for more details on the relationship between Gowers norms and nilsequences.

%%%
%%% Leibman nilmanifold
%%%

\section{Leibman nilmanifold for polynomial progressions}\label{section on Leibman nilmanifold}
Throughout this section, let $\vec{P}(\textbf{x})=(P_1(\textbf{x}), ...,P_t(\textbf{x}))\in\QQ[\textbf{x}]^t$ be an integral polynomial map of degree $d$. Suppose that $G/\Gamma$ is a filtered nilmanifold  of degree $s$, dimension $m$ and complexity $M$. Given a polynomial sequence $g$ adapted to a filtration $G_\bullet$ on $G/\Gamma$, we define 
\begin{align*}
    g^P(\textbf{x}):=(g(P_1(\textbf{x})), ..., g(P_t(\textbf{x}))).
\end{align*}
The main objective of this section is to construct a group $G^P$ with a filtration $G^P_\bullet$ such that $g^P\in\poly(\ZZ^D,G^P_\bullet)$ whenever $g\in\poly(\ZZ,G_\bullet)$. The group $G^P$ originally appeared in Section 5 of \cite{leibman_2007}, but we are not aware of the filtration $G^P_\bullet$ being defined previously except for the relatively simple case of linear forms in \cite{green_tao_2010a}.

%In Section \ref{section on linear forms}, we show that this construction generalizes the notion of Leibman group for a system of linear forms (Definition 1.10 of \cite{green_tao_2010a}); therefore we shall call $G^P$ the \emph{Leibman group} for $\vec{P}$.

 It is a standard fact that each integral polynomial map $\vec{Q}\in\QQ[\textbf{x}]^t$ can be expressed as
\begin{align*}
    \vec{Q}(\textbf{x})=\sum_{\textbf{i}}\vec{b}_{\textbf{i}} {{\textbf{x}}\choose{\textbf{i}}}, 
\end{align*}
for some $\vec{b}_{\textbf{i}}\in\ZZ^t$, and we denote its \emph{degree-$j$ part} by
\begin{align*}
    \D_j \vec{Q}(\textbf{x}):=\sum_{|\textbf{i}|=j}\vec{b}_{\textbf{i}} {{\textbf{x}}\choose{\textbf{i}}}.
\end{align*}
Thus, for instance, $\D_1\left(x + y + {{x}\choose{2}}\right) = x+y$ and $\D_2\left(x + y + {{x}\choose{2}}\right) = {{x}\choose{2}}$. 
To clarify the notation, we denote
\begin{align*}
    {{\textbf{x}}\choose{\textbf{i}}}={{x_1}\choose{i_1}}\cdots {{x_D}\choose{i_D}} \quad {\rm{and}} \quad {{\vec{v}}\choose{\vec{j}}}={{{v}_1}\choose{{j}_1}} \cdots {{{v}_t}\choose{{i}_t}}
\end{align*}
for $D$-dimensional vectors $\textbf{x}$, $\textbf{i}$ and $t$-dimensional vectors $\vec{v}$, $\vec{j}$. If $i$ is a scalar, however, we set
\begin{align*}
    {{\vec{v}}\choose{{i}}}=\left({{{v}_1}\choose{{i}}}, ..., {{{v}_t}\choose{{i}}}\right).
\end{align*}
We endow $\RR^t$ with the structure of a real algebra by letting $$(a_1, ..., a_t)\cdot(b_1, ..., b_t) := (a_1 b_1, ..., a_t b_t)$$ and setting $\vec{1}=(1, ...,1)$ to be the identity vector.

For $i,j\in\NN_+$, we define two families of real vector spaces
\begin{align*}
%\label{definition of P_i,j}
    \PP_{i,j} &:=\Span\{\D_k{{\vec{P}(\textbf{x})}\choose{l}}: k\geqslant j,\; 1\leqslant l\leqslant i,\; \textbf{x}\in\ZZ^D\} \\
    \Q_{i,j} &:= \sum_{\substack{k, i_1, ..., i_k, j_1, ..., j_k\in\NN_+,\\ i_1+...+i_k=i, j_1+...+j_k=j}}\PP_{i_1,j_1}\cdots \PP_{i_k,j_k} =  \PP_{i,j} + \sum_{\substack{i_1+i_2=i, \\ j_1+j_2=j}}\Q_{i_1,j_1}\cdot\Q_{i_2,j_2}.
\end{align*}
We note several facts about these subspaces.
\begin{lemma}\label{properties of polynomial spaces}
For each $i, i_1, i_2, j, j_1, j_2\in\NN_+$, we have the following inclusions:
\begin{enumerate}
    \item $\PP_{i,j}\subseteq\PP_{i+1,j}$;
    \item $\PP_{i,j+1}\subseteq\PP_{i,j}$;
    \item $\PP_{i,j}\subseteq\Q_{i,j}$;
    \item $\Q_{i,j}\subseteq\Q_{i+1,j}$;
    \item $\Q_{i,j+1}\subseteq\Q_{i,j}$;
    \item $\Q_{i_1, j_1}\cdot\Q_{i_2,j_2}\subseteq\Q_{i_1+i_2,j_1+j_2}$.
\end{enumerate}
\end{lemma}
\begin{proof}
Statements (i), (ii), (iii), (vi) follow directly from the definitions. Statements (iv) and (v) follow from properties (i), (ii), (iii) by induction.
\end{proof}

We define groups $G^P_j$ for $j\geqslant 1$ by setting
\begin{align*}
    G^P_j &:= \langle g^{\Vec{v}}: \Vec{v}\in\Q_{i,j},\; g\in G_i,\; i\geqslant 1\rangle
\end{align*}
where we let $g^{\Vec{v}}:=(g^{{v}_1}, ..., g^{{v}_t})$. We moreover set $G^P=G^P_0=G^P_1$ and ${\Gamma^P=\Gamma^t \cap G^P}$. It follows from property $(v)$ of Lemma \ref{properties of polynomial spaces} that 
\begin{align*}
    G^P = G_0^P = G_1^P \supseteq G^P_2 \supseteq G^P_3 \supseteq ...
\end{align*}
Each of the groups $G^P_j$ is normal in $G^t$ because each $G_i$ is normal in $G$.

For instance, if $G_\bullet$ is a filtration of degree 2 and $\vec{P}(x,y) = (x,\; x+y,\; x+2y)$ is the 3-term arithmetic progression, then $\P_{1,1} = \Q_{1,1} = \Span\{(1,1,1), (0,1,2)\}$, $\P_{2,1}=\P_{2,2}=\Q_{2,1}=\Q_{2,2} = \RR^3$ and $\Q_{1,2} = \Q_{2,3} = \{\vec{0}\}$; thus $G^P_\bullet$ is given by
\begin{align*}
    G^P_1 &= \langle g_1^{(1,1,1)},\; g_1^{(0,1,2)},\; g_2^{(0,0,1)}: g_1\in G_1,\; g_2\in G_2\rangle\\
    G^P_2 &= \langle g_2^{(1,1,1)},\; g_2^{(0,1,2)},\; g_2^{(0,0,1)}: g_2\in G_2\rangle\\
    G^P_3 &= G^P_4 = ... = 1.
\end{align*}

\begin{lemma}\label{Filtration on Leibman nilmanifold}
The chain of subgroups $(G^P_j)_{j=0}^\infty$ defines a filtration on $G^P$ of degree $sd$.
\end{lemma}
\begin{proof}
Take generators $g_{1}^{\Vec{v}_{1}}\in G^P_{j_1}$ and $g_{2}^{\Vec{v}_{2}}\in G^P_{j_2}$ for some elements $g_1\in G_1$ and $g_2\in G_2$ as well as vectors $\Vec{v}_{1} \in\Q_{i_1, j_1}$ and $\Vec{v}_{2} \in\Q_{i_2, j_2}$. We want to show that their commutator is  in $G^P_{j_1+j_2}$. If this is true for the generators, then by Lemma 7.3 of \cite{green_tao_2012} it holds for arbitrary two elements of $G_{j_1}$ and $G_{j_2}$, proving the lemma.

By (\ref{Baker-Campbell-Hausdorff 2}), we have
\begin{align*}
    [g_{1}^{\Vec{v}_{1}}, g_{2}^{\Vec{v}_{2}}]=[g_1, g_2]^{\Vec{v}_1 \cdot \Vec{v}_2}\prod_\alpha g_\alpha^{Q_\alpha(\Vec{v}_1, \Vec{v}_2)}
\end{align*}
where each $g_\alpha$ is a commutator of $k_1$ copies of $g_1$ and $k_2$ copies of $g_2$ for some ${k_1, k_2\geqslant 1}$ satisfying $k_1+k_2\geqslant 3$, and $$Q_\alpha(\Vec{u}, \Vec{w}):=(Q_\alpha(u_1,w_1), ..., Q_{\alpha}(u_t, w_t))$$ for some polynomial $Q_\alpha(x_1, x_2)$ that has degree  at most $k_1$ in $x_1$ and at most $ k_2$ in $x_2$, and vanishes when $x_1 = 0$ or $x_2 = 0$.

By the filtration property of $G$, the commutator $[g_1, g_2]$ is in $G_{i_1+i_2}$. We moreover have that
\begin{align*}
    \Vec{v}_1\cdot\Vec{v}_2\in \Q_{i_1, j_1}\cdot\Q_{i_2, j_2}\subseteq\Q_{i_1+i_2,j_1+j_2}
\end{align*}
by part (vi) of Lemma \ref{properties of polynomial spaces}. Consequently, the element $[g_1, g_2]^{\Vec{v}_1 \cdot \Vec{v}_2}$ is contained in $G^P_{j_1+j_2}$.

We handle the terms $g_\alpha^{Q_\alpha(\Vec{v}_1, \Vec{v}_2)}$ in a similar manner. If $g_\alpha$ is a commutator of $k_1$ copies of $g_1$ and $k_2$ copies of $g_2$, then $g_\alpha\in G_{k_1 i_1 + k_2 i_2}$ by the filtration property of $G$. The polynomial $Q_\alpha$ can then be written as $Q_\alpha(x_1, x_2) = \sum_{\substack {1\leqslant l_1\leqslant k_1,\\ 1\leqslant l_2\leqslant k_2}}\beta_{l_1, l_2}x_1^{l_1} x_2^{l_2}$, and so
\begin{align*}
    Q_\alpha(\Vec{v}_1,\Vec{v}_2) &= \sum_{\substack {1\leqslant l_1\leqslant k_1,\\ 1\leqslant l_2\leqslant k_2}}\beta_{l_1, l_2} (\Vec{v}_1)^{l_1} \cdot(\Vec{v}_2)^{l_2}\in \sum_{\substack {1\leqslant l_1\leqslant k_1,\\ 1\leqslant l_2\leqslant k_2}}\Q_{l_1 i_1 + l_2 i_2, l_1 j_1 + l_2 j_2}\\
    &\subseteq \sum_{\substack {1\leqslant l_1\leqslant k_1,\\ 1\leqslant l_2\leqslant k_2}}\Q_{k_1 i_1 + k_2 i_2, l_1 j_1 + l_2 j_2}\subseteq \Q_{k_1 i_1 + k_2 i_2, j_1+j_2}
\end{align*}
by Lemma \ref{properties of polynomial spaces}. Thus $g_\alpha^{Q_\alpha(\Vec{v}_1, \Vec{v}_2)}$  is contained in $G^P_{j_1+ j_2}$ for each $\alpha$ which implies that
$[g_{1}^{\Vec{v}_{1}}, g_{2}^{\Vec{v}_{2}}]\in G^P_{j_1+j_2}$.
\end{proof}

We now aim to prove the topological properties of $G^P$ and $G^P_j$, and we do this by following the arguments presented in \cite{green_tao_2010a} after Lemma 3.5. Our first goal is to show that $G^P$ is a connected, simply connected Lie group. By Lemma \ref{properties of polynomial spaces}, we have a chain of subspaces
\begin{align}\label{chain of subspaces}
    0\subseteq\Q_{1,1}\subseteq\Q_{2,1}\subseteq ...\subseteq \Q_{s,1} \subseteq \RR^t
\end{align}
that can be defined over $\QQ$. Letting $t_i = \dim\Q_{i,1}$, we can find a basis $\Vec{v}_1, ..., \Vec{v}_{t_s}$ for $\Q_{s,1}$ satisfying the following properties:
\begin{enumerate}
    \item (Integrality) $\Vec{v}_1, ..., \Vec{v}_{t_s}$ are all integer-valued,
    \item (Partial span) $\Vec{v}_1, ..., \Vec{v}_{t_i}$ span $\Q_{i,1}$ for each $1\leqslant i\leqslant s$,
    \item (Row echelon form) For each $1\leqslant k\leqslant t_s$ there exists an index $1\leqslant r_k\leqslant t$ such that $\Vec{v}_k(r_k)\neq 0$ but $\Vec{v}_l(r_k) = 0$ for all $k < l\leqslant t_s$.
\end{enumerate}
Fixing such a basis, we let $\deg(\Vec{v}_k)$ to be the smallest $i$ such that $\Vec{v}_k$ is in $\Q_{i,1}$.
%but not in $\Q_{i-1,1}$
We can express each element of $G^P$ as a finite product of $g_k^{\Vec{v}_k}$ where $g_k\in G_{\deg(\Vec{v}_k)}$ and $1\leqslant k\leqslant t_s$. By applying the corollaries (\ref{Baker-Campbell-Hausdorff 1}) and (\ref{Baker-Campbell-Hausdorff 2}) to the Baker-Campbell-Hausdorff formula many times, we can then rewrite an arbitrary element of $G^P$ as 
\begin{align}\label{Taylor representation of Leibman nilmanifold}
    \prod_{k=1}^{t_s}g_k^{\Vec{v}_k},
\end{align}
where $g_k\in G_{\deg\Vec{v}_k}$ for all $1\leqslant k\leqslant t_s$. This representation is unique, implying that $G^P$ is indeed a connected, simply connected Lie subgroup. From the fact that each $\vec{v}_k$ has integer entries, it can be further deduced that $\Gamma^P$ is cocompact in $G^P$.
Similar arguments, where we replace $\Q_{i,1}$ in (\ref{chain of subspaces}) by $\Q_{i,j}$, show that each $G^P_j$ is a closed connected subgroup of $G^P$ and $\Gamma^P_j = \Gamma^t\cap G^P_j$ is cocompact in $G^P_j$. This implies that $G^P/\Gamma^P$ is a filtered nilmanifold. Finally, the same argument combined with the fact that $\PP_{i,j}=\Q_{i,j}=0$ whenever $j>i d$ shows that $G^P$ is a subnilmanifold of $G^t$ when we endow $G^t$ with the filtration $(G^t)'_j = (G_{\lceil\frac{j}{d}\rceil})^t$.

%and that $G^P_i$ are rational subgroups of $G^t$, implying that $G^P/\Gamma^P$ is a subnilmanifold of $G^t/\Gamma^t$.
The next lemma explains why we have imposed this particular filtration on $G^P$.
\begin{lemma}\label{g^P is adapted to the filtration}
If $g\in\poly(\ZZ,G_\bullet)$, then $g^P\in\poly(\ZZ^D,G^P_\bullet)$. 
\end{lemma}
\begin{proof}
We first decompose
\begin{align*}
    {{\vec{P}(\textbf{x})}\choose{i}}=\sum_{\textbf{j}}\vec{v}_{i, \textbf{j}}{{\textbf{x}}\choose{\textbf{j}}}.
\end{align*} and note that $\vec{v}_{i,\textbf{j}}\in\PP_{i, |\textbf{j}|}$. Therefore $g_i^{\vec{v}_{i,\textbf{j}}}\in G^P_{|\textbf{j}|}$ by definition of $G_{|\textbf{j}|}^P$. Using (\ref{Baker-Campbell-Hausdorff 1}) and (\ref{Baker-Campbell-Hausdorff 2}), we regroup the terms of 
\begin{align*}
    g^P(\textbf{x}) = \prod_{i=0}^s g_i^{{\vec{P}(\textbf{x})}\choose{i}} = \prod_{i=0}^s g_i^{\sum_{\textbf{j}}\vec{v}_{i, \textbf{j}}{{\textbf{x}}\choose{\textbf{j}}}}
\end{align*}
to bring all the elements involving the same monomial ${{\textbf{x}}\choose{\textbf{j}}}$ together. Thus,
the Taylor coefficient of ${{\textbf{x}}\choose{\textbf{j}}}$ is of the form
\begin{align}\label{Taylor coefficient of a monomial}
    \prod_{i=0}^s g_i^{\sum_{\textbf{j}}\vec{v}_{i, \textbf{j}}} \prod_\alpha g_\alpha^{\vec{v}_{\alpha}},
\end{align}
where the terms $g_\alpha^{\vec{v}_{\alpha}}$ come from applying (\ref{Baker-Campbell-Hausdorff 1}) and (\ref{Baker-Campbell-Hausdorff 2}).
For each label $\alpha$, the element $g_\alpha$ is a commutator consisting of $k_r$ copies of $g_{i_r}$ for $1\leqslant r\leqslant n$, and so $g_\alpha\in G_{i_1 k_1 + ... + i_n k_n}$ by the filtration property of $G$. The vector $\vec{v}_\alpha$ is a rational multiple of ${\vec{v}_{i_1, |\textbf{j}_1|}^{l_1}} ...{\vec{v}_{i_n, |\textbf{j}_n|}^{l_n}}$  for some $1\leqslant l_1\leqslant k_1$, ..., $1\leqslant l_n\leqslant k_n$ and $|\textbf{j}_1|+ ... + |\textbf{j}_n|\geqslant|\textbf{j}|$. Therefore $\vec{v}_\alpha\in\Q_{i_1 l_1 + ... + i_n l_n, |\textbf{j}_1|+ ... + |\textbf{j}_n|}$ by part (vi) of Lemma \ref{properties of polynomial spaces}. It follows from parts (iv) and (v) of the same lemma that
\begin{align*}
    \Q_{i_1 l_1 + ... + i_n l_n, |\textbf{j}_1|+ ... + |\textbf{j}_n|}\subseteq\Q_{i_1 l_1 + ... + i_n l_n, |\textbf{j}|}\subseteq\Q_{i_1 k_1 + ... + i_n k_n, |\textbf{j}|},
\end{align*}
implying that $g_\alpha^{\vec{v}_\alpha}\in G^P_{|\textbf{j}|}$ for each $\alpha$. Thus, the coefficient (\ref{Taylor coefficient of a monomial}) is in $G^P_{|\textbf{j}|}$, as claimed. 

%$g_1, ..., g_s$ appearing $k_1, ..., k_s$ times in the commutator. 
%The vector $\vec{v}_{\alpha}$ is of the form $a_\alpha \vec{v}_{1, |\textbf{j}_1|}$

%of $g_{i_1}$, ..., $g_{i_l}$ for some $i_1, ..., i_l\in\{1, ..., s\}$, and $\vec{v}_{\alpha}= Q_\alpha(\vec{v}_{i_1, |\textbf{j}_1|}, ..., \vec{v}_{i_l,|\textbf{j}_l|})$, where $Q(t_1, ..., t_l)\in\QQ[t_1, ..., t_l]$ has degree at most $i_k$ in each $t_k$ and moreover satisfies $Q(t_1, ..., t_l)=0$ whenever one of $t_k$'s is 0. 

%We observe from (\ref{Baker-Campbell-Hausdorff 1}) and (\ref{Baker-Campbell-Hausdorff 2})

%This follows from Lemma \ref{Filtration on Leibman nilmanifold}, the way we defined $\Q_{i,j}$, and many applications of corollaries (\ref{Baker-Campbell-Hausdorff 1}) and (\ref{Baker-Campbell-Hausdorff 2}) to Baker-Campbell-Hausdorff formula.
\end{proof}
Even though the definition of $\Q_{i,j}$ guarantees that $G^P$ is filtration, it is not very handy to work with because of the terms $\Q_{i_1,j_1}\cdot\Q_{i_2,j_2}$ that appear there. However, many of the configurations that we look at satisfy a condition that allows us to work with $\PP_{i,j}$ instead.
\begin{definition}[Filtration condition]\label{filtration condition}
We say that an integral polynomial ${\vec{P}\in\QQ[\textbf{x}]^t}$ satisfies the \emph{filtration condition} if $\PP_{i_1, j_1}\cdot\PP_{i_2, j_2}\subseteq \PP_{i_1+i_2, j_1 + j_2}$ for all $i_1, i_2, j_1, j_2\geqslant 1$.
\end{definition}
\begin{lemma}\label{P_i,j equals Q_i,j}
Suppose an integral polynomial ${\vec{P}\in\QQ[\textbf{x}]^t}$ satisfies the \emph{filtration condition}. Then $\PP_{i,j}=\Q_{i,j}$ for all $i,j\in\NN_+$.
\end{lemma}
\begin{proof}
Whenever $i=1$ or $j=1$, the lemma follows by definition. Other cases follow by induction on $(i,j)$.
\end{proof}

For progressions satisfying the filtration condition, we can relate the group $G^P_j$ to $G^P_{j+1}$ in a handy manner.
\begin{lemma}\label{structure of G^P_j}
Let $j\in\NN_+$, and suppose $\vec{P}$ satisfies the filtration condition. For any $i\in\NN_+$, let $X_{i,j+1}=\{\vec{v}_1, ..., \vec{v}_{l_i}\}\subseteq\ZZ^t$ be a basis for $\PP_{i,j+1}$ that extends to a basis $X_{i,j}=\{\vec{v}_1, ..., \vec{v}_{l_i}, ..., \vec{v}_{l_i+k_i}\}\subseteq\ZZ^t$ for $\PP_{i,j}$. Then
\begin{align*}
    G^P_j = \langle G^P_{j+1},\; g^{\vec{v}_r}: l_i+1 \leqslant r\leqslant l_i+k_i,\; g\in G_i,\; i\in\NN_+ \rangle.
    %G^P_j = \langle G^P_{j+1},\; g^{\vec{v}}: \vec{v}\in\Span\{\vec{v}_{l_i+1}, ..., \vec{v}_{l_i+k_i}\},\; g\in G_i,\; i\in\NN_+ \rangle.
\end{align*}
\end{lemma}
\begin{proof}
Since $\vec{P}$ satisfies the filtration condition, we have by Lemma \ref{P_i,j equals Q_i,j} that $\PP_{i,j}=\Q_{i,j}$. By definition of $G^P_j$, the group is generated by elements of the form $g^{\vec{v}}$ for $i\in\NN_+$, $g\in G_i$, $\vec{v}=\sum\limits_{r=1}^{l_i+k_i}a_r\vec{v}_r$ and $a_r\in\RR$. Letting $\vec{w}=\sum\limits_{r=1}^{l_i}a_r\vec{v}_r$, we observe that $g^{\vec{w}}\in G^P_{j+1}$, and so
\begin{align*}
    g^{\vec{v}}=g^{\vec{v}-\vec{w}} = (g^{a_{l_i+1}})^{\vec{v}_{l_i + 1}}... (g^{a_{l_i+k_i}})^{\vec{v}_{l_i + k_i}} \mod\; G^P_{j+1}.
\end{align*}
Thus each generator $g^{\vec{v}}$ in $G^P_{j}$ is a product of an element from $G^P_{j+1}$ and $h_r^{\vec{v}_r}$ for $l_i+1 \leqslant r\leqslant l_i+k_i$ and $h_r=g^{a_r}\in G_i$.
\end{proof}

\subsection{Progressions of a special form}
The technical results presented so far in this section work for arbitrary integral polynomial maps. However, we shall mostly work with polynomial maps of the form
\begin{align}\label{a special form of the integral polynomial map}
    \vec{P}(x,y) = (x,\; x+P_1(y),\; ...,\; x+P_{t-1}(y)).
\end{align}

For configurations of this type, we can relate the coefficients of ${{\vec{P}(x,y)}\choose{i}}$ to the coefficients of ${{\vec{P}(x,y)}\choose{i-k}}$ for $k>0$ in a way that will prove useful in future sections.
\begin{lemma}\label{relating coefficients of different monomials}
Let $i, k,l\in\NN$ satisfy $i>0$, $k\leqslant i$ and $l\leqslant (i-k)d$. Suppose $\vec{P}$ is an integral polynomial map of the form (\ref{a special form of the integral polynomial map}). Then the coefficient of ${{x}\choose{k}}{{y}\choose{l}}$ in ${{\vec{P}(x,y)}\choose{i}}$ is the same as the coefficient of ${{y}\choose{l}}$ in ${{\vec{P}(x,y)}\choose{i-k}}$.
\end{lemma}
\begin{proof}
If $\vec{P}$ is of the form (\ref{a special form of the integral polynomial map}), then $\vec{P}(x,y)=\vec{1}x + \vec{P}(0,y)$. Using the property
\begin{align}\label{identity for binomial coefficients}
    {{a_1+...+a_l}\choose{i}}=\sum_{\substack{0\leqslant k_1, ..., k_l\leqslant i,\\ k_1+...+k_l = i}}{{a_1}\choose{k_1}}...{{a_l}\choose{k_l}},
    %{{a}\choose{n}}+{{a}\choose{n-1}}b + {{a}\choose{n-2}}{{b}\choose{2}}+...{{b}\choose{n}}
\end{align}
which can be proved by looking at two ways in which one picks $i$ elements from a union of disjoint sets of size $a_1$, ..., $a_l$ respectively, we can rewrite
\begin{align*}
    {{\vec{P}(x,y)}\choose{i}}={{\vec{1}x+\vec{P}(0,y)}\choose{i}}=  \sum_{n=0}^i{{x}\choose{i-n}}{{\vec{P}(0,y)}\choose{n}}=\sum_{n=0}^i\sum_{l=1}^{n}{{x}\choose{i-n}}{{y}\choose{l}}\vec{v}_{l,n},
\end{align*}
for some $\vec{v}_{l,n}\in\ZZ^t$. Replacing $i$ by $i-k$, we obtain
\begin{align*}
    {{\vec{P}(x,y)}\choose{i-k}}=\sum_{n=0}^{i-k}\sum_{l=1}^{nd}{{x}\choose{i-k-n}}{{y}\choose{l}}\vec{v}_{l,n}.
\end{align*}
The lemma follows by taking $n=i-k$ in both cases and fixing $1\leqslant l\leqslant (i-k)d$.
\end{proof}

\begin{lemma}\label{inclusion in polynomial spaces of special form}
Suppose $\vec{P}$ is an integral polynomial map of the form (\ref{a special form of the integral polynomial map}). Then $\PP_{i,j}\subseteq\PP_{i+1,j+1}$ for all $i,j\in\NN_+$.
\end{lemma}
\begin{proof}
By Lemma \ref{relating coefficients of different monomials}, the coefficient $\vec{v}$ of ${{x}\choose{k}}{{y}\choose{l}}$ in ${{\vec{P}(x,y)}\choose{i}}$ is the same as the coefficient of ${{x}\choose{k+1}}{{y}\choose{l}}$ in ${{\vec{P}(x,y)}\choose{i+1}}$ whenever $l\leqslant (i-k)d$. If $k+l\geqslant j$, then $k+1+l\geqslant j+1$, and so if $\vec{v}\in\PP_{i,j}$, then $\vec{v}\in\PP_{i+1,j+1}$.
\end{proof}

We remark that the property $\PP_{i,j}\subset\PP_{i+1,j+1}$ for all $i,j\in\NN_+$ also holds for an arbitrary integral polynomial map $\vec{P}$ that satisfies the filtration condition and $\vec{1}\in\PP_{1,1}$. However, the advantage of assuming (\ref{a special form of the integral polynomial map}) comes from the fact that we do not need $\vec{P}$ to satisfy the filtration condition in this case.

Lemma \ref{inclusion in polynomial spaces of special form} has one important corollary that we shall use several times.
\begin{corollary}\label{horizontal characters reduce to i-th level characters}
Let $\vec{P}$ be of the form (\ref{a special form of the integral polynomial map}), $i,j\in\NN_+$, $\vec{v}\in\PP_{i,j}\cap\ZZ^t$, and let $\eta: G^P\to\RR$ be a horizontal character that vanishes on $G_{j+1}^P$. Then the map
\begin{align*}
    \xi: G_{i}&\to\RR\\
    g &\mapsto \eta(g^{\vec{v}})
\end{align*}
is an $i$-th level character.
\end{corollary}
\begin{proof}
It is straightforward to see that $\xi$ is a group homomorphism. Since $\vec{v}$ has integer entries, the element $g^{\vec{v}}\in\Gamma^P$ for any $g\in\Gamma_i$, and so $\xi(\Gamma_i)\in\ZZ$. It vanishes on $[G,G]\cap G_i$ because its codomain is abelian, and to show that it is an $i$-th level character, it remains to show that $\xi|_{G_{i+1}}=0$. Suppose $g\in G_{i+1}$. From Lemma \ref{inclusion in polynomial spaces of special form}, we have that $\vec{v}\in\PP_{i+1,j+1}$, and so $g^{\vec{v}}\in G_{j+1}^P$. It follows that $\xi(g^{\vec{v}})=0$ from the fact that $\eta$ vanishes on $G_{j+1}^P$. 
\end{proof}

%%%
%%%
%%% Counting lemma for $x, x+y, x+y^2, x+y+y^2$
%%%
%%%

\section{An equidistribution result for $x,\; x+y,\; x+y^2,\; x+y+y^2$}\label{section on x, x+y, x+y^2, x+y+y^2}
The goal of the next few sections is to prove Theorem \ref{equidistribution theorem in the intro} for various configurations for which the theorem holds. We give them a name, so that it is easier to refer to them in later sections.
\begin{definition}[Equidistributing progressions]\label{progressions that equidistribute}
Let $\vec{P}=(P_1, ..., P_t)\in\QQ[\textbf{x}]^t$ be an integral polynomial map. We say that $\vec{P}$ \emph{equidistributes} if for each ${s, D\in\NN_+}$, $M>0$, a filtered nilmanifold $G/\Gamma$  of degree $s$ and complexity $M$, and a $p$-periodic, $A$-irrational sequence $g\in\poly(\ZZ,G_\bullet)$ satisfying $g(0)=1$, the sequence $g^P\in\poly(\ZZ^D,G^P_\bullet)$ is $o_{A\to\infty, M}(1)$-equidistributed on $G^P/\Gamma^P$.
\end{definition}

%These equidistribution results derived in this and next sections combine to give Theorem \ref{equidistribution theorem in the intro}. In Section \ref{section on true complexity for equidistributing progressions}, we then prove true complexity for all equidistributing progressions, and in Section \ref{section on the asymptotic count}, we derive asymptotics for the number of certain equidistributing progressions in subsets of $\FF_p$.

%To illustrate more general arguments in Sections \ref{section on x, x+Q(y), x+R(y), x+Q(y)+R(y)} and \ref{section on x, x+Q(y), x+2Q(y), x+R(y), x+2R(y)}, 
We start with a seemingly simple example
\begin{align}\label{presentation of x, x+y, x+y^2, x+y+y^2}
    \vec{P}(x,y) &= (x,\; x+y,\; x+y^2,\; x+y+y^2)\\
    \nonumber
    &= (x,\; x+y,\; x+y+2{{y}\choose{2}},\; x+2y+2{{y}\choose{2}}).
\end{align}
It is a special case of the progression discussed in Section \ref{section on x, x+Q(y), x+R(y), x+Q(y)+R(y)}. However, we analyze it separately to give a  concrete example of how our general argument works.

We start by obtaining the formulas for $\PP_{i,j}$ for various values of $i,j\in\NN_+$. From (\ref{presentation of x, x+y, x+y^2, x+y+y^2}), we deduce that
\begin{align*}
    \PP_{1,1}&=\Span\{(1,1,1,1), (0,1,1,2), (0,0,1,1)\}=\Span\{\vec{v}_1, \vec{v}_2, \vec{v}_3\}\\
    \PP_{1,2}&=\Span\{(0,0,1,1)\} = \Span\{\vec{v}_3\}\\
    \PP_{1,j}&= 0 \quad {\rm{for}} \quad j\geqslant 3.
\end{align*}
where
\begin{align*}
    \vec{v}_1 = (1,1,1,1), \quad \vec{v}_2 = (0,1,1,2), \quad \vec{v}_3 = (0,0,1,1)\quad {\rm{and}} \quad \vec{v}_4 = (0,0,0,1).
\end{align*}
Our next goal is to deduce expressions for $\PP_{i,j}$ whenever $i>1$.
\begin{lemma}\label{structure of P for x, x+y, x+y^2, x+y+y^2}
For $i>1$, the following holds:
\[ \PP_{i,j}=\begin{cases}
\RR^4,\; &1\leqslant j\leqslant i\\
\Span\{\vec{v}_3, \vec{v}_4\},\; &i+1\leqslant j\leqslant 2i-1\\
\Span\{\vec{v}_3\},\; &j=2i\\
0,\; &j>2i.
\end{cases}
\]
\end{lemma}
\begin{proof}
The case $j>2i$ follows from the fact that the polynomial map ${{\vec{P}(x,y)}\choose {i}}$ has degree $2i$. For the case $j=2i$, note that the only monomial of ${{\vec{P}(x,y)}\choose {i}}$ of degree $2i$ is ${{y}\choose{2i}}$, which comes from the term $y^2$ in $\vec{P}(x,y)$, and one can verify directly that the coefficient of ${{y}\choose{2i}}$ in ${{\vec{P}(x,y)}\choose {i}}$ is $(0,0,\frac{(2i)!}{i!}, \frac{(2i)!}{i!})$. Consequently, $\PP_{i,2i}=\Span\{\vec{v}_3\}$.

In the case $i+1\leqslant j\leqslant 2i-1$, we have
\begin{align}\label{inclusion 1 in the proof of x, x+y, x+y^2, x+y+y^2}
    \PP_{i,j}\subseteq\{{0}\}\times\{0\}\times\RR\times\RR
\end{align}
because the polynomials ${{x}\choose{i}}$ and ${{x+y}\choose{i}}$ both have degree $i$, and we claim that (\ref{inclusion 1 in the proof of x, x+y, x+y^2, x+y+y^2}) is an equality. Since $\PP_{i,2i}\subseteq\PP_{i,j}$, we know that $\PP_{i,j}$ contains $\vec{v}_3$, and it remains to show that $\PP_{i,j}$ contains a vector of the form $(0,0,a,b)$ for some $a \neq b$. To this goal, we look at the coefficient of ${{y}\choose{2i-1}}$ in ${{\vec{P}(x,y)}\choose {i}}$. 
Note that
\begin{align*}
    {{y+y^2}\choose{i}}={{y^2}\choose{i}}+y{{y^2}\choose{i-1}}+R(y),
\end{align*}
where $R(y)$ has degree $2i-2$. In particular, the polynomial $y{{y^2}\choose{i-1}}$ has degree $2i-1$, and so has a nonzero coefficient at ${{y}\choose{2i-1}}$. Therefore the coefficient of ${{y}\choose{2i-1}}$ in ${{\vec{P}(x,y)}\choose {i}}$ is of the form $(0,0,a,b)$ for distinct integers $a\neq b$, implying that $\PP_{i,j}=\{0\}\times\{0\}\times\RR\times\RR$, as claimed.

To handle the case $1\leqslant j\leqslant i$, we look at $\PP_{i,i}$ and note that $\PP_{i,j}\supseteq\PP_{i,i}$ for these values of $j$ by Lemma \ref{properties of polynomial spaces}. By the previous case, we already know that $\PP_{i,j}\supseteq \{0\}\times\{0\}\times\RR \times\RR$. The space $\PP_{i,i}$ moreover contains $\vec{v}_1$, which is the coefficient of ${{x}\choose{i}}$ in ${{\vec{P}(x,y)}\choose {i}}$, and $\vec{v}_2$, which is the coefficient of ${{x}\choose{i-1}}y$ by Lemma \ref{relating coefficients of different monomials} and (\ref{presentation of x, x+y, x+y^2, x+y+y^2}). Thus $\PP_{i,i}$ is all of $\RR^4$.
\end{proof}

Having established the structure of $\PP_{i,j}$, it is straightforward to establish the following lemma.
\begin{corollary}\label{x, x+y, x+y^2, x+y+y^2 satisfies the filtration condition}
The polynomial map $\vec{P}$ satisfies the filtration condition. 
\end{corollary}
\begin{proof}
The proof proceeds by verifying that $\PP_{i_1,j_1}\cdot\PP_{i_2,j_2}\subseteq\PP_{i_1+i_2, j_1+j_2}$ for all ${i_1,i_2,j_1,j_2\geqslant 1}$ on a case-by-case basis using Lemma \ref{structure of P for x, x+y, x+y^2, x+y+y^2}. The details are rather tedious and unsophisticated, therefore we leave them to the reader.
\end{proof}

Using Lemma \ref{structure of P for x, x+y, x+y^2, x+y+y^2}, we get an explicit presentation for $G^P_j$ which tells us how this subgroup distinguishes from $G^P_{j+1}$.
\begin{lemma}\label{structure of G^P_j for x, x+y, x+y^2, x+y+y^2}
For $j=1$, we have
\begin{align*}
    G^P_1 =\langle G^P_{2}, h^{\vec{v}_1}, h^{\vec{v}_2}: h\in G_1 \rangle.
\end{align*}
If $j\geqslant 1$ is even, then
\begin{align*}
    G^P_j =\langle G^P_{j+1}, g^{\vec{v}_3}, h^{\vec{v}_1}, h^{\vec{v}_2}: g\in G_{\frac{j}{2}}, h\in G_j\rangle.
\end{align*}
If $j\geqslant 3$ is odd, then 
\begin{align*}
    G^P_j =\langle G^P_{j+1}, g^{\vec{v}_4}, h^{\vec{v}_1}, h^{\vec{v}_2}: g\in G_{\frac{j+1}{2}}, h\in G_j\rangle.
\end{align*}
\end{lemma}
\begin{proof}
The lemma follows from combining Lemmas \ref{x, x+y, x+y^2, x+y+y^2 satisfies the filtration condition} and \ref{structure of G^P_j} with the structural information on $\PP_{i,j}$ that we obtain from Lemma \ref{structure of P for x, x+y, x+y^2, x+y+y^2}.
\end{proof}
Having established the structure of the subgroups $G^P_j$, we are ready to prove that $\vec{P}$ equidistributes.
\begin{theorem}[$x,\; x+y,\; x+y^2,\; x+y+y^2$ equidistributes]\label{equidistribution of x, x+y, x+y^2, x+y+y^2}
Let $G/\Gamma$ be a filtered nilmanifold of degree $s$ and complexity $M$. Suppose $g\in\poly(\ZZ,G_\bullet)$ is $p$-periodic, $A$-irrational, and satisfies $g(0)=1$. Then the sequence $g^P\in\poly(\ZZ^2,G^P_\bullet)$ is $O_{M}(A^{-c_M})$-equidistributed on $G^P/\Gamma^P$ for some $c_M>0$.
\end{theorem}

In the proof of Theorem \ref{equidistribution of x, x+y, x+y^2, x+y+y^2}, we shall use the following useful lemma. 
\begin{lemma}[Integer multiples do not matter]\label{integer multiples don't matter}
Let $j\geqslant 1$, and suppose that $\eta: G \to\RR$ is a horizontal character such that $\eta\circ g^P$ is $\ZZ$-valued and $\eta|_{G_{j+1}^P}=0$. 
Suppose moreover that for some $\vec{v}\in\PP_{i,j}\cap\ZZ^4$ there exists a nonzero integer $a$ such that $\eta(g_i^{a\vec{v}})\in\ZZ$. Then $\eta(g_i^{\vec{v}})\in\ZZ$ assuming that $p$ is sufficiently large with respect to $a$.

%Fix $i\in\NN_+$, and define $I\subset\{1, 2, 3, 4\}$ so that $\vec{v}_l\in\PP_{i,j}$ for all $l\in I$ (so $I=\{1, 2, 3\}$ if $i=1$ and $I= \{1, 2, 3,4\}$ otherwise). Suppose that for all $l\in I$ there exists a nonzero integer $a_l$ such that $a_l \eta(g_i^{\vec{v}_l})\in\ZZ$. Then $\eta(g_i^{\vec{v}})\in\ZZ$ for any $\vec{v}\in\PP_{i,1}\cap\ZZ^4$ assuming that $p$ is sufficiently large.
\end{lemma}
There is nothing special about our particular progression here - Lemma \ref{integer multiples don't matter} works for any polynomial progression, therefore we shall also use it for progressions examined in next sections.
\begin{proof}
%We let all the constants depend on $s$.

By Lemma \ref{scaling a polynomial sequence}, $g_i^{p^i}\in\Gamma_i$ mod $G_{i+1}$, and so $g_i^{\vec{v}} = (\gamma_i h)^{\vec{v}}$ for some $\gamma_i\in \Gamma_i$ and $h\in G_{i+1}$. Using (\ref{Baker-Campbell-Hausdorff 1}), Lemma \ref{properties of polynomial spaces} and Lemma \ref{inclusion in polynomial spaces of special form}, we note that $h^{\vec{v}}\in G_{j+1}^P$, and similarly for commutators emerging from applying (\ref{Baker-Campbell-Hausdorff 1}). We therefore have $g_i^{p^i\vec{v}}\in\Gamma_j^P\mod G^P_{j+1}$, from which we deduce that $p^i\eta(g_i^{\vec{v}})\in\ZZ$. The assumption that ${\eta(g_i^{a\vec{v}})\in\ZZ}$ for some nonzero integer $a$ further implies that $\gcd(a, p^i)\eta(g_i^{\vec{v}})\in\ZZ$. Taking $p$ sufficiently large guarantees that $\gcd(a, p^i)=1$, and so $\eta(g_i^{\vec{v}})\in\ZZ$.

%By Lemma ..., $g_i^{p^s}\in\Gamma_i$ mod $G_i$, and so by (\ref{Baker-Campbell-Hausdorff 1}), we have $g_i^{p^s\vec{v_l}}\in\Gamma_j^P\mod G^P_{j+1}$ for any $l\in\{1,2,3,4\}$. Therefore $p^s\eta(g_i^{\vec{v}_l})\in\ZZ$. For $l\in I$, the assumption $a_l \eta(g_i^{\vec{v}_l})\in\ZZ$ for some nonzero integer $a_l$ further implies that $\gcd(a_l, p^s)\eta(g_i^{\vec{v}_l})\in\ZZ$. Taking $p$ sufficiently large further guarantees that $\gcd(a_l, p^s)=1$, and so $\eta(g_i^{\vec{v}_l})\in\ZZ$. Since the basis for $\PP_{i,1}$ given by $X=\{\vec{v}_l\}_{l\in I}$ is of the row echelon type, each vector $\vec{v}\in\PP_{i,1}\cap\ZZ^4$ is an integer linear combination of vectors from $X$. Hence $\eta(g_i^{\vec{v}})$ is an integer linear combination of integers, and so it is in $\ZZ$.

%By a multiparameter version of Lemma 5.4 of \cite{candela_sisask_2012}, there exists an integer $k\ll_s 1$ such that we have $g_i^{p^k\vec{v_l}}\in\Gamma_\mod G^\nabla_i$. Since each $\vec{v}_k$ is integer-valued, we have $g_i^{p^i\vec{v}_k}\in\Gamma^P\mod G^$  
\end{proof}

The general strategy of our proof of Theorem \ref{equidistribution of x, x+y, x+y^2, x+y+y^2}, as well as Theorems \ref{equidistribution of x, x+Q(y), x+R(y), x+Q(y)+R(y)}, \ref{equidistribution of x, x+Q(y), x+2Q(y), x+R(y), x+2R(y)} and \ref{equidistribution of multiparameter sequence coming from x, x+y, x+y^2, x+y+y^2}, follows the methods used in \cite{green_tao_2010a} to prove Theorem 1.11, the counting lemma. However, the technical details are quite different due to the fact that progressions dealt with in these theorems are no longer homogeneous. The main difficulty comes from the fact that the coefficients of polynomial maps of the form $\eta\circ g^P(x,y)$ for some horizontal character $\eta$ have contributions coming from ${{\vec{P}(x,y)}\choose{i}}$ for several values of $i$. This difficulty is not present when $\vec{P}$ is a linear form in several variables, as then each power ${{\vec{P}}\choose{i}}$ is a homogeneous polynomial map of a different degree.

 \begin{proof}[Proof of Theorem \ref{equidistribution of x, x+y, x+y^2, x+y+y^2}]
 Suppose that the sequence $g^P\in\poly(\ZZ^2,G^P_\bullet)$ is not $O_{M}(A^{-c_M})$-equidistributed. By Theorem \ref{equidistribution theorem}, there exists a nontrivial horizontal character $\eta:G^P\to\RR$ of complexity at most $cA$ for an appropriately chosen $c>0$, for which the polynomial $\eta\circ g^P$ is $\ZZ$-valued. Let $j$ be the largest natural number such that $\eta|_{G^P_j}\neq 0$. By assumption, $\eta$ annihilates ${G^P_{j+1}}$. 
 
 We use the maximality of $j$, the properties of $\eta$, and the structural information on $G^P$ contained in Lemma \ref{structure of G^P_j for x, x+y, x+y^2, x+y+y^2} to contradict the $A$-irrationality of $g$. We do this by inspecting the coefficients of $\eta\circ g^P$.
 
  We first do the model case $j=1$; it is different from and less complicated than the case $j>1$, and it can be used to illustrate the argument for the latter. Assuming $j=1$, we have
\begin{align*}
    \eta\circ g^P(x,y)=\eta(g_1^{\vec{v}_1}) x + \eta(g_1^{\vec{v}_2}) y,
\end{align*}
as all the other terms are annihilated by $\eta$.
Using the fact that $\eta\circ g^P(x,y)\in\ZZ$ for all $x,y\in\FF_p$, we deduce that $\eta(g_1^{\vec{v}_1})$ and $\eta(g_1^{\vec{v}_2})$  are both in $\ZZ$.

We define
\begin{align*}
    \xi_i(h)=\eta(h^{\vec{v}_i})
\end{align*}
for each $i = 1,2$ and $h\in G_1$. By Corollary \ref{horizontal characters reduce to i-th level characters}, the functions $\xi_i$ are 1-st level characters that annihilate $g_1$. By Lemma \ref{structure of G^P_j for x, x+y, x+y^2, x+y+y^2}, if both of them are trivial, then so is $\eta$, therefore at least one of them is non-trivial. The bound on the modulus of $\eta$ and the fact that the vectors $\vec{v}_1$ have entries of size $O(1)$ imply that $|\xi_{i}|\leqslant A$, provided that the constant $c$ is appropriately chosen. This contradicts the $A$-irrationality of $g$, implying that $g^P$ is $O_{M}(A^{-c_M})$-equidistributed.

For the rest of the proof, we assume that $j>1$. We split into two cases based on the parity of $j$. We shall only give the proof when $j$ is even, as the other case follows similarly.
First, the assumption that $\eta\circ g^P$ is $\ZZ$-valued implies that $\eta(g_j^{\vec{v}_1})\in\ZZ$, since this is the coefficient of ${{x}\choose{j}}$. Second, we have $\eta(g_j^{\vec{v}_2})\in\ZZ$, as this is the coefficient of ${{x}\choose{j-1}}y$ by Lemma \ref{relating coefficients of different monomials} and (\ref{presentation of x, x+y, x+y^2, x+y+y^2}). By Lemma \ref{structure of P for x, x+y, x+y^2, x+y+y^2}, we have $g_j^{\vec{v}_3}, g_j^{\vec{v}_4}\in G^P_{j+1}$, implying  $\eta(g_j^{\vec{v}_3})=\eta(g_j^{\vec{v}_4})=0$. Using the fact that each vector in $\ZZ^4$ is an integral linear combination of $\vec{v}_1, \vec{v}_2, \vec{v}_3, \vec{v}_4$, we obtain that
\begin{align}\label{statement 1 in the proof of true complexity for x, x+y, x+y^2, x+y+y^2}
    \eta(g_j^{\vec{v}})\in\ZZ
\end{align}
for any $\vec{v}\in\ZZ^4$.

Our goal now is to show that $\eta(g_{\frac{j}{2}}^{\vec{v}_3})\in\ZZ$. To this end, we look at the coefficient of ${{y}\choose{j}}$ in $\eta\circ g^P(x,y)$, which is of the form
\begin{align}\label{statement 2 in the proof of true complexity for x, x+y, x+y^2, x+y+y^2}
    \frac{j!}{(\frac{j}{2})!}\eta(g_{\frac{j}{2}}^{\vec{v}_3})+\sum_{i=\frac{j}{2}+1}^s\eta(g_i^{\vec{w}_i})
\end{align}
for some $\vec{w}_i\in\PP_{i,j}$. Note that there is no contribution coming from $g_i$ for $i<\frac{j}{2}$ because $\deg({{P(x,y)}\choose{i}})<j$ for these values of $i$. If $i\neq j$ and $i\geqslant\frac{j}{2}+1$, then $\vec{w}_i\in\Span\{\vec{v}_3,\vec{v}_4\}$ because ${{x}\choose{i}}$ and ${{x+y}\choose{i}}$ are homogeneous polynomials of degree $i$, and so $g_i^{\vec{w}_i}\in G_{j+1}^P$ by Lemma \ref{structure of P for x, x+y, x+y^2, x+y+y^2}. Therefore $\eta(g_i^{\vec{w}_i})=0$ by the property of $j$-th level characters. If $i = j$, then $\vec{w}_i\in\ZZ^4$, but we know from (\ref{statement 1 in the proof of true complexity for x, x+y, x+y^2, x+y+y^2}) that $\eta(g_j^{\vec{w}})\in\ZZ$ for any $\vec{w}\in\ZZ^4$. Thus, the entire contribution of $\sum\limits_{i=\frac{j}{2}+1}^s\eta(g_i^{\vec{w}_i})$ is in $\ZZ$. Using Lemma \ref{integer multiples don't matter}, we deduce that $\eta(g_{\frac{j}{2}}^{\vec{v}_3})\in\ZZ$  for sufficiently large $p$.

We define
\begin{align*}
    \tau(h)=\eta(h^{\vec{v}_3}) \quad {\rm{and}}\quad \xi_i(h)=\eta(h^{\vec{v}_i})
\end{align*}
for $i=1,2$ on $G_\frac{j}{2}$ and $G_j$ respectively. By Lemma \ref{horizontal characters reduce to i-th level characters}, the functions $\tau$ and $\xi$ are $\frac{j}{2}$-th and $j$-th level characters which send $g_{\frac{j}{2}}$ and $g_j$ to $\ZZ$ respectively. By Lemma \ref{structure of G^P_j for x, x+y, x+y^2, x+y+y^2}, if all of them are trivial, then so is $\eta$, and therefore at least one of $\tau, \xi_i$ is non-trivial and of complexity at most $O(|\eta|)\leqslant A$ upon taking $c>0$ sufficienly small. This contradicts the $A$-irrationality of $g$, implying that $g^P$ is $O_{M}(A^{-c_M})$-equidistributed.
 \end{proof}

%%%
%%%
%%% Counting lemma for $x, x+Q(y), x+R(y), x+Q(y)+R(y)$
%%%
%%%

\section{An equidistribution result for $x,\; x+Q(y),\; x+R(y),\; x+Q(y)+R(y)$}\label{section on x, x+Q(y), x+R(y), x+Q(y)+R(y)}
We now generalize the result of the previous section by considering the configuration
\begin{align}\label{presentation of x, x+Q(y), x+R(y), x+Q(y)+R(y)}
    \vec{P}(x,y) = (x,\; x+Q(y),\; x+R(y),\; x+Q(y)+R(y))
\end{align}
for integral polynomials $Q,R$ of degrees $1\leqslant d_1 < d_2$ respectively. From (\ref{presentation of x, x+Q(y), x+R(y), x+Q(y)+R(y)}), we can deduce that
\begin{align*}
    \PP_{1,1}&=\Span\{(1,1,1,1), (0,1,0,1), (0,0,1,1)\}=\Span\{\vec{v}_1, \vec{v}_2, \vec{v}_3\}\\
    \PP_{1,2} = ... =\PP_{1,d_1}&=\Span\{(0,1,0,1),(0,0,1,1)\} = \Span\{\vec{v}_2, \vec{v}_3\}\\
    \PP_{1,d_1+1} = ... =\PP_{1,d_2}&= \Span\{(0,0,1,1)\} = \Span\{ \vec{v}_3\}\\
    \PP_{1,j} &=     0 \quad {\rm{for}} \quad j\geqslant d_2+1.
\end{align*}
where
\begin{align*}
    \vec{v}_1 = (1,1,1,1), \quad \vec{v}_2 = (0,1,0,1), \quad \vec{v}_3 = (0,0,1,1)\quad {\rm{and}} \quad \vec{v}_4 = (0,0,0,1).
\end{align*}

Our next goal is to deduce expressions for $\PP_{i,j}$ for $i>1$.
\begin{lemma}\label{structure of P for x, x+Q(y), x+R(y), x+Q(y)+R(y)}
For $i>1$, the following holds:
\[ \PP_{i,j}=\begin{cases}
\RR^4 = \Span\{\vec{v}_1, \vec{v}_2, \vec{v}_3, \vec{v}_4\},\; &1\leqslant j\leqslant i\\
0\times \RR\times\RR\times\RR = \Span\{\vec{v}_2, \vec{v}_3, \vec{v}_4\},\; &i+1\leqslant j\leqslant id_1 \\
0\times 0\times \RR\times \RR =\Span\{\vec{v}_3, \vec{v}_4\},\; &i d_1+1\leqslant j\leqslant (i-1)d_2+d_1\\
\Span\{\vec{v}_3\},\; &(i-1)d_2+d_1+1 \leqslant j \leqslant id_2\\
0,\; &j>i d_2.
\end{cases}
\]
\end{lemma}
\begin{proof}
The case $j>id_2$ follows trivially from the fact that the polynomial ${{\vec{P}(x,y)}\choose{i}}$ has degree $id_2$.

In the process of deducing the other cases, we shall use the fact that
\begin{align}\label{rewriting polynomial map in Taylor basis 1}
    \vec{P}(x,y)=\vec{v}_1 x + \vec{P}(0,y)= \vec{v}_1 x + \vec{v}_2 Q(y) + \vec{v}_3 R(y)
\end{align}
and (\ref{identity for binomial coefficients}).
Combining (\ref{rewriting polynomial map in Taylor basis 1}) and (\ref{identity for binomial coefficients}), we rewrite ${{\vec{P}(x,y)}\choose{i}}$ as
\begin{align}\label{rewriting polynomial map in Taylor basis 2}
  {{\vec{P}(x,y)}\choose{i}} = \vec{v}_1{{x}\choose{i}}+\sum_{l=1}^{i-1}{{x}\choose{l}}{{\vec{P}(0,y)}\choose{i-1}}+\left(0, {{Q(y)}\choose{i}}, {{R(y)}\choose{i}}, {{Q(y)+R(y)}\choose{i}}\right).
\end{align}
We can further rewrite the last vector in the sum as
\begin{align}\label{rewriting polynomial map in Taylor basis 3}
    &\left(0, {{Q(y)}\choose{i}}, {{R(y)}\choose{i}}, {{Q(y)+R(y)}\choose{i}} \right)\\ \nonumber
    &= \vec{v}_2 {{Q(y)}\choose{i}} + \sum_{l=1}^{i-1}\vec{v}_4{{Q(y)}\choose{l}}{{R(y)}\choose{i-l}}+\vec{v}_3{{R(y)}\choose{i}}.
\end{align}

From (\ref{rewriting polynomial map in Taylor basis 2}) and (\ref{rewriting polynomial map in Taylor basis 3}), we see that the only monomials in ${{\vec{P}(x,y)}\choose{i}}$ of degree greater than $(i-1)d_2+d_1$ in ${{\vec{P}(x,y)}\choose{i}}$ have coefficients of the form $a\vec{v}_3$ for some $a\in\ZZ$. In particular, the value $a$ will be nonzero for the coefficient of ${{y}\choose{i d_2}}$, implying the case $(i-1)d_2+d_1+1 \leqslant j \leqslant id_2$ by Lemma \ref{integer multiples don't matter}.

To deduce the case $i d_1+1\leqslant j\leqslant (i-1)d_2+d_1$, we note that the projection of ${{\vec{P}(x,y)}\choose{i}}$ onto the first coordinate has degree $i$, while its projection onto the second coordinate has degree $id_1$, and so 
\begin{align*}
    \PP_{i,j}\subseteq 0\times 0 \times \RR \times \RR
\end{align*}
for these values of $j$. We claim this is an equality. We already know that ${\vec{v}_3\in\PP_{i,j}}$ since its multiple is the coefficient of ${{y}\choose{i d_2}}$. We claim that the coefficient of the monomial ${{y}\choose{(i-1)d_2+d_1}}$ is of the form $a\vec{v}_3+b\vec{v}_4$ for integers $a,b$ such that $b\neq 0$, which will imply imply this case. This follows from (\ref{rewriting polynomial map in Taylor basis 3}), the assumption $d_1<d_2$, and the observation that the polynomial ${Q(y)}{{R(y)}\choose{i-1}}$ has degree $(i-1)d_2+d_1$, thus contributing to the coefficient of ${{y}\choose{(i-1)d_2+d_1}}$.

The next case, $i+1\leqslant j\leqslant id_1$, only happens if $d_1>1$, and so we make this assumption. Since the projection of ${{\vec{P}(x,y)}\choose{i}}$ onto the first coordinate has degree $i$, we deduce that 
\begin{align*}
    \PP_{i,j}\subseteq 0\times \RR \times \RR \times \RR.
\end{align*}
To prove that this is an equality, it remains to show in the light of the previous cases that a vector of the form $a\vec{v}_2+b\vec{v}_3+c\vec{v}_3$ is in $\PP_{i,j}$ for some $a\neq 0$. Since the projection of  ${{\vec{P}(x,y)}\choose{i}}$ onto the second coordinate has degree $i d_1$, the coefficient of ${{y}\choose{i d_1}}$ is of this form, implying this case. 

If $d_1 =1$, then the same argument implies that $\vec{v}_2\in\PP_{i,i}$.

Finally, the case $1\leqslant j\leqslant i$ follows from combining the previous case, Lemma \ref{properties of polynomial spaces} and the observation that $\vec{v}_1$ is the coefficient of ${{x}\choose{i}}$.
\end{proof}

Having established the structure of $\PP_{i,j}$, it is straightforward to deduce the following lemma.
\begin{corollary}\label{x, x+Q(y), x+R(y), x+Q(y)+R(y) satisfies the filtration condition}
The polynomial map $P$ satisfies the filtration condition. 
\end{corollary}

We come to the main result of this section, which is case (i) of Theorem \ref{equidistribution theorem in the intro}.
\begin{theorem}[$x,\; x+Q(y),\; x+R(y),\; x+Q(y)+R(y)$ equidistributes]\label{equidistribution of x, x+Q(y), x+R(y), x+Q(y)+R(y)}
Let $G/\Gamma$ be a filtered nilmanifold of degree $s$ and complexity $M$. Suppose $g\in\poly(\ZZ,G_\bullet)$ is $p$-periodic, $A$-irrational, and satisfies $g(0)=1$. Then the sequence $g^P\in\poly(\ZZ^2,G^P_\bullet)$ is $O_{M}(A^{-c_M})$-equidistributed on $G^P/\Gamma^P$ for some $c_M>0$.
\end{theorem}

 \begin{proof}%[Proof of Theorem \ref{equidistribution of x, x+Q(y), x+R(y), x+Q(y)+R(y)}]
 Suppose that the sequence $g^P\in\poly(\ZZ^2,G^P_\bullet)$ is not $O_{M}(A^{-c_M})$-equidistributed. By Theorem \ref{equidistribution theorem}, there exists a nontrivial horizontal character $\eta:G^P\to\RR$ of complexity at most $cA$ for some $c>0$ to be chosen later, such that $\eta\circ g^P\in\ZZ$. Let $j$ be the largest natural number such that $\eta|_{G^P_j}\neq 0$. By assumption, $\eta$ annihilates ${G^P_{j+1}}$. 
 
If $j=1$, we proceed exactly as in Theorem \ref{equidistribution of x, x+y, x+y^2, x+y+y^2}, and so we assume that $j>1$.  For any $i\geqslant 1$, we define
\begin{align*}
    \xi_{i,k}(h)=\eta(h^{\vec{v}_k})
\end{align*}
for $h\in G_i$ and each $k\in\{1,2,3,4\}$ such that $\vec{v}_k\in\PP_{i,j}$ but $\vec{v}_k\notin\PP_{i,j+1}$. The maps $\xi_{i,k}$ define $i$-th level characters on $G$ by Corollary \ref{horizontal characters reduce to i-th level characters}. By Lemma \ref{structure of G^P_j}, if all of $\xi_{i,k}$ were trivial, so would be $\eta$, implying that at least one of $\xi_{i,k}$ is nontrivial. The bound on the modulus of $\eta$ and the fact that the vectors $\vec{v}_k$ have entries of size $O(1)$ imply that $|\xi_{i,k}|\leqslant A$, provided that the constant $c$ is appropriately chosen.

Our goal is to show that for all pairs $(i,k)$ as above, we have $\xi_{i,k}(g_i)\in\ZZ$. Since at least one of these $\xi_{i,k}$ is nontrivial and of modulus at most $A$, we obtain a contradiction of the $A$-irrationality of $g$, implying that $g^P$ is $O_{M}(A^{-c_M})$-equidistributed. By enumerating all such pairs $(i,k)$, we observe that all we have to show is that $\eta$ sends the following elements to $\ZZ$:
\begin{enumerate}
    \item $g_j^{\vec{v}_1}$;
    \item $g_\frac{j}{d_1}^{\vec{v}_2}$, if $d_1|j$;
    \item $g_\frac{j+d_2-d_1}{d_2}^{\vec{v}_4}$, if $d_2| (j-d_1)$;
    \item $g_\frac{j}{d_2}^{\vec{v}_3}$, if $d_2|j$.
\end{enumerate}
That $\eta(g_j^{\vec{v}_1})\in\ZZ$ follows from observing that this is precisely the coefficient of ${{x}\choose{j}}$ is $\eta(g_j^{\vec{v}_1})$; showing that other elements are sent to $\ZZ$ is a bit more involved.

We assume $d_1|j$, and we claim that $\eta(g_\frac{j}{d_1}^{\vec{v}_2})\in\ZZ$. From Lemma \ref{relating coefficients of different monomials} we know that the coefficient of ${{x}\choose{\frac{j}{d_1}-1}}{{y}\choose{d_1}}$ is of the form
\begin{align*}
    a\eta(g_{\frac{j}{d_1}}^{\vec{v}_2})+b\eta(g_{\frac{j}{d_1}}^{\vec{v}_3}) + \sum_{i=\frac{j}{d_1}+1}^s \eta(g_i^{\vec{w}_i})
\end{align*}
for some integers $a,b$ such that $a\neq 0$, and some integer vectors $\vec{w}_i\in\Span\{\vec{v}_2, \vec{v}_3, \vec{v}_4\}$. If $i\geqslant\frac{j}{d_1}+1$, then $j\leqslant id_1 - d_1$, and so $j+1\leqslant id_1$. It follows from this that $\vec{w}_i\in\PP_{i,j+1}$, therefore $\eta(g_i^{\vec{w}_i})=0$. We moreover have that $\vec{v}_3\in\PP_{\frac{j}{d_1},j+1}$, and so $\eta(g_{\frac{j}{d_1}}^{\vec{v}_3})=0$ as well. From this, Lemma \ref{integer multiples don't matter} and the vanishing of the coefficient of ${{x}\choose{\frac{j}{d_1}-1}}{{y}\choose{d_1}}$ mod $\ZZ$, we conclude that $\eta(g_{\frac{j}{d_1}}^{\vec{v}_2})\in\ZZ$.

We are left with showing that the elements in (iii) and (iv) are sent to $\ZZ$ by $\eta$. Note that since $1\leqslant d_1<d_2$, the number $d_2$ cannot simultaneously divide $j-d_1$ and $j$, therefore only one of these cases at a time is available. We assume that $d_2$ divides $j$, and the other case will follow similarly. To show that $\eta(g_\frac{j}{d_2}^{\vec{v}_3})\in\ZZ$, we look at the coefficient of ${{y}\choose{d_2}}$, which is of the form
\begin{align}\label{equation in Section 5}
    a\eta(g_\frac{j}{d_2}^{\vec{v}_3}) + \sum_{i=\frac{j}{d_2}+1}^s \eta(g_i^{\vec{w}_i})
\end{align}
for some nonzero integer $a$ and vectors $w_i\in\Span\{\vec{v}_2, \vec{v}_3, \vec{v}_4\}\cap\PP_{i,j}\cap\ZZ^4$. By Lemma \ref{structure of P for x, x+Q(y), x+R(y), x+Q(y)+R(y)}, we have $w_i\in\PP_{i,j+1}$ unless $j = id_1$ or $j=(i-1)d_2+d_1$. If $d_1$ divides $j$, then we have already shown that $\eta(g_\frac{j}{d_1}^{\vec{v}_2})\in\ZZ$, and we moreover have $\eta(g_\frac{j}{d_1}^{\vec{v}_3})=0$ and $\eta(g_\frac{j}{d_1}^{\vec{v}_4})=0$ since $\vec{v}_3, \vec{v}_4\in\PP_{\frac{j}{d_1}, j+1}$. Therefore $\eta(g_\frac{j}{d_1}^{\vec{w}})\in\ZZ$ for any $w\in\Span\{\vec{v}_2, \vec{v}_3, \vec{v}_4\}\cap\PP_{i,j}\cap\ZZ^4$. The case $j=(i-1)d_2+d_1$ does not happen by our assumption that $d_2$ does not divide $j-d_1$. Thus the sum in (\ref{equation in Section 5}) vanishes mod $\ZZ$, implying that $a\eta(g_\frac{j}{d_2}^{\vec{v}_3})\in\ZZ$. That $\eta(g_\frac{j}{d_2}^{\vec{v}_3})\in\ZZ$ follows by Lemma \ref{integer multiples don't matter}.

%The proofs that elements in (iii) and (iv) are sent by $\eta$ to $\ZZ$ follows by similar methods. We look at appropriate coefficients of ${{\vec{P}(x,y)}\choose{i}}$ (we take ${{y}\choose{(i-1)d_2+d_1}}$ for the former and ${{y}\choose{d_2}}$ for the latter), show that they are equal to an integer multiple of $\eta$ evaluated at  $g_\frac{j+d_2-d_1}{d_2}^{\vec{v}_4}$ and $g_\frac{j}{d_2}^{\vec{v}_3}$ respectively plus some vanishing terms, and conclude by Lemma \ref{integer multiples don't matter} that $\eta(g_\frac{j+d_2-d_1}{d_2}^{\vec{v}_4})$ and $\eta(g_\frac{j}{d_2}^{\vec{v}_3})$ are both in $\ZZ$.
%Applying Lemma \ref{integer multiples don't matter} once again and using the fact that $\eta$ sends all of $g_s^{\vec{v}_1}, g_s^{\vec{v}_2}, g_s^{\vec{v}_3}, g_s^{\vec{v}_4}$ to $\ZZ$, we have that $\eta(g_s^{\vec{v}})\in\ZZ$ for all $\vec{v}\in\ZZ^4$. Thus, all the expressions in $\eta\circ g^P$ involving $g_s$ vanish mod $\ZZ$. 

 %This contradicts the $A$-irrationality of $g$, implying that $g^P$ is $O_{M}(A^{-c_M})$-equidistributed.
 \end{proof}

%%%
%%%
%%% Counting lemma for x, x+Q(y), x+2Q(y), x+R(y), x+2R(y)
%%%
%%%
\section{An equidistribution result for $x,\; x+Q(y),\; x+2Q(y),\; x+R(y),\; x+2R(y)$}\label{section on x, x+Q(y), x+2Q(y), x+R(y), x+2R(y)}

We now turn our attention to the configuration 
\begin{align}\label{presentation of x, x+Q(y), x+2Q(y), x+R(y), x+2R(y)}
    \vec{P}(x,y) &= (x,\; x+Q(y),\; x+2Q(y),\; x+R(y),\; x+2R(y))
\end{align}
for polynomials $Q, R\in\ZZ[y]$ with zero constant terms of degrees $d_1, d_2$ respectively that moreover satisfy $1\leqslant d_1 <d_2 / 2$. Letting
\begin{align*}
    \vec{v}_1 = (1,1,1,1,1), \quad \vec{v}_2 = (0,1,2,0,0), \quad \vec{v}_3 = (0,0,0,1,2),
\end{align*}
we observe that
\begin{align*}
    \PP_{i,1} &= \Span\{(1,1,1,1,1), (0,1,2,0,0),  (0,0,0,1,2)\} = \Span\{\vec{v}_1, \vec{v}_2, \vec{v}_3\}, \\
    \PP_{i,2} = ... = \PP_{i, d_1} &= \Span\{(0,1,2,0,0),  (0,0,0,1,2)\} = \Span\{\vec{v}_2, \vec{v}_3\},\\
    \PP_{i, d_1+1} = ... = \PP_{i, d_2} &= \Span\{(0,0,0,1,2)\} = \Span\{\vec{v}_3\}\\
    \PP_{1,j} &= 0 \quad {\rm{for}} \quad j\geqslant d_2 + 1, 
\end{align*}
and we prove the following lemma giving the structure of the spaces $\PP_{i,j}$ for $i>1$
 \begin{lemma}\label{structure of P for x, x+Q(y), x+2Q(y), x+R(y), x+2R(y)}
Let $i>1$. Then 
\[ \PP_{i,j}=\begin{cases}
\Span\{\vec{v}_1, \vec{v}_2^{i-1}, \vec{v}_2^i, \vec{v}_3^{i-1}, \vec{v}_3^i\} = \RR^5,\; &1\leqslant j\leqslant i\\
\Span\{\vec{v}_2^{i-1}, \vec{v}_2^i, \vec{v}_3^{i-1}, \vec{v}_3^i\} = 0\times \RR\times\RR\times\RR \times \RR,\; &i+1\leqslant j\leqslant (i-1)d_1+1 \\
\Span\{\vec{v}_2^i, \vec{v}_3^{i-1}, \vec{v}_3^i\},\; &(i-1) d_1+2\leqslant j\leqslant i d_1\\
\Span\{\vec{v}_3^{i-1}, \vec{v}_3^i\} = 0\times 0\times 0\times\RR \times \RR,\; &i d_1+1\leqslant j\leqslant (i-1)d_2+1 \\
\Span\{\vec{v}_3^i\},\; &(i-1)d_2+2 \leqslant j \leqslant id_2\\
0,\; &j>i d_2
\end{cases}
\]
where $\vec{v}^k = (\vec{v}(1)^k, \vec{v}(2)^k, \vec{v}(3)^k, \vec{v}(4)^k, \vec{v}(5)^k)$  for any $k\in\RR\setminus{\{0\}}$ and $\vec{v}\in\RR^5$.

\end{lemma}
\begin{proof}

The statement is trivial for $j>id_2$.
%follows trivially from the fact that the polynomial ${{\vec{P}(x,y)}\choose{i}}$ has degree $id_2$. For $i d_1+1\leqslant j\leqslant (i-1)d_2+1$, we observe that all the monomials of this degree take the form only contain $y$ variable 
To obtain the expressions for $\PP_{i,j}$ for other values of $j$, we make two observations. First, we note from Lemma \ref{relating coefficients of different monomials} that for $0\leqslant k\leqslant i-1$, the coefficient of ${{x}\choose{k}}{{y}\choose{l}}$
\begin{align*}
    a_{k,l}\vec{v}_2^{i-k} + b_{k,l}\vec{v}_3^{i-k}
\end{align*}
for some integers $a_{k,l}$ which satisfy $a_{k,l}=0$ if $l>(i-k)d_1$ and $b_{k,l}=0$ if $l>(i-k)d_2$. Moreover, the numbers $a_{k,(i-k)d_1}$ and $b_{k, (i-k)d_2}$ are nonzero since $Q$ and $R$ have degrees $d_1$, $d_2$ respectively. By substituting $k=0,1$ and using $d_1<d_2$, we deduce that
\begin{align*}
    \vec{v}_2^i\in\PP_{i,i d_1}, \quad \vec{v}_2^{i-1}\in\PP_{i,(i-1)d_1+1}, \quad \vec{v}_3^i\in\PP_{i,i d_2}, \quad \vec{v}_3^{i-1}\in\PP_{i,(i-1)d_2+1},
\end{align*}
but
\begin{align*}
    \vec{v}_2^i\notin\PP_{i,i d_1+1}, \quad \vec{v}_2^{i-1}\notin\PP_{i,(i-1)d_1+2}, \quad \vec{v}_3^i\notin\PP_{i,i d_2 + 1}, \quad \vec{v}_3^{i-1}\notin\PP_{i,(i-1)d_2+2}.
\end{align*}

Second, we observe that for $k>1$, we have $\vec{v}_2^{i-k}\in\Span\{\vec{v}_2^i, \vec{v}_2^{i-1}\}$ and $\vec{v}_3^{i-k}\in\Span\{\vec{v}_3^i, \vec{v}_3^{i-1}\}$, and moreover $(i-k')d_r+k'\leqslant (i-k)d_r + k$ for any $k'<k$ and $r\in\{1, 2\}$. From this we see that to specify a basis for $\PP_{i,j}$, it is sufficient to look at whether $\vec{v}_2^{i-1}$, $\vec{v}_2^{i}$, $\vec{v}_3^{i-1}$ and $\vec{v}_3^{i}$ are in $\PP_{i,j}$ or not. The statements for $j>i$ follow by combining these observations. 

Lastly, we note that $\vec{v}_1$ is the coefficient of ${{x}\choose{i}}$, which together with $\PP_{i,j}\supseteq\PP_{i,j+1}$ implies the case $1\leqslant j\leqslant i$.

\end{proof}
\begin{corollary}\label{x, x+Q(y), x+2Q(y), x+R(y), x+2R(y) satisfies the filtration condition}
The polynomial map $P$ satisfies the filtration condition. 
\end{corollary}

We are now ready to show that if $g$ is irrational on $G/\Gamma$, then $g^P$ equidistributes on $G^P/\Gamma^P$. The proof follows the same logic as the proofs of Theorem \ref{equidistribution of x, x+y, x+y^2, x+y+y^2}, \ref{equidistribution of x, x+Q(y), x+R(y), x+Q(y)+R(y)} and \ref{equidistribution of x, x+Q(y), x+2Q(y), x+R(y), x+2R(y)}, however the technical details are different, as we are working with a different configuration 
\begin{theorem}[$x,\; x+Q(y),\; x+2Q(y),\; x+R(y),\; x+2R(y)$ equidistributes]\label{equidistribution of x, x+Q(y), x+2Q(y), x+R(y), x+2R(y)}
Let $G/\Gamma$ be a filtered nilmanifold of degree $s$ and complexity $M$. Suppose $g\in\poly(\ZZ,G_\bullet)$ is $p$-periodic, $A$-irrational, and satisfies $g(0)=1$. Then the sequence $g^P\in\poly(\ZZ^2,G^P_\bullet)$ is $O_{M}(A^{-c_M})$-equidistributed on $G^P/\Gamma^P$ for some $c_M>0$.
\end{theorem}

 \begin{proof} %[Proof of Theorem \ref{equidistribution of x, x+Q(y), x+2Q(y), x+R(y), x+2R(y)}]
 Suppose that the sequence $g^P\in\poly(\ZZ^2,G^P_\bullet)$ is not $O_{M}(A^{-c_M})$-equidistributed. By Theorem \ref{equidistribution theorem}, there exists a nontrivial horizontal character $\eta:G^P\to\RR$ of complexity at most $cA$ for an appropriately chosen $c>0$, such that $\eta\circ g^P\in\ZZ$. Let $j$ be the largest natural number such that $\eta|_{G^P_j}\neq 0$. By assumption, $\eta$ annihilates ${G^P_{j+1}}$. 
 
Like in Theorems \ref{equidistribution of x, x+y, x+y^2, x+y+y^2} and \ref{equidistribution of x, x+Q(y), x+R(y), x+Q(y)+R(y)}, we use the properties of $\eta$ and the structural information on $G^P$ contained in Lemma \ref{structure of P for x, x+Q(y), x+2Q(y), x+R(y), x+2R(y)} to contradict the $A$-irrationality of $g$. We do this by inspecting the coefficients of $\eta\circ g^P$. For any $i\geqslant 1$, we define
\begin{align*}
    \xi_{i,k,l}(h)=\eta(h^{\vec{v}_k^l})
\end{align*}
for $h\in G_i$ and each $k\in\{1,2,3\}$, $l\in\{i-1,i\}$ such that $\vec{v}_k^l\in\PP_{i,j}$ but $\vec{v}_k^l\notin\PP_{i,j+1}$. The maps $\xi_{i,k,l}$ define $i$-th level characters on $G$ by Corollary \ref{horizontal characters reduce to i-th level characters}. By definition of $G^P_{j}$, the group is generated precisely by $G^P_{j+1}$ and the elements of the form $h^{\vec{v}_k^l}$ for $h\in G_i$, $\vec{v}_k^l\in\PP_{i,j}\setminus{\PP_{i,j+1}}$ and $i\geqslant 1$. Therefore if all of $\xi_{i,k,l}$ were trivial, so would be $\eta$, implying that at least one of $\xi_{i,k,l}$ is nontrivial. The bound on the modulus of $\eta$ and the fact that the vectors $\vec{v}_k^l$ have entries of size $O(1)$ imply that $|\xi_{i,k,l}|\leqslant A$, provided that the constant $c$ is appropriately chosen.

Our goal is to show that for all triples $(i,k,l)$ as above, we have $\xi_{i,k,l}(g_i)\in\ZZ$. Since at least one of these $\xi_{i,k,l}$ is nontrivial and of modulus at most $A$, we obtain a contradiction of the $A$-irrationality of $g$, implying that $g^P$ is $O_{M}(A^{-c_M})$-equidistributed. We are thus left to show that $\eta$ sends the following elements to $\ZZ$:
\begin{enumerate}
    \item $g_j^{\vec{v}_1}$,
    \item $g_{\frac{j-1}{d_1}+1}^{\vec{v}_2^{\frac{j-1}{d_1}}}$, if $d_1|(j-1)$ and $j>1$,
    \item $g_{\frac{j}{d_1}}^{\vec{v}_2^\frac{j}{d_1}}$, if $d_1|j$,
    \item $g_{\frac{j-1}{d_2}+1}^{\vec{v}_3^{\frac{j-1}{d_2}}}$, if $d_2|(j-1)$ and $j>1$,
    \item $g_{\frac{j}{d_2}}^{\vec{v}_3^{\frac{j}{d_2}}}$, if $d_2|j$.
\end{enumerate}

We first look at the model case $j=1$, and then move on to the case $j>1$. Assuming $j=1$, we have to show that $\eta(g_j^{\vec{v}_1})\in\ZZ$, and also $\eta(g_1^{\vec{v}_2})$ if $d_1 = 1$. The first statement follows from inspecting the coefficient of $x$. For the second statement, we assume $d_1 = 1$; then $Q(y)=ay$ for some $a\in\ZZ$, and so the coefficient of $y$ is of the form $a\eta(g_1^{\vec{v}_2})$ plus terms that vanish by assumption that $\eta|_{G^P_2}=0$. Lemma \ref{integer multiples don't matter} thus implies that $\eta(g_1^{\vec{v}_2})\in\ZZ$, which finishes this case.

We assume from now on that $j>1$. The number $\eta(g_j^{\vec{v}_1})$ vanishes mod $\ZZ$ because it is the coefficient of ${{x}\choose{j}}$. We now proceed to show that the elements in (iii) and (v) are in $\ZZ$. To this end, we look at the coefficient of ${{y}\choose{j}}$ and assume that at least one of $d_1$, $d_2$ divides $j$. By evaluating the contributions coming from $\eta\left(g_i^{{\vec{P}(x,y)}\choose{i}}\right)$ for each $i\geqslant 1$, we observe that the coefficient of ${{y}\choose{j}}$ is of the form
\begin{align*}
    \sum_{i=\left\lceil\frac{j}{d_1}\right\rceil}^s a_i \eta(g_i^{\vec{v}_2^i}) + \sum_{i=\left\lceil\frac{j}{d_2}\right\rceil}^s b_i \eta(g_i^{\vec{v}_3^i})
\end{align*}
for integers $a_i, b_i\in\ZZ$. 
If $i>\frac{j}{d_1}$, then $j+1\leqslant i d_1$, and so $\vec{v}_2^i\in\PP_{i,j+1}$ by Lemma \ref{structure of P for x, x+Q(y), x+2Q(y), x+R(y), x+2R(y)}. Similarly, we have $\vec{v}_3^i\in\PP_{i,j+1}$ whenever $i>\frac{j}{d_2}$. Therefore, the coefficient of ${{y}\choose{j}}$ reduces to the rather unfortunate looking
\begin{align}\label{reduced coefficient of y choose j}
    a \eta\left(g_{\left\lceil\frac{j}{d_1}\right\rceil}^{\vec{v}_2^{\left\lceil\frac{j}{d_1}\right\rceil}}\right) + b \eta\left(g_{\left\lceil\frac{j}{d_2}\right\rceil}^{\vec{v}_3^{\left\lceil\frac{j}{d_2}\right\rceil}}\right)
\end{align}
for some integers $a$ and $b$.

We pause for a moment to analyse what happens if $j$ is not divisible by one of $d_1$, $d_2$. If $d_1$ does not divide $j$, then $\left\lceil\frac{j}{d_1}\right\rceil > \frac{j}{d_1}$, implying that $\vec{v}_2^{\left\lceil\frac{j}{d_1}\right\rceil}$ by the argument in the previous paragraph. Then the coefficient of ${{y}\choose{j}}$ is an integer multiple of $\eta\left(g_{\frac{j}{d_2}}^{\vec{v}_3^{\frac{j}{d_2}}}\right)$, and so $\eta\left(g_{\frac{j}{d_2}}^{\vec{v}_3^{\frac{j}{d_2}}}\right)\in\ZZ$ by Lemma \ref{integer multiples don't matter}. We similarly have $\eta\left(g_{\frac{j}{d_1}}^{\vec{v}_2^{\frac{j}{d_1}}}\right)\in\ZZ$ if $d_2$ does not divide $j$. 

The interesting case is when both $d_1$ and $d_2$ divide $j$, which we assume from now on. In this case, both $a$ and $b$ are nonzero due to the fact that $Q$ has degree $d_1$ and $R$ has degree $d_2$. 
%Thus (\ref{reduced coefficient of y choose j}) and Lemma \ref{integer multiples don't matter} tell us that if one of the terms vanishes mod $\ZZ$, so does the other.
The strategy now is this: by looking at the coefficients of $x{{y}\choose{j-d_1}}$, ${{x}\choose{2}}{{y}\choose{j-2d_1}}$ and ${{x}\choose{2}}{{y}\choose{j-d_1}}$, we shall show that $\eta\left(g_{\frac{j}{d_1}}^{\vec{v}_2^{\frac{j}{d_1}-1}}\right)$ and $\eta\left(g_{\frac{j}{d_1}}^{\vec{v}_2^{\frac{j}{d_1}-2}}\right)$ are both in $\ZZ$. Since the vector $\vec{v}_2^{\frac{j}{d_1}}$ is an integer linear combination of $\vec{v}_2^{\frac{j}{d_1}-1}$ and $\vec{v}_2^{\frac{j}{d_1}-2}$, we deduce that $\eta\left(g_{\frac{j}{d_2}}^{\vec{v}_3^{\frac{j}{d_2}}}\right)$ is in $\ZZ$. By Lemma \ref{integer multiples don't matter} and the fact that the coefficient of ${{y}\choose{d_1}}$ takes the form (\ref{reduced coefficient of y choose j}), it follows that $\eta\left(g_{\frac{j}{d_2}}^{\vec{v}_3^{\frac{j}{d_2}}}\right)\in\ZZ$ as well.

We start by analyzing the coefficient of $x{{y}\choose{y-d_1}}$, which is
\begin{align*}
    \sum_{i=\frac{j}{d_1}}^s a_i \eta(g_i^{\vec{v}_2^{i-1}}) + \sum_{i=\frac{j}{d_2}+1}^s b_i \eta(g_i^{\vec{v}_3^{i-1}})
\end{align*} for $a_i, b_i\in\ZZ$. The lower bounds in the range for $i$ come from $d_1<d_2$ and the observation that terms with smaller $i$ do not contribute to this monomial. By rearranging the inequality $i\geqslant\frac{j}{d_1}+1$ and Lemma \ref{structure of P for x, x+Q(y), x+2Q(y), x+R(y), x+2R(y)}, we infer that $\vec{v}_2^{i-1}\in\PP_{i,j+1}$ whenever $i\geqslant\frac{j}{d_1}+1$. Similarly, $\vec{v}_3^{i-1}\in\PP_{i,j+1}$ whenever $i\geqslant\frac{j}{d_2}+1$. Since $\eta$ vanishes on $G^P_{j+1}$, we deduce that the coefficient of $x{{y}\choose{y-d_1}}$ is an integer multiple of $\eta\left(g_{\frac{j}{d_1}}^{\vec{v}_2^{\frac{j}{d_1}-1}}\right)$, implying that $\eta\left(g_{\frac{j}{d_1}}^{\vec{v}_2^{\frac{j}{d_1}-1}}\right)\in\ZZ$ by Lemma \ref{integer multiples don't matter}.

We move on to the coefficient of ${{x}\choose{2}}{{y}\choose{j-2d_1}}$, which takes the form
\begin{align*}
    \sum_{i=\frac{j}{d_1}}^s a_i \eta(g_i^{\vec{v}_2^{i-2}}) + \sum_{i=\frac{j}{d_2}+2}^s b_i \eta(g_i^{\vec{v}_3^{i-2}})
\end{align*} for some $a_i, b_i\in\ZZ$. An important point here is that the second sum starts at $i=\frac{j}{d_2}+2$; this results from the observation that by the assumption $d_2 > 2d_1$, all monomials of ${{x}\choose{2}}{{R(y)}\choose{\frac{j}{d_2}-1}}$ have degree at most $j-d_2 < j-2d_1$ in $y$, therefore they do not contribute to ${{x}\choose{2}}{{y}\choose{j-2d_1}}$ (this is the only point where we are using the assumption). Performing a similar analysis as above, we deduce that all the terms involving $\vec{v}_2$ and $\vec{v}_3$ with $i\geqslant\frac{j}{d_1}+2$ and $i\geqslant\frac{j}{d_2}+2$ respectively vanish mod $\ZZ$. This leaves us with a coefficient of the form 
\begin{align*}
    a \eta\left(g_{\frac{j}{d_1}}^{\vec{v}_2^{\frac{j}{d_1}-2}}\right)+ b \eta\left(g_{\frac{j}{d_1}+1}^{\vec{v}_2^{\frac{j}{d_1}-1}}\right)
\end{align*} for some integers $a,b$ with $a\neq 0$. This is not exactly what we wanted; however, analysing the coefficient of ${{x}\choose{2}}{{y}\choose{j-d_1}}$ and using Lemma \ref{integer multiples don't matter} allows us to conclude that $\eta\left(g_{\frac{j}{d_1}+1}^{\vec{v}_2^{\frac{j}{d_1}-1}}\right)\in\ZZ$. We leave the details on how this is done to the reader; they are no less tedious and no more informative than our analysis of the coefficients of $x{{y}\choose{y-d_1}}$ and ${{x}\choose{2}}{{y}\choose{j-2d_1}}$. As a consequence, we deduce that $\eta\left(g_{\frac{j}{d_1}}^{\vec{v}_2^{\frac{j}{d_1}-2}}\right)\in\ZZ$. 

This is the last missing step needed to show that $\eta\left(g_{\frac{j}{d_1}}^{\vec{v}_2^{\frac{j}{d_1}}}\right), \eta\left(g_{\frac{j}{d_2}}^{\vec{v}_3^{\frac{j}{d_2}}}\right)\in\ZZ$. We have thus showed that the elements in (iii) and (v) in the statement of the proof are sent to integers by $\eta$. The argument showing that $\eta$ sends the elements in (ii) and (iv) to $\ZZ$ is very similar: instead of analyzing the coefficients of ${{y}\choose{j}}$, $x{{y}\choose{j-d_1}}$, ${{x}\choose{2}}{{y}\choose{j-2d_1}}$ and ${{x}\choose{2}}{{y}\choose{j-d_1}}$, we would look at the coefficients of $x{{y}\choose{j-1}}$, ${{x}\choose{2}}{{y}\choose{j-1-d_1}}$, ${{x}\choose{3}}{{y}\choose{j-1-2d_1}}$ and ${{x}\choose{3}}{{y}\choose{j-1-d_1}}$. We leave the details to an interested reader.
\end{proof}

%%%
%%%
%%% Leibman group for linear forms
%%%
%%%

\section{The connection with the Leibman group for a system of linear forms}\label{section on linear forms}
As remarked in the introduction, the construction of the group $G^P$ generalizes the construction of Leibman group $G^\Psi$ for a system of linear forms given in Definition 1.10 of \cite{green_tao_2010a}. In this section, we illustrate how the definition of $G^\Psi$ fits into this framework. Let $\vec{\Psi}=(\Psi_1, ..., \Psi_t)$ be a tuple of $t$ linear forms in  variable ${\textbf{x}=(x_1, ..., x_D)}$. 
We observe that
\begin{align}\label{polynomial spaces for linear forms}
    \P_{i,i} = \Span\left\{\frac{\vec{\Psi^i}(\textbf{x})}{i!}: \textbf{x}\in\RR^D\right\} = \Span\left\{\vec{\Psi^i}(\textbf{x}): \textbf{x}\in\RR^D\right\}.
\end{align}
In \cite{green_tao_2010a}, Green and Tao labelled the space in (\ref{polynomial spaces for linear forms}) as $\Psi^{[i]}$. Green and Tao also defined
\begin{align*}
    G^\Psi_j := \langle g^{\vec{v}}: g\in G_i, \vec{v}\in\Psi^{[i]}, i\geqslant j \rangle
\end{align*}
for $j\geqslant 1$, calling $G^\Psi = G^\Psi_0:=G^\Psi_1$ the \emph{Leibman group} for $\vec{\Psi}$. The property $\PP_{i,j}=0$ for $i<j$ implies that $G^\Psi_j = G^P_j$, and so Leibman group for $\vec{\Psi}$ is a special instance of our construction. 

The system $\vec{\Psi}$ satisfies \emph{flag condition} if $\Psi^{[i]}\subseteq \Psi^{[i+1]}$, or equivalently if $\P_{i,i}\subseteq\P_{i+1,i+1}$, for any $i\in\NN_+$. If $\vec{\Psi}$ satisfies the flag condition, then $\vec{\Psi}$ equidistributes by the periodic version of Theorem 1.11 of \cite{green_tao_2010a}, which has been stated as Theorem 4.1 in \cite{candela_sisask_2012}\footnote{The necessity of the flag condition has only been discovered in November 2020 by Daniel Altman. Therefore the journal versions of \cite{green_tao_2010a} and \cite{candela_sisask_2012} do not mention this condition. See \cite{tao_2020} for an extended discussion of how the flag condition comes into play.}.

The reader might want to know whether the flag condition is related in any way to the filtration condition that we have defined in Definition \ref{filtration condition}. It turns out that \emph{any} $\vec{\Psi}$ satisfies the filtration condition, and so these two conditions are unrelated. We prove this in the next two lemmas.

\begin{lemma}\label{nice expression for polynomial spaces of linear forms}
For any $\vec{\Psi}$ and integers $1\leqslant j\leqslant i$, we have
\begin{align*}
    \P_{i,j} = \Psi^{[i]} + ... + \Psi^{[j]}.
\end{align*}
\end{lemma}
\begin{proof}
Let $a_0, ..., a_i$ be rational numbers such that ${{n}\choose{i}} = a_i n^i + ... + a_0$. Then 
\begin{align*}
    {{\vec{\Psi}(\textbf{x})}\choose{i}} = a_i \vec{\Psi}(\textbf{x})^i + ... + a_1 \vec{\Psi}(\textbf{x}) + a_0.
\end{align*}
Since $\vec{\Psi}$ is a linear form, each $\vec{\Psi}^l$ is a homogeneous polynomial of degree $l$. It therefore follows that 
\begin{align*}
    \D_l{{\vec{\Psi}(\textbf{x})}\choose{i}} = a_l \vec{\Psi}(\textbf{x})^l \quad {\rm{and}}\quad \Psi^{[l]} = \Span\left\{\D_l{{\vec{\Psi}(\textbf{x})}\choose{i}}: \textbf{x}\in\RR^D\right\}.
\end{align*}
The lemma follows from the observation that since the polynomials $\D_l{{\vec{\Psi}(\textbf{x})}\choose{i}}$ are homogeneous of distinct degrees, we have
\begin{align*}
    \P_{i,j} = \sum_{l=j}^i \Span\left\{\D_l{{\vec{\Psi}(\textbf{x})}\choose{i}}\right\}.
\end{align*}
\end{proof}
\begin{corollary}
Any $\vec{\Psi}$ satisfies the filtration condition.
\end{corollary}
\begin{proof}
Let $i_1, i_2, j_1, j_2\in\NN_+$. If $i_1 < j_i$ then $\P_{i_1, j_1} = \{0\}$, and so $\P_{i_1, j_1}\cdot\P_{i_2,j_2}\subseteq\P_{i_1+i_2, j_1 +j_2}$ trivially. The same happens if $i_2 < j_2$. We can therefore assume that $i_1\geqslant j_1$ and $i_2\geqslant j_2$. From Lemma \ref{nice expression for polynomial spaces of linear forms} and Corollary 3.4 of \cite{green_tao_2010a}, it follows that
\begin{align*}
    \P_{i_1, j_1}\cdot\P_{i_2, j_2} = \sum_{l_1 = j_1}^{i_1}\Psi^{[l_1]}\cdot \sum_{l_2 = j_2}^{i_2}\Psi^{[l_2]}\subseteq \sum_{l = j_1+j_2}^{i_1+i_2} \Psi^{[l]} = \P_{i_1+i_2, j_1+j_2}.
\end{align*}
\end{proof}

We also record a corollary which describes the spaces $\P_{i,j}$ for systems satisfying the flag condition.
\begin{corollary}
Suppose that $\vec{\Psi}$ satisfies the flag condition. Then
\begin{align*}
    \PP_{i,1} = ... = \PP_{i,i} = \Psi^{[i]}
\end{align*}
for any $i\in\NN_+$.
\end{corollary}
\begin{proof}
By the flag condition, $\Psi^{[l]}\subseteq\Psi^{[i]} =\P_{i,i}$ for every $1\leqslant l\leqslant i$. Therefore
\begin{align*}
    \P_{i,j} = \Psi^{[i]} + ... + \Psi^{[j]} = \Psi^{[i]}
\end{align*}
for any $1\leqslant j\leqslant i$. 
\end{proof}

\section{True complexity of equidistributing progressions}\label{section on true complexity for equidistributing progressions}
 We have shown in previous sections that many progressions, including $x,\; x+y,\; x+y^2,\; x+y+y^2$ or $x,\; x+y,\; x+2y, \; x+y^3,\; x+2y^3$, equidistribute, i.e. if $g$ is highly irrational on $G$, then the corresponding sequence $g^P$ is close to being equidistributed on $G^P$. 
In this section, we shall prove Conjecture \ref{conjecture for true complexity} for all equidistributing progressions.
\begin{theorem}\label{true complexity for equidistributing progressions}
Let $t\in\NN_+$, and fix $1\leqslant l\leqslant t$. Let $\vec{P}=(P_1, ..., P_t)\in\QQ[\textbf{x}]^t$ be a Gowers controllable integral polynomial map that equidistributes and is algebraically independent of degree $s+1$ at $l$. Then the true complexity of $\vec{P}$ at $l$ is at most $s$.
\end{theorem}

The logic of the proof is very similar to the proof of Theorem 7.1 in \cite{green_tao_2010a}, with small modifications that allow us to get a control of weights on different terms of the progression by different  Gowers norms.
\begin{proof}
We let all implied constants in this proof depend on $\vec{P}$, $s$ and $t$ without mentioning the dependence explicitly.

Fix $\epsilon>0$. Since $\vec{P}$ is Gowers controllable, there exists an integer $s_0\geqslant 1$, a threshold $p_0\in\NN$, and a real number $\delta>0$ such that for all primes $p>p_0$,
\begin{align*}
    |\EE_{\textbf{x}\in\FF_p^D}f_1(P_1(\textbf{x}))\cdots f_t(P_t(\textbf{x}))|\leqslant \epsilon
\end{align*}
for all 1-bounded functions $f_1, ..., f_t:\FF_p\to\CC$, at least one of which satisfies $\|f_{i}\|_{U^{s_0+1}}\leqslant\delta$. We let $\mathcal{F}:\RR_+\to\RR_+$ be a growth function depending on $\epsilon$ to be fixed later. If $s\geqslant s_0$, then we are done, so suppose $s<s_0$.

Suppose that $f_1, ..., f_t:\FF_p\to\CC$ are 1-bounded functions, and suppose moreover that $\|f_{l}\|_{U^{s+1}}\leqslant\delta$. We use Lemma \ref{regularity lemma} to find $M=O_{\epsilon, \mathcal{F}}(1)$, a filtered nilmanifold $G/\Gamma$ of degree $s_0$ and complexity $M$, a $p$-periodic, $\mathcal{F}(M)$-irrational sequence $g\in\poly(\ZZ,G_\bullet)$ with $g(0)=1$, and decompositions 
\begin{align}\label{decomposition in the proof of true complexity}
    f_i = f_{i, nil} + f_{i, sml} + f_{i, unf}
\end{align}
satisfying the conditions of Lemma \ref{regularity lemma}. Decomposing each of $f_i$ this way, we get $3^t$ terms. All the terms involving $f_{i,sml}$ can be bounded by $O(\epsilon)$. By choosing $\mathcal{F}$ growing sufficiently fast depending on $\delta$, we can assume that $\|f_{i,unf}\|_{s_0+1}\leqslant\frac{1}{\F(M)}\leqslant\frac{\delta}{4}$, which together with 4-boundedness of $f_{i,unf}$ implies that terms involving $f_{i,unf}$ contribute at most $O(\epsilon)$.
%and the terms involving $f_{i,unf}$ can be bounded by $o_{\mathcal{F}(M)\to\infty}(1)+o_{p\to\infty}$. By picking $\mathcal{F}$ growing sufficiently fast depending on $\epsilon$ and assuming that $p$ is sufficiently large, we can assume that the terms involving $f_{i,unf}$ are of size $O(\epsilon)$. 
This leaves us with
\begin{align*}
    \EE_{\textbf{x}\in\FF_p^D}f_1(P_1(\textbf{x}))\cdots f_t(P_t(\textbf{x})) &=\EE_{\textbf{x}\in\FF_p^D}f_{1,nil}(P_1(\textbf{x}))\cdots f_{t,nil}(P_t(\textbf{x})) + O(\epsilon)\\
    &= \EE_{\textbf{x}\in\FF_p^D}F(g^P(\textbf{x})\Gamma^P)+O(\epsilon),
\end{align*}
where $F((u_1, ..., u_t)\Gamma^P)=F_1(u_1\Gamma)...F_t(u_t\Gamma)$. Since $\vec{P}$ equidistributes, we have
\begin{align*}
    \EE_{\textbf{x}\in\FF_p^D}f_1(P_1(\textbf{x}))\cdots f_t(P_t(\textbf{x}))
    &=\int_{G^P/\Gamma^P}F + o_{\mathcal{F}(M)\to\infty, M, \epsilon}(1)+O(\epsilon).
\end{align*}
By the assumption of algebraic independence, the polynomial ${{P_l}\choose{s+1}}$ is not a linear combination of ${{P_1}\choose{s+1}}, ..., {{P_{l-1}}\choose{s+1}}, {{P_{l+1}}\choose{s+1}}, ..., {{P_t}\choose{s+1}}$. Consequently, the space $\Q_{s+1,1}$ contains the vector $\vec{e}_l$ that has 1 in the $l$-th coordinate and 0 elsewhere. This implies that the group
\begin{align*}
    H = \langle h^{\vec{e}_l}: h\in G_{s+1}\rangle = \{1\}^{l-1}\times G_{s+1}\times\{1\}^{t-l}
\end{align*}
is contained in $G^P$. In fact, $H$ is a normal subgroup of $G^P$ due to the normality of $G_{s+1}$ in $G$. Therefore,
\begin{align*}
    \int_{G^P/\Gamma^P}F = \int_{G^P/\Gamma^P}F_{\leqslant s},
\end{align*}
where $F_{\leqslant s}((u_1, ..., u_t)\Gamma^P)=\left(\prod\limits_{\substack{1\leqslant i \leqslant t, \\ i\neq l}} F_i(u_i\Gamma)\right)F_{l,\leqslant s}(u_l\Gamma)$ and $F_{l,\leqslant s}$ is the average of $F_l$ over cosets of $G_{s+1}$:
\begin{align*}
    F_{l,\leqslant s}(u\Gamma) = \int_{G_{s+1}/\Gamma_{s+1}}F_l(uw\Gamma) dw.
\end{align*}
It is straightforward to see that $F_{l, \leqslant s}$ is 1-bounded and $M$-Lipschitz. We moreover have the bound
\begin{align*}
    |F_{\leqslant s}((u_1, ..., u_t)\Gamma^P)|\leqslant|F_{l, \leqslant s}(u_{l}\Gamma)|
\end{align*}
which implies that
\begin{align*}
    \left|\int_{G^P/\Gamma^P}F\right|\leqslant\int_{G/\Gamma}|F_{l,\leqslant s}|\leqslant\left(\int_{G/\Gamma}|F_{l,\leqslant s}|^2\right)^\frac{1}{2}.
\end{align*}
The function $F_{l, \leqslant s}$ is invariant on $G_{s+1}$-cosets by construction while $F_{l}-F_{l, \leqslant s}$ vanishes on each coset. As a consequence, the two functions are orthogonal, implying
\begin{align*}
    \int_{G/\Gamma}|F_{l,\leqslant s}|^2 = \int_{G/\Gamma}F_{l} \overline{F_{l,\leqslant s}}. 
\end{align*}
By the $\F(M)$-irrationality of $g$, we have
\begin{align*}
    \int_{G/\Gamma}F_{l} \overline{F_{l,\leqslant s}} = \EE_{n\in\FF_p} (F_{l}\overline{F_{l,\leqslant s}})(g(n)\Gamma) + o_{\F(M)\to\infty, M,\epsilon}(1).
\end{align*}
We let $\psi(n)=\overline{F_{l,\leqslant s}}(g(n)\Gamma)$. By the $G_{s+1}$-invariance of $F_{\leqslant s}$, this is a nilsequence of degree $\leqslant s$ and complexity $M$. By (\ref{decomposition in the proof of true complexity}), we have
\begin{align*}
    F_{l}(g(n)\Gamma) = f_{l}(n)-f_{l, sml}(n)-f_{l, unf}(n).
\end{align*}
We then split $ \EE_{n\in\FF_p} F_{l}(g(n)\Gamma)\psi(n)$ into three terms. Using the Cauchy-Schwarz inequality, the term involving $f_{l,sml}$ can be bounded as
\begin{align*}
    |\EE_{n\in\FF_p}f_{l, sml}(n)\psi(n)|\ll\epsilon.
\end{align*}
To evaluate the contribution coming from $f_{l}$, we use $\|f_{l}\|_{U^{s+1}}\leqslant\delta$ and the converse to the inverse theorem for Gowers norms (Proposition 1.4 of Appendix G of \cite{green_tao_ziegler_2011}) to conclude that
\begin{align*}
    |\EE_{n\in\FF_p}f_{l}(n)\psi(n)|=o_{\delta\to 0, M,\epsilon}(1).
\end{align*}
Similarly, we use $\|f_{l,unf}\|_{U^{s_0+1}}\leqslant\delta$ and $s_0\geqslant s$ to conclude that 
\begin{align*}
    |\EE_{n\in\FF_p}f_{l, unf}(n)\psi(n)|=o_{\mathcal{F}(M)\to\infty, M, \epsilon}(1).
\end{align*}
Combining all these estimates, we have
\begin{align*}
    |\EE_{\textbf{x}\in\FF_p^D}f_1(P_1(\textbf{x}))\cdots f_t(P_t(\textbf{x}))|= O(\epsilon)+o_{\mathcal{F}(M)\to\infty, M, \epsilon}(1)+o_{\delta\to 0, M,\epsilon}(1).
\end{align*}
By choosing $\mathcal{F}$ growing sufficiently fast and $\delta$ sufficiently small depending on $\epsilon$, we obtain
\begin{align*}
    |\EE_{\textbf{x}\in\FF_p^D}f_1(P_1(\textbf{x}))\cdots f_t(P_t(\textbf{x}))| \ll \epsilon,
\end{align*}
which proves the theorem.
\end{proof}

%%%
%%%
%%% An asymptotic for the count of progressions of complexity 1
%%%
%%%

\section{An asymptotic for the count of progressions of complexity 1}\label{section on the asymptotic count}
One of the applications of true complexity is that we can obtain an asymptotic for the count of polynomial progressions of complexity 1 like those in Theorem \ref{asymptotic count for systems of complexity 1}. We rewrite the integral polynomial map $\vec{P}\in\QQ[\textbf{x}]^t$ in the form
\begin{align*}
    \vec{P}(\textbf{x})=\sum_{i=1}^r\vec{v}_i Q_i(\textbf{x})
\end{align*}
for some $\vec{v}_1, ..., \vec{v}_r\in\ZZ^t$ and integer-valued $Q_1, ..., Q_r\in\QQ[\textbf{x}]$. Given such a polynomial map, we define the corresponding linear map
\begin{align*}
 \vec{\Psi}(y_1, ..., y_r) = \sum_{i=1}^r\vec{v}_i y_i
\end{align*}
 The relationship between the two progressions is given by 
\begin{align}\label{relationship between P and Psi}
    \vec{P}(\textbf{x})=\vec{\Psi}(Q_1(\textbf{x}), ..., Q_r(\textbf{x})),
\end{align}
and we aim to understand the relationship between the appropriate counts
\begin{align*}
    \Lambda_P(f_1, ..., f_{t})=\EE_{\textbf{x}\in\FF_p^D}\prod_{k=1}^t f_k \left(P_k(\textbf{x})\right)
\end{align*}
and 
\begin{align*}
    \Lambda_\Psi(f_1, ..., f_{t})=\EE_{y_1, ..., y_r\in\FF_p}\prod_{k=1}^t f_k \left(\Psi_k(y_1, ..., y_r)\right)
\end{align*}
where $P_k$ and $\psi_k$ denote the $k$-th coordinates of $\vec{P}$ and $\vec{\Psi}$ respectively.
%$P_k(\textbf{x}) = \sum\limits_{i=1}^r\vec{v}_i(k) Q_i(\textbf{x})$ and $\psi_k(y_1, ..., y_r) = \sum\limits_{i=1}^r\vec{v}_i(k) y_i$.
\begin{theorem}\label{asymptotic for the count of configurations}
Let $\vec{P}$ and $\vec{\Psi}$ be given as above. Suppose moreover that $\vec{P}$ is Gowers controllable, equidistributes and is algebraically independent of degree 2. Then
\begin{align*}
     \Lambda_P(f_1, ..., f_{t}) = \Lambda_\Psi(f_1, ..., f_{t})  + o(1)
\end{align*}
for an error term $o(1)$ that depends on $\vec{P}$ but not on the choice of 1-bounded functions ${f_1, ..., f_{t}:\FF_p\to\CC}$. 
\end{theorem}

\begin{corollary}\label{corollary to asymptotic for the count of configurations}
Let $\vec{P}$ and $\vec{\Psi}$ be given as above, and suppose that $\vec{P}$ is Gowers controllable, equidistributes and is algebraically independent of degree 2. For any $A\subseteq\FF_p$, we have
\begin{align*}
    |\{\vec{P}(\textbf{x})\in A^t: \textbf{x}\in\FF_p^D\}|=p^{D-r}|\{\vec{\Psi}(y_1, ..., y_r)\in A^t: y_1, ..., y_r\in\FF_p\}|+o(p^D),
\end{align*}
and the error term is uniform in all subsets $A$.
\end{corollary}
What Theorem \ref{asymptotic for the count of configurations} is indicating is that for configurations satisfying only linear relations and of true complexity 1, each polynomial $Q_i$ can be thought of as a separate variable. Therefore the counts of $\vec{P}$ and $\vec{\Psi}$ are so strongly related.

We specialize Corollary \ref{corollary to asymptotic for the count of configurations} to two families of polynomial progressions that we have explicitly looked at.
\begin{corollary}\label{asymptotic count for x, x+Q(y), x+R(y), x+Q(y)+R(y)}
Let $Q,R\in\ZZ[y]$ be nonzero polynomials that have zero constant terms and satisfy $1\leqslant\deg Q<\deg R$.
For any 1-bounded functions $f_0, f_1, f_2, f_3:\FF_p\to\CC$, we have 
\begin{align*}
    &\EE_{x,y\in\FF_p}f_0(x)f_1(x+Q(y))f_2(x+R(y))f_3(x+Q(y)+R(y))\\ 
    &= \EE_{x,y,z\in\FF_p}f_0(x)f_1(x+y)f_2(x+z)f_3(x+y+z) + o(1),
\end{align*}
where the error term is independent of the choice of $f_0, f_1, f_2, f_3$. Moreover, for any $A\subseteq\FF_p$, we have
\begin{align*}
    &|\{(x,x+Q(y), x+R(y), x+Q(y)+R(y))\in A^4: x,y\in\FF_p\}|\\ 
    &= \frac{1}{p}\{(x,y,u,z)\in A^4: x+y=u+z\}| + o(p^2)
\end{align*}
uniformly in the choice of $A$.
\end{corollary}
We have thus related the number of progressions $x,\; x+Q(y),\; x+R(y),\; x+Q(y)+R(y)$ in an arbitrary subset $A\subseteq\FF_p$ to the number of solutions to the Sidon equation $x+y=u+z$, which is a well-studied quantity known as additive energy. To learn more about Sidon equation or additive energy, consult e.g. \cite{tao_vu_2006}. 
\begin{proof}
The first part of Corollary \ref{asymptotic count for x, x+Q(y), x+R(y), x+Q(y)+R(y)} is a straightforward application of Theorem \ref{asymptotic for the count of configurations}. The second part follows by observing that 2-dimensional cubes $x,\; x+y,\; x+z,\; x+y+z$ parametrize solutions to the Sidon equation.
\end{proof}
\begin{corollary}\label{asymptotic count for x, x+Q(y), x+2Q(y), x+R(y), x+2R(y)}
Let $Q,R\in\ZZ[y]$ be nonzero polynomials that have zero constant terms and satisfy $1\leqslant\deg Q<(\deg R)/2$. For any 1-bounded functions $f_0, f_1, f_2, f_3, f_4:\FF_p\to\CC$, we have 
\begin{align*}
    &\EE_{x,y\in\FF_p}f_0(x)f_1(x+Q(y))f_2(x+2Q(y))f_3(x+R(y))f_4(x+2R(y))\\
    &= \EE_{x,y,z\in\FF_p}f_0(x)f_1(x+y)f_2(x+2y)f_3(x+z)f_4(x+2z) + o(1),
\end{align*}
and the error term is independent of the choice of $f_0, f_1, f_2, f_3, f_4$.
Moreover, for any $A\subseteq\FF_p$, we have
\begin{align*}
    &|\{(x,\; x+Q(y),\; x+2Q(y),\; x+R(y),\; x+2R(y))\in A^5: x,y\in\FF_p\}|\\
    &= \frac{1}{p}\{(x,\; x+ y,\; x+2y,\; x+z,\; x+2z)\in A^5: x,y,z\in\FF_p\}| + o(p^2)
\end{align*}
uniformly in the choice of $A$.
\end{corollary}
Corollaries \ref{asymptotic count for x, x+Q(y), x+R(y), x+Q(y)+R(y)} and \ref{asymptotic count for x, x+Q(y), x+2Q(y), x+R(y), x+2R(y)} together imply Theorem \ref{asymptotic count for systems of complexity 1}.

\begin{proof}[Proof of Theorem \ref{asymptotic for the count of configurations}]
We adopt the notation from \cite{green_tao_2010a} to set
\begin{align*}
    \Psi^{[i]} = \Span\{\vec{\Psi}^k(y_1, ..., y_r): 1\leqslant k \leqslant i,\; y_1, ..., y_r\in\ZZ\}
\end{align*}
to be the analogue of $\PP_{i,j}$ for the progression $\vec{\Psi}$ for $1\leqslant j \leqslant i$. We also let $G^\Psi$ denote the Leibman group for $\Vec{\Psi}$.

By assumption, the squares $P_1(\textbf{x})^2, ..., P_t(\textbf{x})^2$ are linearly independent, implying that $\PP_{2,1}=\RR^t$. From (\ref{relationship between P and Psi}), it follows that $\PP_{1,1}=\Psi^{[1]}$ and $\PP_{i,1}\subseteq\Psi^{[i]}$ for $i>1$. Together with the fact that $\PP_{2,1} = \RR^t$, this implies that $\Psi^{[2]}=\PP_{2,1}$, and so the groups $G^P=G^\Psi$ are in fact the same for any group $G$.

Given $\epsilon>0$, we take $\delta>0$ and $p_0\in\NN$ that works as in Theorem \ref{true complexity for equidistributing progressions} for both $\vec{P}$ and $\vec{\Phi}$, and we let $\mathcal{F}:\RR_+\to\RR_+$ be a growth function to be fixed later. We moreover assume from now on that $p>p_0$.  By Lemma \ref{regularity lemma}, there exist $M=O_{\epsilon, t,\mathcal{F}}(1)$, a filtered nilmanifold $G/\Gamma$ of degree $1$ and complexity $M$, and a $p$-periodic, $\mathcal{F}(M)$-irrational sequence $g\in\poly(\ZZ,G_\bullet)$ with $g(0)=1$ such that there exist decompositions
\begin{align*}
    f_i = f_{i, nil} + f_{i,sml} + f_{i,unf}
\end{align*}
of functions $f_1, ..., f_{t}:\FF_p\to\CC$ satisfying the conditions of Lemma \ref{regularity lemma}.
By taking $\mathcal{F}$ growing fast enough with respect to $\delta$, we can assume that $\|f_{i,unf}\|_{U^2}\leqslant \frac{1}{\mathcal{F}(M)}\leqslant\delta/4$ for each $i$.

By applying the aforementioned decomposition to $f_1, ..., f_{t}$, each of the operators $\Lambda_P(f_1, ..., f_{t})$ and $\Lambda_\Psi(f_1, ..., f_{t})$ splits into $3^t$ terms. The expressions involving at least one $f_{i, sml}$ can be bounded crudely by $O(\epsilon)$. Using Theorem \ref{true complexity for equidistributing progressions}, the expressions involving at least one $f_{i, unf}$ can be bounded by $O(\epsilon)$ as well. We thus have
\begin{align*}
    \Lambda_P(f_1, ..., f_{t})=\Lambda_P(f_{1,nil}, ..., f_{t, nil})+O(\epsilon),
\end{align*}
and similarly for $\Lambda_\Psi(f_1, ..., f_{t})$. Since both $\vec{P}$ and $\vec{\Psi}$ equidistribute, we have
\begin{align*}
    \Lambda_P(f_{1,nil}, ..., f_{t, nil})=\int_{G^P/\Gamma^P}F + o_{\mathcal{F}(M)\to\infty, M, \epsilon}(1),
\end{align*}
where $F((u_1, ..., u_{t})\Gamma^P)=F_1(u_1\Gamma)...F_{t}(u_{t}\Gamma)$. Likewise, we have
\begin{align*}
    \Lambda_P(f_{1,nil}, ..., f_{t, nil})=\int_{G^\Psi/\Gamma^\Psi}F + o_{\mathcal{F}(M)\to\infty, M, \epsilon}(1).
\end{align*}
Using the fact that $G^P=G^\Psi$ and combining all the estimates so far, we obtain that 
\begin{align*}
    \Lambda_P(f_1, ..., f_{t}) = \Lambda_\Psi(f_1, ..., f_{t}) + O(\epsilon) + o_{\mathcal{F}(M)\to\infty, M, \epsilon}(1).
\end{align*}
The theorem follows by letting $\mathcal{F}$ grow sufficiently fast with respect to $\epsilon$, and by taking $\epsilon\to 0$ as $p\to\infty$.
\end{proof}

Note that the only two facts that we use in the proof of Theorem \ref{asymptotic for the count of configurations} is that the progressions $\vec{P}$ and $\vec{\Psi}$ are controlled by \emph{some} Gowers norm (so that we can apply regularity lemma) and that the Leibman groups $G^P$ and $G^\Psi$ are the same. It is the latter fact that follows from the algebraic independence of degree 2 of $\vec{P}$. We do not strictly require the information that $\vec{P}$ and $\vec{\Psi}$ are controlled by the $U^2$ norm.

%%%
%%%
%%% A progression not satisfying filtration condition
%%%
%%%

\section{A progression not satisfying filtration condition}\label{section on a progression not satisfying filtration condition}
While many naturally defined polynomial progressions satisfy filtration condition, one can also find a configuration for which the condition fails. We present one such example in this section. Let
\begin{align*}
    \vec{P}(x,y) &= (x,\; x+y+y^2+y^3,\; x+y^2+2y^3,\; x+y^2+3y^3,\; x+y^2+4y^3)\\
    &= (1,1,1,1,1) x + (0,3,3,4,5) y + (0,8,14,20,26) {{y}\choose{2}} + (0,6,12,18,24) {{y}\choose{3}}.
\end{align*}
It is straightforward to deduce that 
\[ \PP_{1,j}=\begin{cases}
\Span\{(1,1,1,1,1), (0,1,0,0,0), (0,1,1,1,1), (0,1,2,3,4)\},\; &j=1\\
\Span\{(0,1,1,1,1), (0,1,2,3,4)\},\; &j=2\\
\Span\{(0,1,2,3,4)\},\; &j=3\\
0,\; &j\geqslant 4.
\end{cases}
\]
We claim that $\PP_{1,1}\cdot\PP_{1,3}\notin\PP_{2,4}$. To obtain $\PP_{2,4}$, we write
\begin{align*}
    {{\vec{P}(x,y)}\choose{2}} &=(1,1,1,1,1){{x}\choose{2}} + (0,3,3,6,10) y + (0,3,3,4,5)xy\\ &+(0,85,184,366,610){{y}\choose{2}}+ (0, 8, 14, 20, 26)x{{y}\choose{2}}\\
    &+(0, 477, 1392, 2889, 4926) {{y}\choose{3}}
    + (0, 6, 12, 18, 24) x{{y}\choose{3}}\\
    &+ (0, 1056, 3612, 7752, 13452) {{y}\choose{4}} + (0, 1020, 3840, 8460, 14880) {{y}\choose{5}}\\
    &+ (0, 360, 1440, 3240, 5760) {{y}\choose{6}}.
\end{align*}
From the fact that
\begin{align*}
    (0, 360, 1440, 3240, 5760) &= 360\cdot(0,1,4,9,16) \\ (0, 1020, 3840, 8460, 14880) &= 900\cdot(0,1,4,9,16) + 120\cdot (0,1,2,3,4) \\ (0, 1056, 3612, 7752, 13452) &= 780\cdot(0,1,4,9,16) + 120\cdot (0,1,2,3,4) + 12\cdot (0,3,1,1,1) \\ (0, 6, 12, 18, 24) &=6\cdot (0,1,2,3,4),
\end{align*}
we deduce that 
\begin{align*}
    \PP_{2,4}=\Span\{(0,1,4,9,16), (0,1,2,3,4), (0,3,1,1,1)\}.
\end{align*}
From the description of $\PP_{1,j}$ above, we have that $\vec{v}=(0,1,0,0,0)\in\PP_{1,1}$ and $\vec{w}=(0,1,2,3,4)\in\PP_{1,3}$, however the product $\vec{v}\cdot\vec{w}=(0,1,0,0,0)$ is not contained in $\PP_{2,4}$. Therefore the progression $\vec{P}$ does not satisfy the filtration condition.

%%%
%%%
%%% Failure of equidistribution for x, x+y, x+2y, x+y^2
%%%
%%%

\section{Failure of equidistribution for $x,\; x+y,\; x+2y,\; x+y^2$}\label{section on failure of x, x+y, x+2y, x+y^2}

Failure to satisfy filtration condition is one reason why it may be hard to work with Leibman group for more general polynomial progressions. Perhaps more interestingly, we can find progressions that satisfy filtration condition, yet they do not equidistribute on the Leibman nilmanifold. In particular, these arguments break for the configuration
\begin{align*}
    \vec{P}(x,y) = (x,\; x+y,\; x+2y,\; x+y^2) = (x,\; x+y,\; x+2y,\; x+y+2{{y}\choose{2}}).
\end{align*}
For this progression, we have
\[ \PP_{1,j}=\begin{cases}
\Span\{(1,1,1,1), (0,1,2,1), (0,0,0,1)\},\; &j=1\\
0 \times 0 \times 0 \times \RR,\; &j=2\\
0,\; &j \geqslant 3,
\end{cases}
\]
and we can moreover prove the following.
\begin{lemma} \label{structure of P for x, x+y, x+2y, x+y^2}
For $i\geqslant 2$, we have
\[ \PP_{i,j}=\begin{cases}
\RR^4,\; &1\leqslant j\leqslant i\\
0 \times 0 \times 0 \times \RR,\; &i+1 \leqslant j\leqslant 2i \\
0,\; &j > 2i,
\end{cases}
\]
\end{lemma}
\begin{proof}
The case $j> 2i$ follows from the fact that $\deg({{\vec{P}(x,y)}\choose{i}})=2i$. For $i+1 \leqslant j\leqslant 2i$, we note that ${{x}\choose{i}}$, ${{x+y}\choose{i}}$ and ${{x+2y}\choose{i}}$ all have degree $i$, and so $\PP_{i,j}\subseteq 0\times 0 \times 0 \times \RR$. That this is equality follows from the fact that the coefficient of ${{y}\choose{2}}$ is a nonzero multiple of the vector $(0,0,0,1)$. The case $1\leqslant j\leqslant i$ follows from the fact that
\begin{align}
    {{\vec{P}(x,y)}\choose{2}} &= (1,1,1,1){{x}\choose{2}} + (0,1,2,1)xy+(0,1,4,6){{y}\choose{2}}+(0,0,0,1)y \\
    \nonumber
    &+ (0,0,0,2)x{{y}\choose{2}} + (0,0,0,18){{y}\choose{3}}+(0,0,0,12){{y}\choose{4}}
\end{align}
and $\PP_{i,j}\supseteq\PP_{2,j}$ for $i\geqslant 2$ by Lemma \ref{properties of polynomial spaces}.
\end{proof}
%The following lemma is an immediate corollary of Lemma \ref{structure of P for x, x+y, x+2y, x+y^2}.
\begin{corollary}
The polynomial map $\vec{P}$ satisfies the filtration condition.
\end{corollary}
To prove that progressions in Sections \ref{section on x, x+y, x+y^2, x+y+y^2}, \ref{section on x, x+Q(y), x+R(y), x+Q(y)+R(y)} and \ref{section on x, x+Q(y), x+2Q(y), x+R(y), x+2R(y)} equidistribute, we showed that if a polynomial sequence $g$ is highly irrational on a nilmanifold $G/\Gamma$, then $g^{{P}}$ is close to being equidistributed on the nilmanifold $G^P/\Gamma^P$. More precisely, we proved the contrapositive: if there exists a nontrivial horizontal character on $G^P/\Gamma^P$ of small modulus that annihilates $g^{{P}}$, then for some $j\geqslant 1$ there must exist a $j$-th level character on $G/\Gamma$ of small modulus that annihilates the $j$-th Taylor coefficient $g_j$ of $g$. It turns out this is not the case for $x,\; x+y,\; x+2y,\; x+y^2$: we can find a highly irrational sequence $g$ on a nilmanifold $G/\Gamma$ such that $g^{{P}}$ is annihilated by a horizontal character of a small modulus.

That our arguments from previous sections would not work here is already clear from Lemma \ref{structure of P for x, x+y, x+2y, x+y^2}. If $\vec{P}$ equidistributed, then the fact that $\PP_{2,1}$ is all of $\RR^4$ would imply that the sequence $(x,y)\mapsto g^P(x,y)G_2^4$ would be close to being equidistributed on the 1-step nilmanifold $G/G_2\Gamma$ for any highly irrational $p$-periodic sequence $g$, and so we would expect $\vec{P}$ to be of complexity 1. We know by Theorem \ref{lower bound for true complexity} that this cannot possibly happen because of the quadratic relation (\ref{quadratic relation satisfied by x, x+y, x+2y, x+y^2}). In the ergodic theoretic language, this instantiates the fact that the \emph{Vandermonde complexity} and the \emph{Weyl complexity} of the progression are different \footnote{\emph{Vandermonde complexity} of $\vec{P}\in\QQ[\textbf{x}]^t$ is the smallest $i$ such that $\PP_{i,1}=\RR^t$; \emph{Weyl complexity} is the smallest $i$ such that $\vec{P}$ is algebraically independent of degree $i$. In our case, the Vandermonde complexity is 2 but Weyl complexity is 3. Both of these concepts have been defined and discussed in \cite{bergelson_leibman_lesigne_2007}.}.

\begin{lemma}\label{example for failure of equidistribution}
There exists a degree-2 filtered nilmanifold $G/\Gamma$ of complexity $O(1)$, a {$p^\frac{1}{2}$-irrational} sequence $g\in\poly(\ZZ, G_\bullet)$, and a horizontal character $\eta:G^P\to\RR$ of modulus $O(1)$ such that $\eta\circ g^P = 0$.
\end{lemma}
\begin{proof}
We take $G=\RR \times \RR$ and $\Gamma=\ZZ\times\ZZ$ with the degree-2 filtration given by
\begin{align*}
    G_0=G_1=\RR\times\RR, \quad G_2 = 0\times\RR, \quad G_3 = 0\times 0. 
\end{align*}
Let $\alpha = \lfloor \sqrt{p}\rfloor/p$. We define a sequence $g\in\poly(\ZZ, G_\bullet)$ by setting $g_1=(\alpha, 0)$, $g_2 = (0, \alpha)$, so that $g(n) = (\alpha n, \alpha {{n}\choose{2}})$. It is straightforward to see that $g$ is indeed $p^\frac{1}{2}$-irrational.

Having established irrationality of $g$, we shall construct a horizontal character on $G^{P}$ of bounded modulus that annihilates $g^P$. Let $(x,y,u,z)$ denote an arbitrary element of $G^4$, where $x=(x_1, x_2)$ and similarly for $y,u,z$. We define $\eta: G^4\to \RR$ by setting
\begin{align*}
    \eta(x,y,u,z) = (x_1 - y_1 + u_1 - z_1) + (x_2 - 2y_2 + u_2).
\end{align*}
The function $\eta$ defines a horizontal character on $G^4$, and by abuse of notation we use $\eta$ to denote its restriction to $G^{P}$. It is clear that $|\eta|\ll 1$. We claim that $\eta$ annihilates $g^P$.  Expanding $\eta\circ g^P$, we obtain
\begin{align*}
%\label{expansion for x, x+y, x+2y, x+y^2}
    \eta\circ g^P(x,y) &= \eta(g_1^{(1,1,1,1)}) x + (\eta(g_1^{(0,1,2,1)}) + \eta(g_2^{(0,0,0,1)}))y + \eta(g_2^{(1,1,1,1)}){{x}\choose{2}}\\
    &+ \eta(g_2^{(0,1,2,1)}) xy + (\eta(g_1^{(0,0,0,2)}) + \eta(g_2^{(0,1,4,6)})) {{y}\choose{2}} + \eta(g_2^{(0,0,0,2)}) x{{y}\choose{2}}\\
    &+ \eta(g_2^{(0,0,0,18)}) {{y}\choose{3}} + \eta(g_2^{(0,0,0,12)}) {{y}\choose{4}}.
\end{align*}
Because of the way we defined $\eta$, we see that it annihilates $g_1^{(1,1,1,1)}$ because
\begin{align*}
    \eta(g_1^{(1,1,1,1)}) = \alpha - \alpha + \alpha - \alpha = 0.
\end{align*}
Other terms of the polynomial $\eta\circ g^P$ are annihilated for similar reasons, with one interesting exception: the coefficient of ${{y}\choose{2}}$. The function $\eta$ annihilates neither $g_1^{(0,0,0,2)}$ nor $g_2^{(0,1,4,6)}$, but it does annihilate their product, and from this it follows that $\eta\circ g^P =0$. This is the point where the argument from Theorems \ref{equidistribution of x, x+y, x+y^2, x+y+y^2}, \ref{equidistribution of x, x+Q(y), x+R(y), x+Q(y)+R(y)} and \ref{equidistribution of x, x+Q(y), x+2Q(y), x+R(y), x+2R(y)} breaks; we can no longer conclude that nontriviality of $\eta$ implies irrationality of a Taylor coefficient of $g$, which was a crucial step in obtaining contradictions in Theorems \ref{equidistribution of x, x+y, x+y^2, x+y+y^2}, \ref{equidistribution of x, x+Q(y), x+R(y), x+Q(y)+R(y)} and \ref{equidistribution of x, x+Q(y), x+2Q(y), x+R(y), x+2R(y)}. 

%the 2-nd level character obtained by restricting $\eta$ to $G^P_2$ annihilates the degree 2 terms of $g^P$, which was the crucial step in obtaining contradictions in Theorems \ref{equidistribution of x, x+y, x+y^2, x+y+y^2}, \ref{equidistribution of x, x+Q(y), x+R(y), x+Q(y)+R(y)} and \ref{equidistribution of x, x+Q(y), x+2Q(y), x+R(y), x+2R(y)}. 

This example illustrates that irrationality of $g$ is in general not sufficient to guarantee equidistribution of $g^P$ on $G^P/\Gamma^P$. The main obstruction in our example is that the sequence $g$ is irrational but not \emph{jointly irrational}; that is, there exist a 1-horizontal character $\eta_1$ and a 2-horizontal character $\eta_2$ satisfying $|\eta_1|, |\eta_2|\ll 1$ such that $\eta_1(g_1) + \eta_2(g_2)\in\ZZ$ but $\eta_1(g_1), \eta_2(g_2)\notin\ZZ$. This type of obstruction does not appear if one works with linear forms since each power of a linear form is a homogeneous polynomial of different degree. In the case of general polynomial maps, however, one may get the same monomial coming from different powers of the same polynomial, like ${{y}\choose{2}}$ in Lemma \ref{example for failure of equidistribution}. Therefore some sort of ``joint irrationality" is necessary.

%Instead of equidistributing on the 7-dimensional nilmanifold $G^P/\Gamma^P$, the sequence $g^p$ is contained in the 6-dimensional torus

%While $\eta$ does not annihilate either $g_1^{(0,0,0,2)}$ or $g_2^{(0,1,4,6)}$ individually, it does annihilate their product, and from this it follows that $\eta\circ g^P =0$.

\end{proof}

%%%
%%%
%%% True complexity of x, x+y, ..., x+(m-1)y, x+y^d
%%%
%%%

\section{True complexity of $x,\; x+y,\; ...,\; x+(m-1)y,\; x+y^d$}\label{section on x, ..., x+(m-1)y, x+y^d}
The reasoning presented in Section \ref{section on failure of x, x+y, x+2y, x+y^2} shows that the arguments used to tackle $x,\; x+y,\; x+y^2,\; x+y+y^2$ or $x, \; x+y,\; x+2y,\; x+y^3,\; x+2y^3$ cannot be used for $x,\; x+y,\; x+2y,\; x+y^2$. However, we can circumvent the difficulties and determine true complexity for this and related configurations via a different method. This method comes down to making the progression more homogeneous by replacing it with a longer progressions involving higher number of variables using several applications of the Cauchy-Schwarz inequality.

In this section, we prove Conjecture \ref{conjecture for true complexity} for
\begin{align*}
    x,\; x+y, \; ...,\; x+(m-1)y, \; x+y^d
\end{align*}
whenever $2\leqslant d\leqslant m-1$, the case $d\geqslant m$ being handled quantitatively in \cite{kuca_2019}. We start by proving true complexity for the nonlinear term at index $m$.
\begin{proposition}\label{U^2 control of x, x+y, x+2y, x+y^2}
Let $m, d\in\NN_+$ satisfy $m\geqslant 3$ and $d\geqslant 2$.
Given $\epsilon>0$, there exists $\delta>0$ and $p_0\in\NN$ s.t. for all $p>p_0$, we have
\begin{align*}
    |\EE_{x,y\in\FF_p}f_0(x)f_1(x+y)\cdots f_{m-1}(x+(m-1)y)f_m(x+y^d)|\ll \epsilon
\end{align*}
uniformly for all 1-bounded functions $f_0, ..., f_m:\FF_p\to\CC$ satisfying $\|f_m\|_{U^{\left\lceil\frac{m}{d}\right\rceil}}\leqslant\delta$.
\end{proposition}
We note that one cannot get a control by a lower-degree Gowers norm here; this follows from
$$ \left\lfloor\frac{m-1}{d}\right\rfloor = \left\lceil\frac{m}{d}\right\rceil - 1 $$
and the fact that the the space of polynomials in $x$ and $y$ of degree at most $m-1$ is spanned by polynomials in $x$, $x+y$, ..., $x+(m-1)y$ of degree at most $m-1$, so in particular it contains the $\left\lfloor\frac{m-1}{d}\right\rfloor$-th power of $x+y^d$.
\begin{proof}
We let all the constants depend on $m$ and $d$ without mentioning the dependence explicitly. We only prove the case $2\leqslant d\leqslant m-1$, as the case $d\geqslant m$ has been handled in \cite{kuca_2019}.

By Proposition 2.2 of \cite{peluse_2019b}, we have
\begin{align*}
    |\EE_{x,y\in\FF_p}f_0(x)f_1(x+y)\cdots f_{m-1}(x+(m-1)y)f_m(x+y^d)|\leqslant\|f_m\|_{U^{s+1}}^c+O(p^{-c})
\end{align*}
for some $c>0$ and $s\in\NN$ independent of the choice of 1-bounded functions $f_0, ..., f_m:\FF_p\to\CC$.

Let $\mathcal{F}:\RR_+\to\RR_+$ be a growth function to be fixed later. By Lemma \ref{regularity lemma}, 
there exist $M=O_{\epsilon,\mathcal{F}}(1)$, a filtered manifold $G/\Gamma$ of degree $s$ and complexity at most $M$, and a $p$-periodic, $\mathcal{F}(M)$-irrational sequence $g\in\poly(\ZZ,G_\bullet)$ with $g(0)=1$, for which there exists a decomposition
\begin{align*}
    f_m = f_{nil} + f_{sml} + f_{unf}
\end{align*}
such that $f_{nil}(n)=F(g(n)\Gamma)$ for an $M$-Lipschitz function $F: G/\Gamma\to\CC$, ${\|f_{sml}\|_2\leqslant \epsilon}$ and $\|f_{unf}\|_{U^{s+1}}\leqslant \frac{1}{\mathcal{F}(M)}$. By picking $\mathcal{F}$ to be growing sufficiently fast, we can assume that $\|f_{unf}\|_{U^{s+1}}\leqslant \epsilon^\frac{1}{c}$. Assuming that $p$ is large enough with respect to $\epsilon$, we thus have
\begin{align}\label{rewriting the operator for x, x+y, x+2y, x+y^2, 1}
    &\EE_{x,y\in\FF_p}f_0(x)f_1(x+y)\cdots f_{m-1}(x+(m-1)y)f_m(x+y^d)\\
    \nonumber
    &= \EE_{x,y\in\FF_p}f_0(x)f_1(x+y)\cdots f_{m-1}(x+(m-1)y)F(g(x+y^d)\Gamma) + O(\epsilon).
\end{align}

By applying the triangle inequality and translating $x\mapsto x-y$ exactly $m$ times to remove $f_0, f_1, ..., f_{m-1}$, we have
\begin{align}\label{Cauchy-Schwarzing weights}
   &|\EE_{x,y\in\FF_p}f_0(x)f_1(x+y)\cdots f_{m-1}(x+(m-1)y)F(g(x+y^d)\Gamma)|^{2^m}\\
   \nonumber
   &\leqslant\EE_{x,y,h_1,...,h_m\in\FF_p}\prod_{w\in\{0,1\}^m}\C^{|w|}F(g({\epsilon_w})\Gamma),
\end{align}
where
\begin{align*}
    \epsilon_w(x, y, h_1, ..., h_m) = x+ \left(y + \sum_{i=1}^m w_i h_i\right)^d - \sum_{i=1}^m (i-1) w_i h_i
\end{align*}
for each $w\in\{0,1\}^m$. Given $w\in\{0,1\}^m$, we let $\vec{e}_w$ denote the basis vector in $\RR^{\{0,1\}^m}$ of the form
\[ \vec{e}_w(w') = \begin{cases}
1,\; w' = w\\
0,\; w' \neq w.
\end{cases}
\]

Let \begin{align*}
\vec{P}(x,y,h_1,...,h_m) = (\epsilon_w(x,y,h_1,...,h_m))_{w\in\{0,1\}^3}    
\end{align*}
and $G^P$ be the corresponding Leibman group. The next lemma gives the structure of the polynomial spaces $\PP_{i,j}$.
\begin{lemma}\label{structure of polynomial space for x, x+y, x+2y, x+y^2}
For each $i\in\NN_+$ and $1\leqslant j\leqslant id$, the space $\PP_{i,j}$ is spanned by the vectors
\begin{align*}
%\label{spanning vectors}
    \sum_{w}(-1)^{|w|}\vec{e}_w, \sum_{w: w_{k_1}=1}(-1)^{|w|}\vec{e}_w, ..., \sum_{w: w_{k_1}=...=w_{k_{id}}=1}(-1)^{|w|}\vec{e}_w
\end{align*}
for all $k_1, ..., k_{id}\in\{1, ..., m\}$. For $j>id$, we have $\PP_{i,j}=0$.
\end{lemma}
\begin{proof}
The case $j>id$ is easy to see from the fact that ${{\vec{P}(x,y)}\choose{i}}$ has degree $id$, and so we proceed to the other case. The vector $\sum_{w}(-1)^{|w|}\vec{e}_w$ is in $\PP_{i,id}$ because its integer multiple is the coefficient of ${{y}\choose{id}}$.
To see that each vector of the form $\sum\limits_{\substack{w: \\ w_{k_1} = ...= w_{k_n}=1}}(-1)^{|w|}\vec{e}_w$ is in $\PP_{i,id}$ for $1\leqslant n\leqslant id$ and $k_1, ..., k_n\in\{1, ..., m\}$, we observe that the coefficient of 
\begin{align*}
    {{h_{k_1}}\choose{id+1-n}}h_{k_2}...h_{k_n}
\end{align*}
is a nonzero integer multiple of $\sum\limits_{\substack{w: \\ w_{k_1} = ...= w_{k_n}=1}}(-1)^{|w|}\vec{e}_w$ and use Lemma \ref{integer multiples don't matter}. To show the converse, we note that the coefficient of a monomial of ${{\vec{P}(x,y)}\choose{i}}$ is an integer multiple of $\sum\limits_{\substack{w: w_{k_1} = ...= w_{k_n}=1}}(-1)^{|w|}\vec{e}_w$ for $0\leqslant n\leqslant id$ if and only if the monomial contains the variables $h_{k_1}, ..., h_{k_l}$ but does not contain $h_k$ for $k\in\{1, ..., m\}\setminus{\{k_1, ..., k_n\}}$. 
\end{proof}

\begin{corollary}
The progression $\vec{P}$ satisfies the filtration condition.
\end{corollary}
\begin{proof}
This follows from Lemma \ref{structure of polynomial space for x, x+y, x+2y, x+y^2} and the observation that
\begin{align*}
    \sum_{w: w_{k_1}=...=w_{k_{n_1}}=1}(-1)^{|w|}\vec{e}_w \cdot \sum_{w: w_{k'_1}=...=w_{k'_{n_2}}=1}(-1)^{|w|}\vec{e}_w = \sum_{\substack{w: w_{k_1}=...=w_{k_{n_1}}\\
    =w_{k'_1}=...=w_{k'_{n_2}} = 1}}(-1)^{|w|}\vec{e}_w
\end{align*}
for any $k_1, ..., k_{n_1}, k'_1, ..., k'_{n_2}\in\{1, ..., m\}$.
\end{proof}

\begin{corollary}
If $i\geqslant\frac{m}{d}$, then $\PP_{i,1}=...=\PP_{i,id}=\RR^{\{0,1\}^m}$.
\end{corollary}
\begin{proof}
We first observe that the set
\begin{align}\label{spanning vectors}
    X_i= \left\{\sum_{w: w_{k_1}=...=w_{k_{n}}=1}(-1)^{|w|}\vec{e}_w: {\{k_1, ..., k_n\}\subseteq\{1, ...,m\}},\; n\leqslant id\right\},
\end{align}
spans $\PP_{i,1}=...=\PP_{i,id}$ and consists of linearly independent vectors as long as $id\leqslant m$. If $id\geqslant m$, then $X$ has $2^m$ elements, implying that $\PP_{i,1}=...=\PP_{i,id}=\RR^{\{0,1\}^m}$, as required.
\end{proof}
This leads to the following important corollary which we shall need to prove the that the sequence $g^P$ is close to being equidistributed on $G^P$.
\begin{corollary}
Let $i=\left\lceil\frac{m}{d}\right\rceil$. Then $G_i^{\{0,1\}^m}\subseteq G^P$.
\end{corollary}

\begin{theorem}\label{equidistribution of multiparameter sequence coming from x, x+y, x+y^2, x+y+y^2}
The sequence $g^P\in\poly(\ZZ^{m+2},G^P_\bullet)$ is $O_{M}(\mathcal{F}(M)^{-c_M})$-equidistributed.
%for some $c_M>0$.
\end{theorem}
 \begin{proof}%[Proof of Theorem \ref{equidistribution of multiparameter sequence coming from x, x+y, x+y^2, x+y+y^2}]
 Suppose that $g^P\in\poly(\ZZ^{m+2},G^P_\bullet)$ is not $O_{M}(\mathcal{F}(M)^{-c_M})$-equidistributed. By Theorem \ref{equidistribution theorem}, there exists a nontrivial horizontal character $\eta:G^P\to\RR$ of complexity at most $c\mathcal{F}(M)$ for some $c>0$ to be chosen later, such that $\eta\circ g^P\in\ZZ$. Let $j$ be the largest natural number such that $\eta|_{G^P_j}\neq 0$. By assumption, $\eta$ annihilates ${G^P_{j+1}}$. 
 
 When $j$ is not divisible by $d$, we have $\PP_{i,j}=\PP_{i,j+1}$ for all $i\geqslant 1$, implying that any $j$-th level character is trivial. We can therefore assume without loss of generality that $d$ divides $j$. Moreover, the only $i$ such that $\PP_{i,j}\neq\PP_{i,j+1}$ is $i=\frac{j}{d}$, in which case we have $\PP_{i,j+1}=0$. We therefore fix $i=\frac{j}{d}$.
 Given $\vec{v}\in X_i$, where $X_i$ is defined as in (\ref{spanning vectors}), we let
\begin{align*}
    \xi_{\vec{v}}(h)=\eta(h^{\vec{v}})
\end{align*}
for $h\in G_i$. The map $\xi_{\vec{v}}$ defines an $i$-th level character on $G$  by a straightforward generalization of Corollary \ref{horizontal characters reduce to i-th level characters}. By Lemma \ref{structure of polynomial space for x, x+y, x+2y, x+y^2}, the nontriviality of $\eta$ implies that $\xi_{\vec{v}}$ is nontrivial for at least one $\vec{v}\in X_i$. The bound on the modulus of $\eta$ and the fact that the vectors $\vec{v}_k$ have entries of size $O(1)$ imply that $|\xi_{\vec{v}}|\leqslant A$, provided that the constant $c$ is appropriately chosen.

We claim that $\xi_{\vec{v}}(g_i)\in\ZZ$ for each $\vec{v}\in X_i$. This follows from inspecting the coefficients of ${{y}\choose{id}}$ and
${{h_{k_1}}\choose{id+1-n}}h_{k_2}...h_{k_n}$ for all $k_1, ..., k_n\in\{1, ..., m\}$. They are integer multiples of the vectors $\sum\limits_{w}(-1)^{|w|}\vec{e}_w$ and $\sum\limits_{w: w_{k_1}=...=w_{k_n}=1}(-1)^{|w|}\vec{e}_w$ respectively, and so the claim follows by Lemma \ref{integer multiples don't matter}. Together with the argument from the previous paragraph, this contradicts the $\mathcal{F}(M)$-irrationality of $g$, implying that $g^P$ is $O_{M}(\mathcal{F}(M)^{-c_M})$-equidistributed.

 \end{proof}
Combining (\ref{Cauchy-Schwarzing weights}) with Theorem \ref{equidistribution of multiparameter sequence coming from x, x+y, x+y^2, x+y+y^2}, we see that 

\begin{align}\label{Cauchy-Schwarzing weights 2}
   &|\EE_{x,y\in\FF_p}f_0(x)f_1(x+y)\cdots f_{m-1}(x+(m-1)y)F(g(x+y^d)\Gamma)|^{2^m}\\
   \nonumber
   &\leqslant\int_{G^P/\Gamma^P}\prod_{w\in\{0,1\}^m}\C^{|w|}F(x_w\Gamma) dx_w + o_{\mathcal{F}(M)\to\infty, M, \epsilon}(1).
\end{align}
The rest of the proof follows the logic of the proofs of Theorem \ref{true complexity for equidistributing progressions} and Theorem 7.1 from \cite{green_tao_2010a}.
We let
\begin{align*}
    F_{\leqslant {\left\lceil\frac{m}{d}\right\rceil-1}}(x\Gamma) = \int_{G_{\left\lceil\frac{m}{d}\right\rceil}/\Gamma_{\left\lceil\frac{m}{d}\right\rceil}}F(xy\Gamma)d(y\Gamma)= \int_{xG_{\left\lceil\frac{m}{d}\right\rceil}/\Gamma_{\left\lceil\frac{m}{d}\right\rceil}}F(y\Gamma)d(y\Gamma)
\end{align*}
to be the average of $F$ over the coset of $G_{\left\lceil\frac{m}{d}\right\rceil}/\Gamma_{\left\lceil\frac{m}{d}\right\rceil}$ containing $x\Gamma$. 
%The function $F_0$ can be thought of as an $M$-Lipschitz function over the 1-step nilmanifold $G/(G_2\Gamma)$, which is isomorphic to an $m$-dimensional torus for some $m=\dim G - \dim G_2\leqslant M$.
Using the fact that $G_{\left\lceil\frac{m}{d}\right\rceil}^{\{0,1\}^m}\subseteq G^P$ and the crude bound 
\begin{align*}
   \left|\prod_{w\in\{0,1\}^m}\C^{|w|}F_{\leqslant {\left\lceil\frac{m}{d}\right\rceil}-1}(x_w\Gamma)\right|\leqslant \left|F_{\leqslant {\left\lceil\frac{m}{d}\right\rceil}-1}(x_w\Gamma)\right|,
\end{align*}
we obtain 
\begin{align*}
    \left|\int_{G^P/\Gamma^P}\prod_{w\in\{0,1\}^m}\C^{|w|}F(x_w\Gamma)\right|\leqslant\int_{G/\Gamma}\left|F_{\leqslant {\left\lceil\frac{m}{d}\right\rceil}-1}\right|\leqslant\left(\int_{G/\Gamma}\left|F_{\leqslant {\left\lceil\frac{m}{d}\right\rceil}-1}\right|^2\right)^\frac{1}{2}.
\end{align*}

By the $\mathcal{F}(M)$-irrationality of $g$, we have
\begin{align}\label{expression in the last section 1}
    \int_{G/\Gamma}F \overline{F_{\leqslant {\left\lceil\frac{m}{d}\right\rceil}-1}} = \EE_{n\in\FF_p} \left(F\overline{F_{\leqslant {\left\lceil\frac{m}{d}\right\rceil}-1}}\right)(g(n)\Gamma) + o_{\mathcal{F}(M)\to\infty, M,\epsilon}(1).
\end{align}
We let $\psi(n)=\overline{F_{\leqslant {\left\lceil\frac{m}{d}\right\rceil}-1}}(g(n)\Gamma)$. By the $G_{{\left\lceil\frac{m}{d}\right\rceil}}$-invariance of $F_{\leqslant {\left\lceil\frac{m}{d}\right\rceil}-1}$, this is a nilsequence of degree $\leqslant {\left\lceil\frac{m}{d}\right\rceil}-1$ and complexity $M$. By (\ref{decomposition in the proof of true complexity}), we have
\begin{align*}
    F(g(n)\Gamma) = f_m(n)-f_{sml}(n)-f_{unf}(n).
\end{align*}
We then split the average on the right-hand side of (\ref{expression in the last section 1}) into three terms. By Cauchy-Schwarz inequality, we have
\begin{align*}
    |\EE_{n\in\FF_p}f_{sml}(n)\psi(n)|\ll\epsilon.
\end{align*}
To evaluate the contribution coming from $f_m$, we use $\|f_{m}\|_{U^{{\left\lceil\frac{m}{d}\right\rceil}}}\leqslant\delta$ and the converse to the inverse theorem for Gowers norms (Proposition 1.4 of Appendix G of \cite{green_tao_ziegler_2011}) to conclude that
\begin{align*}
    |\EE_{n\in\FF_p}f_{m}(n)\psi(n)|=o_{\delta\to 0, M,\epsilon}(1).
\end{align*}
Similarly, we use $\|f_{unf}\|_{U^{s+1}}\leqslant\frac{1}{\F(M)}$ and monotonicity of Gowers norms to conclude that 
\begin{align*}
    |\EE_{n\in\FF_p}f_{unf}(n)\psi(n)|=o_{\mathcal{F}(M)\to\infty, M, \epsilon}(1).
\end{align*}
Combining all these estimates, we have
\begin{align*}
    &|\EE_{x,y\in\FF_p}f_0(x)f_1(x+y)\cdots f_{m-1}(x+(m-1)y)F(g(x+y^d)\Gamma)|\\
    &= O(\epsilon)+o_{\mathcal{F}(M)\to\infty, M, \epsilon}(1)+o_{\delta\to 0, M,\epsilon}(1).
\end{align*}
By choosing $\mathcal{F}$ growing sufficiently fast and $\delta$ sufficiently small depending on $\epsilon$, we obtain
\begin{align*}
    |\EE_{x,y\in\FF_p}f_0(x)f_1(x+y)\cdots f_{m-1}(x+(m-1)y)F(g(x+y^d)\Gamma)| \ll \epsilon,
\end{align*}
which proves Proposition \ref{U^2 control of x, x+y, x+2y, x+y^2}.
\end{proof}

The control by a low-degree Gowers norm of the nonlinear term $x+y^d$ is useful in that  when combined with the regularity lemma (Lemma \ref{regularity lemma}), it allows us to replace the function $f_m$ by a low-degree nilsequence $\psi$. Lemma \ref{twisted von Neumann} shows how we can deal with $\psi$ if it has sufficiently low degree.

\begin{lemma}[Twisted generalized von Neumann's lemma]\label{twisted von Neumann}
Let $2\leqslant m\leqslant M$. There exists $c_M>0$ such that for any $\delta_1, \delta_2>0$ and any 1-bounded functions $f_0, ..., f_{m-1}:\FF_p\to\CC$ satisfying
\begin{align*}
    \min_{0\leqslant i\leqslant m-1}\|f_i\|_{U^{m-1}}\leqslant\delta_1 \quad {\rm{and}} \quad \min_{0\leqslant i\leqslant m-1}\|f_i\|_{U^{m}}\leqslant\delta_2,
\end{align*}
the following holds:
\begin{enumerate}
    \item if $\psi(x,y)=F(g(x,y)\Gamma)$ is a $p$-periodic nilsequence of complexity $M$ and degree $m-2$, then
\begin{align*}
    \EE_{x,y\in\FF_p}f_0(x)\cdots f_{m-1}(x+(m-1)y)\psi(x,y)\ll_M \delta_1^{c_M};
\end{align*}
\item if $\psi(x,y)=F(g(x,y)\Gamma)$ is a $p$-periodic nilsequence of complexity $M$ and degree $m-1$, then 
\begin{align*}
    \EE_{x,y\in\FF_p}f_0(x)\cdots f_{m-1}(x+(m-1)y)\psi(x,y)\ll_M \delta_2^{c_M}.
\end{align*}
\end{enumerate}

\end{lemma}
\begin{proof}
Lemma \ref{twisted von Neumann} is a variation of Lemma 4.2 of \cite{green_tao_2010a}, and our proof follows very closely the proof of Lemma 4.2 of \cite{green_tao_2010a}. We proceed by induction on $m$. For $m=2$, the statement $(i)$ is trivial since $\psi$, being a 0-step nilsequence, is just a constant. 

To prove $(ii)$ for $m=2$, we let $\delta>0$ be a parameter to be fixed later. Since $\psi$ is of the form $\psi(x,y)=F(\alpha x+\beta y)$ for some 1-bounded, $M$-Lipschitz function $F:\RR^M/\ZZ^M\to\CC$ and $\alpha,\beta\in \left(\frac{1}{p}\ZZ/\ZZ\right)^M$, we can convolve $F$ with Fej\'{e}r kernel to find a 1-bounded trigonometric polynomial $F_1:\RR^M/\ZZ^M\to\CC$ of degree $O_M(\delta^{-C_M})$ satisfying $\|F-F_1\|_\infty\leqslant\delta$. Details of how this can be done may be found in the proof of Proposition 3.1 of \cite{green_tao_2012}, for instance. It then follows from the pigeonhole principle  that
\begin{align*}
    |\EE_{x,y\in\FF_p}f_0(x)f_1(x+y)\psi(x,y)|\ll_M \delta^{-C_M} |\EE_{x,y\in\FF_p}f_0(x)f_1(x+y)e_p(ax+by)| + \delta
\end{align*}
for some $a,b\in\ZZ$. Incorporating $e_p(ax)$ into $f_0$ and applying Cauchy-Schwarz inequality twice to remove $f_0(x)$ and $e_p(by)$ respectively allows us to bound 
\begin{align*}
    \EE_{x,y\in\FF_p}f_0(x)f_1(x+y)e_p(ax+by)
\end{align*}
from above by $\|f_1\|_{U^2}$, and similar maneuvers also give a bound by $\|f_0\|_{U^2}$; thus
\begin{align*}
    |\EE_{x,y\in\FF_p}f_0(x)f_1(x+y)\psi(x,y)|\ll_M \delta^{-C_M}\delta_2 + \delta.
\end{align*}
Letting $\delta=\delta_2^{c_M}$ for a sufficiently small $0<c_M<1$, we obtain the claim.  

We now assume $m>2$, and we let $\psi(x,y)=F(g(x,y)\Gamma)$ be a nilsequence of complexity $M$ and degree $s\in\{m-2,m-1\}$. Let $\delta>0$. Using the vertical decomposition of $F$ (Section 3 of \cite{green_tao_2012}), we can find a 1-bounded function $F_1:G/\Gamma\to\CC$ that is a linear combination of $O_{M}(\delta^{-C_M})$ functions with \emph{vertical characters}, i.e. functions $f:G/\Gamma\to\CC$ for which there exists a continuous homomorphism $\xi:G_s/\Gamma_s\to\RR/\ZZ$ satisfying $f(g_s u)=e(\xi(g_s))f(u)$ for any $g_s\in G_s$. Using pigeonhole principle, we can thus find a 1-bounded, $M$-Lipschitz function $F_2:G/\Gamma\to\CC$ with a vertical character $\xi: G_s/\Gamma_s\to\CC$ satisfying
\begin{align*}
    &|\EE_{x,y\in\FF_p}f_0(x)\cdots  f_{m-1}(x+(m-1)y)\psi_2(x,y)|\\
    &\ll_M \delta^{-C_M} |\EE_{x,y\in\FF_p}f_0(x)\cdots  f_{m-1}(x+(m-1)y)\psi_2(x,y)| + \delta,
\end{align*}
where $\psi_2(x,y)=F_2(g(x,y)\Gamma)$.

By the Cauchy-Schwarz inequality and change of variables, we have 
\begin{align*}
    &|\EE_{x,y\in\FF_p}f_0(x)\cdots  f_{m-1}(x+(m-1)y)\psi(x,y)|\\
    &\leqslant |\EE_{x,y,h\in\FF_p}\Delta_h f_1(x+y)\cdots \Delta_{(m-1)h}f_{m-1}(x+(m-1)y) \psi_2(x,y+h)\overline{\psi_2(x,y)}|,
\end{align*}
where we recall that $\Delta_h f(x):= f(x+h)\overline{f(x)}$. A straightforward adaptation of the arguments from Section 7 of \cite{green_tao_2012} shows that the function $\Tilde{\psi}_h(x,y)=\psi_2(x,y+h)\overline{\psi_2(x,y)}$ is a nilsequence of complexity $O_M(1)$ and degree $s-1$. Picking $\delta =\delta_2^{c_M}$ for an appropriate value of $0<c_M<1$ and applying inductive hypothesis, we obtain
\begin{align*}
    |\EE_{x,y\in\FF_p}f_0(x)\cdots f_{m-1}(x+(m-1)y)\psi(x,y)| \ll_M \min_{1\leqslant i\leqslant m-1}\EE_{h\in\FF_p}\|\Delta_{ih}f_i\|_{U^{s-1}}^{c_M}.
\end{align*}
An application of the H\"{o}lder inequality and the recursive definition of Gowers norms give
\begin{align*}
    |\EE_{x,y\in\FF_p}f_0(x)\cdots f_{m-1}(x+(m-1)y)\psi(x,y)|
    \ll_M \min_{1\leqslant i\leqslant m-1}\|f_i\|_{U^{s}}^{c'_M}.
\end{align*}
for some $0<c'_M<1$.
A slight modification of the argument gives the same bound in terms of $\|f_0\|_{U^{s}}$, completing the proof of the lemma.

%Since our argument is the same as in \cite{green_tao_2012} except the fact that we have a 2-dimensional nilsequence, we merely outline the main steps and refer the reader to \cite{green_tao_2012} for more details. 

%Let $G^\square=G\times_{G_2} G=\{(g,g')\in G^2: g^{-1}g'\in G_2\}$, and define ${F^\square}_h(u,v)=F_2(\{g(0,1)^h\}u)\overline{F_2(v)}$. By Lemma 7.4 of \cite{green_tao_2012}, the function $F^\square$ is $O_M(1)$-Lipschitz. We observe that the sequence

%$\Tilde{\psi}(x,y)=F(g(x,y+h))\overline{F(g(x,y))}$ is a nilsequence of complexity $O_M(1)$ and degree $s-1$.  

%, we use Fourier decomposition to replace $\psi_2$ by a function of the form $\psi_2(x,y)=e_p(ax + by)$ for some $a,b\in\ZZ$.
\end{proof}

Knowing thanks to Lemma \ref{twisted von Neumann} how to proceed in the special case of $f_m$ being a nilsequence, we now prove the general case. Proposition \ref{U^2 control of x, x+y, x+2y, x+y^2} and Proposition \ref{control by higher degree Gowers norms} below together prove Theorem \ref{True complexity of x, x+y, ..., x+(m-1)y, x+y^d}.

\begin{proposition}\label{control by higher degree Gowers norms}
Let $m,d\in\NN_+$ satisfy $2\leqslant d\leqslant m-1$. Given any $\epsilon>0$, there exists $\delta>0$ and $p_0\in\NN$ s.t. for all $p>p_0$, we have
\begin{align*}
    |\EE_{x,y\in\FF_p}f_0(x)f_1(x+y)\cdots f_{m-1}(x+(m-1)y)f_m(x+y^{d})|\ll \epsilon
\end{align*}
uniformly for all 1-bounded functions $f_0, ...,f_m:\FF_p\to\CC$ such that $\|f_i\|_{U^{s}}\leqslant\delta$ for some $i\in\{0,..., m-1\}$, where
\[ s=\begin{cases}
m, \; &d\ |\ m-1\\
m-1, \; &d \nmid m-1
\end{cases}
\]
\end{proposition}
\begin{proof}
We fix $\epsilon>0$, and we let $\delta>0$, $p_0\in\NN_+$ and a growth function $\mathcal{F}:\RR_+\to\RR_+$ be chosen later. Suppose that $\min\limits_{0\leqslant i\leqslant m-1}\|f_i\|_{U^s}\leqslant\delta$. By Lemma \ref{regularity lemma}, 
there exist $M=O_{\epsilon,\mathcal{F}}(1)$, a filtered manifold $G/\Gamma$ of degree $s_0=\left\lceil\frac{m}{d}\right\rceil - 1$ and complexity at most $M$, and a $p$-periodic sequence $g\in\poly(\ZZ,G_\bullet)$ with $g(0)=1$, for which there exists a decomposition
\begin{align*}
    f_m = f_{nil} + f_{sml} + f_{unf}
\end{align*}
such that $f_{nil}(n)=F(g(n)\Gamma)$ for an $M$-Lipschitz function $F: G/\Gamma\to\CC$, ${\|f_{sml}\|_2\leqslant \epsilon}$ and $\|f_{unf}\|_{U^{s_0+1}}\leqslant \frac{1}{\mathcal{F}(M)}$. Using the bound on $f_{sml}$, we crudely evaluate its contribution by
\begin{align}\label{contribution from f_sml}
    |\EE_{x,y\in\FF_p}f_0(x)f_1(x+y)\cdots f_{m-1}(x+(m-1)y)f_{sml}(x+y^d)|\leqslant \epsilon.
\end{align}

To bound the contribution of $f_{unf}$, we choose $\delta'>0$ and $p_0$ that work for $\epsilon$ as in Proposition \ref{U^2 control of x, x+y, x+2y, x+y^2}. We then pick $\mathcal{F}$ to be growing sufficiently fast so that $\|f_{unf}\|_{U^{s_0+1}}\leqslant \delta'$. Assuming that $p>p_0$ and applying Proposition \ref{U^2 control of x, x+y, x+2y, x+y^2}, we have
\begin{align}\label{contribution from f_unf}
    |\EE_{x,y\in\FF_p}f_0(x)f_1(x+y)\cdots f_{m-1}(x+(m-1)y)f_{unf}(x+y^d)|\ll \epsilon.
\end{align}

Finally, we observe that $f_{nil}(x+y^d)$ is a $p$-periodic nilsequence of complexity $M$ and degree $d\left\lfloor\frac{m-1}{d}\right\rfloor\leqslant s-1$. Using Lemma \ref{twisted von Neumann}, we choose $\delta>0$ in such a way as to guarantee that
\begin{align}\label{contribution from f_nil}
    |\EE_{x,y\in\FF_p}f_0(x)f_1(x+y)\cdots f_{m-1}(x+(m-1)y)f_{nil}(x+y^d)|\leqslant \epsilon.
\end{align}
The Proposition follows from combining (\ref{contribution from f_sml}), (\ref{contribution from f_unf}) and (\ref{contribution from f_nil}).
\end{proof}

In the case of $x,\; x+y,\; x+2y,\; x+y^2$, Proposition \ref{control by higher degree Gowers norms} gives us control of the first three terms by the $U^3$ norm. It turns out, however, that for this specific example we can get control by the $u^3$ norm instead.
\begin{proposition}\label{u^3 control}
Given any $\epsilon>0$, there exists $\delta>0$ and $p_0\in\NN$ s.t. for all $p>p_0$, we have
\begin{align*}
    |\EE_{x,y\in\FF_p}f_0(x)f_1(x+y)f_2(x+2y)f_3(x+y^2)|\ll \epsilon
\end{align*}
uniformly for all 1-bounded functions $f_0, f_1, f_2, f_3:\FF_p\to\CC$ satisfying $\|f_i\|_{u^3}\leqslant\delta$ for some $i\in\{0, 1,2\}$.
\end{proposition}

Propositions \ref{U^2 control of x, x+y, x+2y, x+y^2} and \ref{u^3 control} together prove Theorem \ref{True complexity of x, x+y, x+2y, x+y^2}.
\begin{proof}
Fix $\epsilon>0$, and let $\epsilon'>0$ be chosen later. Given $\epsilon'>0$, we choose $\delta'>0$ and $p_0$ given by Proposition \ref{U^2 control of x, x+y, x+2y, x+y^2}. Suppose that for at least one of $f_0, f_1, f_2$, we have $\|f_i\|_{u^3}\leqslant \delta$ for $\delta>0$ to be chosen later. Without loss of generality, suppose this holds for $f_0$.

We apply the following decomposition based on the Hahn-Banach theorem, a variant of which was used in \cite{gowers_wolf_2010, gowers_wolf_2011a, gowers_wolf_2011b, gowers_wolf_2011c, gowers_2010, peluse_2019b, peluse_prendiville_2019, kuca_2019}.
\begin{lemma}[Hahn-Banach decomposition]\label{Hahn-Banach decomposition}
Let $f:\FF_p\to\CC$ and $\|\cdot\|$ be a norm on the space of $\CC$-valued functions from $\FF_p$. Suppose $\|f\|_{L^2}\leqslant 1$ and $\eta>0$. Then there exists a decomposition
\begin{align*}
    f = f_a + f_b + f_c
\end{align*}
with $\|f_a\|^*\leqslant \delta'^{-2}\epsilon'^{-\frac{1}{2}}$, $\|f_b\|_{1}\leqslant \epsilon'^{\frac{1}{4}}$, $\|f_c\|_{\infty}\leqslant \epsilon'^{-\frac{1}{2}}$, $\|f_c\|\leqslant \delta'\epsilon'^{\frac{1}{2}}$ provided $0<\delta', \epsilon'<\frac{1}{10}$.
\end{lemma}
We use Lemma \ref{Hahn-Banach decomposition} to split $f_3$ with respect to the $U^2$ norm. The contribution of the term $f_b$ to the counting operator is given by
\begin{align*}
    |\EE_{x,y\in\FF_p}f_0(x)f_1(x+y)f_{2}(x+2y)f_b(x+y^2)|\leqslant\|f_b\|_{L^1}\leqslant\epsilon'^{\frac{1}{4}}.
\end{align*}
Using Proposition \ref{U^2 control of x, x+y, x+2y, x+y^2}, the contribution of $f_c$ is
\begin{align*}
    &|\EE_{x,y\in\FF_p}f_0(x)f_1(x+y)f_2(x+2y)f_c(x+y^2)|\\
    &=\max(\|f_c\|_{\infty},1)\cdot\left|\EE_{x,y\in\FF_p}f_0(x)f_1(x+y)f_2(x+2y)\frac{f_c(x+y^2)}{\max(\|f_c\|_{\infty},1)}\right|\\
    &\ll \epsilon'^{-\frac{1}{2}}\epsilon'=\epsilon'^{\frac{1}{2}}.
\end{align*}
Finally, the contribution coming from $f_a$ can be evaluated using $U^2$ inverse theorem as
\begin{align*}
    &|\EE_{x,y\in\FF_p}f_0(x)f_1(x+y)f_2(x+2y)f_a(x+y^2)|\\
    &\leqslant\|f_a\|_{U^2}^*\max_{\alpha\in\FF_p}|\EE_{x,y\in\FF_p}f_0(x)f_1(x+y)f_2(x+2y)e_p(\alpha(x+y^2))|^{\frac{1}{2}}.
\end{align*}
Since there exist quadratic polynomials $Q_0, Q_1, Q_2$ satisfying 
\begin{align*}
    x+y^{2} = Q_0(x)+Q_1(x+y) +Q_2(x+2y),
\end{align*}
we can bound
\begin{align*}
    &|\EE_{x,y\in\FF_p}f_0(x)f_1(x+y)f_2(x+2y)e_p(\alpha(x+y^2))|\\
    &=|\EE_{x,y\in\FF_p}f_0(x)e_p(\alpha Q_0(x))f_1(x+y)e_p(\alpha Q_1(x+y))f_2(x+2y)e_p(\alpha Q_{2}(x+2y))|\\
    &\leqslant \|f_0 e_p(\alpha Q_0(\cdot))\|_{u^2} \leqslant\|f_0\|_{u^3}.
\end{align*}

Bringing all the bounds together, we have
\begin{align*}
    |\EE_{x,y\in\FF_p}f_0(x)f_1(x+y)f_2(x+2y)f_a(x+y^2)|\leqslant \delta'^{-2}\epsilon'^{-\frac{1}{2}}\|f_0\|_{U^s} + \epsilon'^{\frac{1}{4}} + O(\epsilon'^{\frac{1}{2}}).
\end{align*}
Upon setting $\epsilon'=\epsilon^4$ and $\delta =  \epsilon^3\delta'^2$, and using $\|f_0\|_{U^s}\leqslant\delta$, we obtain
\begin{align*}
    |\EE_{x,y\in\FF_p}f_0(x)f_1(x+y)\cdots f_{m-1}(x+(m-1)y)f_a(x+y^2)|\ll\epsilon,
\end{align*}
as required.

\end{proof}
\appendix
\section{Proof of Lemma \ref{regularity lemma}}\label{section on ARL}
In this section, we give the proof of Lemma \ref{regularity lemma}, which is the simultaneous and periodic version of arithmetic regularity lemma.
\begin{proof}[Proof of Lemma \ref{regularity lemma}]
Fix $\epsilon>0$ and a growth function $\F:\RR_+\to\RR_+$. We pick another growth function $\F_0$ that grows sufficiently slowly with respect to $\F$. By Theorem 3.4 of \cite{candela_sisask_2012}, there exists $0<M_0=O_{ s,\epsilon, \F_0}(1)$ such that for each $i$ there is a filtered nilmanifold $G_i/\Gamma_i$ of complexity $M_0$ and degree $s$, a $p$-periodic sequence $g_i\in\poly(\ZZ,(G_i)_\bullet)$, and an $M_0$-Lipschitz function $F'_i:G_i/\Gamma_i\to\CC$ for which $f_i$ decomposes into 
\begin{align*}
    f_i = f_{i, nil} + f_{i,sml} + f_{i,unf}
\end{align*}
where the properties (ii), (iii), (iv) in Lemma \ref{regularity lemma} hold with $M_0$ in place of $M$ and $\F_0$ in place of $\F$, and moreover $f_{i,nil}(n)=F'_i(g_i(n)\Gamma_i)$. By redefining $F'_i$ and increasing its Lipschitz norm by a factor $O_{M_0}(1)$ if necessary, we can also assume that $g_i'(0)=1$ for all $1\leqslant i\leqslant t$.

We let
\begin{align*}
    G=G_1\times ...\times G_t, \quad \Gamma = \Gamma_1 \times ... \times \Gamma_t, \quad {\rm{and}} \quad g(n)=(g_1(n), ..., g_t(n)),  
\end{align*}
and we define $F_i(x_1\Gamma_1, ..., x_t\Gamma_t) := F'_i(x_i\Gamma_i)$. With this definition, we can realize each $f_{i,nil}$ as a $p$-periodic nilsequence $f_{i,nil}(n)=F_i(g(n)\Gamma)$ of degree $s$ and complexity $M_0 t$ on the same nilmanifold $G/\Gamma$ using the same $p$-periodic sequence $g$ for all $1\leqslant i\leqslant t$.

The next step is to obtain irrationality on the nilsequences $f_{1,nil}$, ..., $f_{t, nil}$. In doing so, we apply the proof of Theorem 5.1 of \cite{candela_sisask_2012}, which we rerun here for completeness. Given a growth function $\F_1$ to be chosen later, we use Proposition 5.2 of \cite{candela_sisask_2012} to obtain $M_1\in [M_0, O_{M_0, t, \F_1}(1)]$ and a $p$-periodic polynomial $g'\in\poly(\ZZ,G'_\bullet)$ on some nilmanifold $G'/\Gamma'$ of complexity $O_{M_1}(1)$ satisfying $g'(n)\Gamma = g(n)\Gamma$. By abuse of notation, we let $F_i$ denote now its restriction to $G'/\Gamma'$ for each $1\leqslant i\leqslant t$. It is $O_{M_1}(1)$-Lipschitz on $G'/\Gamma'$. Therefore the nilsequence $f_{i, nil}$ has complexity $M\leqslant \F_2(M_1)$ for some function $\F_2$. Letting $\F_1(x)=\F(\F_2(x))$ thus guarantees that $g'$ is $\F(M)$-irrational. To guarantee $\|f_{i,nil}\|_{U^s}\leqslant\frac{1}{\F(M)}$, we pick $\F_0$ so that $\F_0(M_0)\geqslant \F(M)$ using $M = O_{M_1}(1)=O_{M_0,t,\F}(1)$. Combining all the bounds, we have $M=O_{s,t,\epsilon, \F}(1)$, as desired.

In their statement of Theorem 3.4 of \cite{candela_sisask_2012}, the authors only considered functions from $\FF_p$ to $[0,1]$. However, the statement works for arbitrary 1-bounded functions from $\FF_p$ to $\CC$ by splitting them into the real and imaginary part, and the positive and negative part. This way, we split a 1-bounded function from $\FF_p$ to $\CC$ into four 1-bounded functions from $\FF_p$ to $[0,1]$, implying the 4-boundedness of $f_{i,nil}$, $f_{i,sml}$ and $f_{i,unf}$.

\end{proof}

\section{Baker-Campbell-Hausdorff formula}
This section describes some useful consequences of Baker-Campbell-Hausdorff formula and contains the same material as Appendix C in \cite{green_tao_2010a}, which we restate here for completeness.

Let $G$ be a $s$-step nilpotent, connected, simply-connected Lie group with Lie algebra $\mathfrak{g}$ and the exponential map $\exp:\mathfrak{g}\to G$. For any $X_1, X_2\in\mathfrak{g}$, we have
\begin{align*}
    \exp(X_1)\exp(X_2) = \exp(X_1 + X_2 + \frac{1}{2}[X_1, X_2] + \sum\limits_\alpha c_\alpha X_\alpha),
\end{align*}
where $c_\alpha=\frac{c_{1,\alpha}}{c_{2,\alpha}}\in\QQ$ for integers $c_{1,\alpha}, c_{2,\alpha}\ll_s 1$, and $X_\alpha$ is a Lie bracket of $k_1=k_{1,\alpha}$ copies of $X_1$ and $k_2=k_{2,\alpha}$ copies of $X_2$ for some $k_1, k_2\geqslant 1$ and $k_1+k_2\geqslant 3$.

In particular, for any $g_1, g_2\in G$ and $x\in\RR$, we have
\begin{align}\label{Baker-Campbell-Hausdorff 1}
    (g_1 g_2)^x = g_1^x g_2^x \prod_\alpha g_\alpha^{Q_\alpha(x)},
\end{align}
where $g_\alpha$ is an iterated commutator of $k_1=k_{1,\alpha}$ copies of $g_1$ and $k_2 = k_{2,\alpha}$ copies of $g_2$ for $k_1, k_2\geqslant 1$, and $Q_\alpha:\RR\to\RR$ is a polynomial of degree at most $k_1+k_2$ satisfying $Q_\alpha(0)=0$.

Moreover, for any $g_1, g_2\in G$ and $x_1, x_2\in\RR$, we have
\begin{align}\label{Baker-Campbell-Hausdorff 2}
    [g_1^{x_1}, g_2^{x_2}] = [g_1, g_2]^{x_1 x_2}\prod_\alpha g_\alpha^{Q_\alpha(x,y)},
\end{align}
where $g_\alpha$ is an iterated commutator of $k_1=k_{1,\alpha}$ copies of $g_1$ and $k_2 = k_{2,\alpha}$ copies of $g_2$ for $k_1, k_2\geqslant 1$, $k_1+k_2\geqslant 3$ whereas $Q_\alpha:\RR\times\RR\to\RR$ is a polynomial of degree at most $k_1$ in $x_1$ and $k_2$ in $x_2$, which moreover satisfies $Q_\alpha(x_1,0)=Q_\alpha(0,x_2)=0$.

\section{Scaling a polynomial sequence}
The following lemma is a stronger version of Lemma 5.3 in \cite{candela_sisask_2012}, which itself specializes Lemma A.8 of \cite{green_tao_2010a}.
\begin{lemma}\label{scaling a polynomial sequence}
Let $G/\Gamma$ be a filtered nilmanifold of degree $s$ and $g\in\poly(\ZZ,G_\bullet)$ be given by $g(n)=\prod\limits_{i=0}^s g_i^{{n}\choose{i}}$. Define $h(n)=g(pn)=\prod\limits_{i=0}^s h_i^{{n}\choose{i}}$. Then $$h_i = g_i^{p^i}\mod G_{i+1}$$ for any $i\in\NN_+$.
\end{lemma}
\begin{proof}
We fix $i\in\NN_+$. Since we only care about the value of $h_i$ mod $G_{i+1}$, we can quotient out $G$ by $G_{i+1}$, so that 
\begin{align*}
    g(n) = \prod_{k=0}^{i}g_k^{{n}\choose{k}}\mod G_{i+1}.
\end{align*}
By observing that ${{pn}\choose{k}}=p^k{{n}\choose{k}}+\sum\limits_{l=1}^{k-1}a_{k,l}{{n}\choose{l}}$ for some ${a_{k,l}\in\ZZ}$, we rewrite
\begin{align*}
    h(n) = \prod_{k=0}^{i} g_k^{p^k{{n}\choose{k}}+\sum\limits_{l=1}^{k-1}a_{k,l}{{n}\choose{l}}} \mod G_{i+1}.
\end{align*}
After rearranging the terms of $g(n)$ using (\ref{Baker-Campbell-Hausdorff 1}) and (\ref{Baker-Campbell-Hausdorff 2}), we obtain
\begin{align*}
    h(n) = \left(\prod_{k=0}^{i-1}h_k^{{n}\choose{k}}\right) \left(g_i^{p^i}\prod_\alpha g_\alpha\right)^ {{{n}\choose{i}}} \mod G_{i+1}.
\end{align*}
for some $h_k \in G_k$ and some commutators $g_\alpha$.

The important observation here is that each commutator $g_\alpha$ is obtained iteratively from (\ref{Baker-Campbell-Hausdorff 1}) or (\ref{Baker-Campbell-Hausdorff 2}) applied to elements $\Tilde{g}_{i_1}^{{n}\choose{l_1}}$, $\Tilde{g}_{i_2}^{{n}\choose{l_2}}$ belonging to $G_{i_1}$ and $G_{i_2}$ respectively, where $1\leqslant l_1\leqslant i_1$ and $1\leqslant l_2 \leqslant i_2$. However, we do not have $l_1 = i_1$ and $l_2 = i_2$ simultaneously because we never commute the elements $g_{i_1}^{p^{i_1}{{n}\choose{i_1}}}$ and $g_{i_2}^{p^{i_2}{{n}\choose{i_2}}}$ with each other. Without loss of generality, assume then that $l_2<i_2$.
%, and let $g_\alpha^{Q_\alpha(n)}$ be a term obtained via an application of (\ref{Baker-Campbell-Hausdorff 1}) or (\ref{Baker-Campbell-Hausdorff 2}). 
If $g_\alpha$ therefore is a $(k_1+k_2)$-fold commutator consisting of $k_1$ copies of $\Tilde{g}_{i_1}$ and $k_2$ copies of $\Tilde{g}_{i_2}$, then $g_\alpha\in G_{i_1 k_1+i_2 k_2}$ while $Q_\alpha$ is a polynomial of degree at most $$k_1 l_1 + k_2 l_2\leqslant k_1 i_1 + k_2 (i_2-1)\leqslant k_1 i_1 + k_2 i_2 -1.$$
If $g_\alpha^{Q_\alpha(n)}$ thus contributes to the Taylor coefficient of ${{n}\choose{i}}$ in $h$, then $k_1 l_1 + k_2 l_2\geqslant i$, implying that $k_1 i_1 + k_2 i_2 \geqslant i+1$. Hence $g_\alpha\in G_{i+1}$, as claimed.

\end{proof}

%\quad {\rm{and}} \quad h(n) = \prod_{k=0}^{i}h_k^{{n}\choose{k}}\mod G_{i+1}.
    %h(n) = \left(\prod_{k=0}^{i-1}g_k^{{n}\choose{k}}\right) g_i^{p^i{{n}\choose{i}}+\sum\limits_{l=1}^{i-1}a_{i,l}{{n}\choose{l}}} \mod G_{i+1}.

\bibliography{library}

\begin{thebibliography}{GW11b}

\bibitem[BC17]{bourgain_chang_2017}
J.~Bourgain and M.-C. Chang.
\newblock {Nonlinear Roth type theorems in finite fields}.
\newblock {\em Israel J. Math.}, \textbf{221}:853--867, 2017.

\bibitem[BL96]{bergelson_leibman_1996}
V.~Bergelson and A.~Leibman.
\newblock {Polynomial extensions of van der Waerden's and Szemer\'{e}di's
  theorems}.
\newblock {\em J. Amer. Math. Soc.}, \textbf{9}:725--753, 1996.

\bibitem[BLL07]{bergelson_leibman_lesigne_2007}
V.~Bergelson, A.~Leibman, and E.~Lesigne.
\newblock {Complexities of finite families of polynomials, Weyl systems, and
  constructions in combinatorial number theory}.
\newblock {\em {Journal d'analyse math{\'e}matique}}, \textbf{103}:47--92,
  2007.

\bibitem[CS12]{candela_sisask_2012}
P.~Candela and O.~Sisask.
\newblock Convergence results for systems of linear forms on cyclic groups and
  periodic nilsequences.
\newblock {\em SIAM J. Discrete Math.}, \textbf{28}:786--810, 2012.

\bibitem[DLS17]{dong_li_sawin_2017}
D.~Dong, X.~Li, and W.~Sawin.
\newblock {Improved estimates for polynomial Roth type theorems in finite
  fields}.
\newblock 2017.

\bibitem[FK05]{frantzikinakis_kra_2005}
N.~Frantzikinakis and B.~Kra.
\newblock Polynomial averages converge to the product of integrals.
\newblock {\em Israel J. Math.}, \textbf{148}(1):267--276, 2005.

\bibitem[FK06]{frantzikinakis_kra_2006}
N.~Frantzikinakis and B.~Kra.
\newblock Ergodic averages for independent polynomials and applications.
\newblock {\em J. Lond. Math. Soc.}, \textbf{74}:131 -- 142, 2006.

\bibitem[Fra08]{frantzikinakis_2008}
N.~Frantzikinakis.
\newblock Multiple ergodic averages for three polynomials and applications.
\newblock {\em Trans. Amer. Math. Soc.}, \textbf{360}(10):5435--5475, 2008.

\bibitem[Gow01]{gowers_2001}
W.~T. Gowers.
\newblock {A new proof of Szemer{\'e}di's theorem}.
\newblock {\em Geom. Funct. Anal.}, \textbf{11}(3):465--588, 2001.

\bibitem[Gow10]{gowers_2010}
W.~T. Gowers.
\newblock {Decompositions, approximate structure, transference, and the
  Hahn-Banach theorem}.
\newblock {\em Bull. Lond. Math. Soc}, \textbf{42}(4):573--606, 2010.

\bibitem[Gre07]{green_2007}
B.~Green.
\newblock {Montreal lecture notes on quadratic Fourier analysis}.
\newblock 2007.

\bibitem[GT08]{green_tao_2008}
B.~Green and T.~Tao.
\newblock {An inverse theorem for the Gowers $U^3(G)$ norm}.
\newblock {\em Proc. Edinb. Math. Soc.}, \textbf{51}(1):73–153, 2008.

\bibitem[GT10a]{green_tao_2010a}
B.~Green and T.~Tao.
\newblock An arithmetic regularity lemma, an associated counting lemma, and
  applications.
\newblock {\em Bolyai Soc. Math. Stud.}, \textbf{21}:261--334, 2010.

\bibitem[GT10b]{green_tao_2010b}
B.~Green and T.~Tao.
\newblock Linear equations in primes.
\newblock {\em Ann. of Math.}, \textbf{171}:1753--1850, 2010.

\bibitem[GT12]{green_tao_2012}
B.~Green and T.~Tao.
\newblock {The quantitative behaviour of polynomial orbits on nilmanifolds}.
\newblock {\em Ann. of Math.}, \textbf{175}:465--540, 2012.

\bibitem[GTZ11]{green_tao_ziegler_2011}
B.~Green, T.~Tao, and T.~Ziegler.
\newblock {An inverse theorem for the Gowers $U^4$ norm}.
\newblock {\em Glasg. Math. J.}, \textbf{53}(1):1–50, 2011.

\bibitem[GTZ12]{green_tao_ziegler_2012}
B.~Green, T.~Tao, and T.~Ziegler.
\newblock {An inverse theorem for the Gowers $U^{s+1}[N]$-norm}.
\newblock {\em Ann. of Math.}, \textbf{176}(2):1231--1372, 2012.

\bibitem[GW10]{gowers_wolf_2010}
W.~T. Gowers and J.~Wolf.
\newblock The true complexity of a system of linear equations.
\newblock {\em Proc. Lond. Math. Soc.}, \textbf{100}(1):155--176, 2010.

\bibitem[GW11a]{gowers_wolf_2011a}
W.~T. Gowers and J.~Wolf.
\newblock {Linear forms and higher-degree uniformity for functions on
  ${\mathbb{F}^{n}_{p}}$}.
\newblock {\em Geom. Funct. Anal.}, \textbf{21}:36--69, 2011.

\bibitem[GW11b]{gowers_wolf_2011b}
W.~T. Gowers and J.~Wolf.
\newblock {Linear forms and quadratic uniformity for functions on
  $\mathbb{F}_p^n$}.
\newblock {\em Mathematika}, \textbf{57}:215--237, 2011.

\bibitem[GW11c]{gowers_wolf_2011c}
W.~T. Gowers and J.~Wolf.
\newblock {Linear forms and quadratic uniformity for functions on
  $\mathbb{Z}_N$}.
\newblock {\em J. Anal. Math.}, \textbf{115}(1):121--186, 2011.

\bibitem[HK05a]{host_kra_2005b}
B.~Host and B.~Kra.
\newblock Convergence of polynomial ergodic averages.
\newblock {\em Israel J. Math.}, \textbf{149}(1):1--19, 2005.

\bibitem[HK05b]{host_kra_2005a}
B.~Host and B.~Kra.
\newblock Nonconventional ergodic averages and nilmanifolds.
\newblock {\em Ann. of Math.}, \textbf{161}(1):397--488, 2005.

\bibitem[HK18]{host_kra_2018}
B.~Host and B.~Kra.
\newblock {\em Nilpotent structures in ergodic theory}.
\newblock AMS, 2018.

\bibitem[Kuc19]{kuca_2019}
B.~Kuca.
\newblock {Further quantitative bounds in the polynomial Szemer\'{e}di theorem
  over finite fields}.
\newblock 2019.

\bibitem[Lei09]{leibman_2007}
A.~Leibman.
\newblock Orbit of the diagonal in the power of a nilmanifold.
\newblock {\em Trans. Amer. Math. Soc.}, \textbf{362}(03):1619--1658, 2009.

\bibitem[Man18]{manners_2018}
F.~Manners.
\newblock Good bounds in certain systems of true complexity 1.
\newblock {\em Discrete Anal.}, \textbf{21}:40 pp., 2018.

\bibitem[Pel18]{peluse_2018}
S.~Peluse.
\newblock Three-term polynomial progressions in subsets of finite fields.
\newblock {\em Israel J. Math.}, \textbf{228}:379--405, 2018.

\bibitem[Pel19a]{peluse_2019a}
S.~Peluse.
\newblock {Bounds for sets with no polynomial progressions}.
\newblock 2019.

\bibitem[Pel19b]{peluse_2019b}
S.~Peluse.
\newblock {On the polynomial Szemer\'{e}di theorem in finite fields}.
\newblock {\em Duke Math. J.}, \textbf{168}(5):749--774, 2019.

\bibitem[PP19]{peluse_prendiville_2019}
S.~Peluse and S.~Prendiville.
\newblock {Quantitative bounds in the non-linear Roth theorem}.
\newblock 2019.

\bibitem[PP20]{peluse_prendiville_2020}
S.~Peluse and S.~Prendiville.
\newblock {A polylogarithmic bound in the nonlinear Roth theorem}.
\newblock 2020.

\bibitem[Tao20]{tao_2020}
T.~Tao.
\newblock {A correction to “An arithmetic regularity lemma, an associated
  counting lemma, and applications”}, 2020.

\bibitem[TV06]{tao_vu_2006}
T.~Tao and V.~Vu.
\newblock {\em Additive Combinatorics}.
\newblock Cambridge Studies in Advanced Mathematics. Cambridge U. P., 2006.

\end{thebibliography}
\bibliographystyle{alpha}

\end{document}